\documentclass[11pt, a4paper]{amsart}

\usepackage{graphicx}
\usepackage[utf8]{inputenc}
\usepackage{amsmath}
\usepackage{amssymb}
\usepackage{hyperref}
\usepackage{amsthm}
\usepackage{enumerate}
\usepackage{wasysym}

\usepackage{extarrows}
\usepackage{polynom}
\usepackage{dsfont}
\usepackage{verbatim}
\usepackage[shortlabels]{enumitem}
\usepackage[toc,page]{appendix}
\usepackage{a4wide}
\usepackage{todonotes}
\usepackage{pst-node}

\theoremstyle{definition}
\newtheorem{theorem}{Theorem}[section]
\newtheorem{prop}[theorem]{Proposition}
\newtheorem{definition}[theorem]{Definition}

\newtheorem{remark}[theorem]{Remark}
\newtheorem{lemma}[theorem]{Lemma}
\newtheorem{coro}[theorem]{Corollary}
\numberwithin{equation}{section}

\newcommand{\abs}[1]{\left\lvert#1\right\rvert}
\newcommand{\norm}[1]{\left\|#1\right\|}

\newcommand{\R}{\mathbb R}
\newcommand{\N}{\mathbb N}
\newcommand{\Z}{\mathbb Z}
\newcommand{\BE}{\mathbb{E}}
\newcommand{\BB}{\mathbb{B}}

\newcommand{\pv}{\mathrm{p.v.}}

\renewcommand{\epsilon}{\varepsilon}
\renewcommand{\S}{\mathcal S}
\newcommand{\dx}{\mathrm d}
\newcommand{\id}{\mathrm{Id}}
\newcommand{\tr}{\mathrm{Tr}}

\newcommand{\T}{\mathcal T}

\newcommand{\F}{\mathcal F}
\newcommand{\B}{\mathcal B}

\renewcommand{\L}{\mathcal L}
\renewcommand{\phi}{\varphi}
\DeclareMathOperator{\supp}{supp}

\begin{document}
\allowdisplaybreaks

\title[Kinetic maximal $L^p_\mu$-regularity]{Kinetic maximal $L^p$-regularity with temporal weights and application to quasilinear kinetic diffusion equations}

\date{\today}
\author{Lukas Niebel}
\email{lukas.niebel@uni-ulm.de}
\author{Rico Zacher$^*$}
\thanks{$^*$Corresponding author. The first author is supported by a graduate
scholarship (''Landesgraduiertenstipendium'') granted by the State of Baden-Wuerttemberg, Germany}
\email[Corresponding author]{rico.zacher@uni-ulm.de}
\address[Lukas Niebel, Rico Zacher]{Institut f\"ur Angewandte Analysis, Universit\"at Ulm, Helmholtzstra\ss{}e 18, 89081 Ulm, Germany.}

\maketitle

\begin{abstract}
	We introduce the concept of kinetic maximal $L^p$-regularity with temporal weights and prove that this property is satisfied for the (fractional) Kolmogorov equation. We show that solutions are continuous with values in the trace space and prove, in particular, that the trace space can be characterized in terms of anisotropic Besov spaces. We further extend the property of kinetic maximal $L^p_\mu$-regularity to the Kolmogorov equation with variable coefficients. Finally, we show how kinetic maximal $L^p_\mu$-regularity can be used to obtain local existence of solutions to a class of quasilinear kinetic equations and illustrate our result with a quasilinear kinetic diffusion equation.   
\end{abstract}
\vspace{1em}
{\centering \textbf{AMS subject classification.} 35K65, 35B65, 35K59 \par}
\vspace{1em}
\textbf{Keywords.} kinetic maximal $L^p$-regularity, Kolmogorov equation, anisotropic Besov spaces, instantaneous smoothing, quasilinear kinetic equations, short-time existence

\section{Introduction}
The main purpose of this article is to introduce and study the concept of kinetic maximal $L^p$-regularity with temporal weights for linear kinetic partial differential equations. The most prominent example for such an equation is the (fractional) Kolmogorov equation
\begin{equation} \label{eq:intkol}
	\begin{cases}
		\partial_t u + v \cdot \nabla_x u = -(-\Delta_v)^\frac{\beta}{2}u +f \\
		u(0) = u_0
	\end{cases}
\end{equation}
with $\beta \in (0,2]$. The Kolmogorov equation has been studied for almost a decade \cite{kolmogoroff_zufallige_1934} and has become the focus of attention of many mathematicians in the last 30 years \cite{armstrong_variational_2019,bramanti_invitation_2014,golse_harnack_2016,lanconelli_class_1994,lorenzi_analytical_2017,lunardi_schauder_2005}. This is mainly due to the fact that it serves as the prototype of a kinetic partial differential equation similar to the famous Boltzmann or Landau equation. As the name of the article suggests we are interested in solutions which are element of some Lebesgue space. The Kolmogorov equation in $L^p$-spaces can be studied by the methods in \cite{folland_estimates_1974} for $\beta = 2$ and is studied in \cite{chen_lp-maximal_2018} for $ \beta \in (0,2)$. Let $p \in (1,\infty)$, then, for all $f \in L^p((0,\infty);L^p(\R^{2n}))$ the solution $u$ of equation \eqref{eq:intkol} given in terms of the fundamental solution corresponding to $u_0 = 0$ satisfies the inequality
\begin{equation} \label{eq:intineq}
	\norm{\partial_t u + v \cdot \nabla_x u }_p + \norm{(-\Delta_v)^\frac{\beta}{2}u}_p  \lesssim \norm{f}_p.
\end{equation}
This inequality reminds of the property of maximal $L^p$-regularity. This observation is the motivation of our work.

Let us first recall the definition of maximal $L^p$-regularity with temporal weights \cite{pruss_maximal_2004,pruss_moving_2016}. Let $X$ be a Banach space and $A \colon D(A) \to X$ be the generator of a strongly continuous semigroup. For $T>0$ we introduce the Lebesgue spaces with temporal weights as $L^p_\mu((0,T);X)) = \{ u \colon (0,T) \to X \colon t^{1-\mu} \in L^p((0,T);X) \} $ equipped with the norm $\norm{u}_{p,\mu,X}^p := \int_0^T t^{p-\mu p} \norm{u(t)}_X^p \dx t$ for $p \in (1,\infty)$ and $\mu \in (1/p,1]$. We say that $A$, respectively the equation
\begin{equation} \label{eq:clasmaxreg}
	\begin{cases}
		\partial_t u = Au+f \\
		u(0) = u_0,
	\end{cases}
\end{equation}
admits maximal $L^p_\mu(X)$-regularity, if for all $f \in L^p_\mu((0,T);X)$ there exists a unique solution
\begin{equation} \label{eq:mrspaceclass}
	u \in \BE_\mu(0,T):= W^{1,p}_\mu((0,T);X) \cap L^p_\mu((0,T);D(A))
\end{equation}
	of the problem \eqref{eq:clasmaxreg} with $u_0 = 0$. This solution then satisfies the inequality
	\begin{equation} \label{eq:ineqmaxreg}
		\norm{\partial_t u}_{p,\mu,X}+\norm{Au}_{p,\mu,X} \lesssim \norm{f}_{p,\mu,X}.
	\end{equation}
The property of maximal $L^p$-regularity is a very useful and strong feature of an operator. 	

For example, under the assumption of maximal $L^p$-regularity, one can solve the initial value problem \eqref{eq:clasmaxreg} in $\BE_\mu(0,T)$ whenever $u_0$ is an element of the real interpolation space 
	$$X_{\gamma,\mu} := \tr (\BE_\mu(0,T)) = (X,D(A))_{\mu-1/p,p}, $$
and we have $u \in C([0,T];X_{\gamma,\mu})$. In particular, solutions of the homogeneous equation $\partial_t u = Au$, $u(0) = u_0$ define a semi-flow in the trace space $X_{\gamma,\mu}$. Maximal $L^p$-regularity thus gives a characterization of the solution $u$ of equation \eqref{eq:clasmaxreg} being an element of $\BE_\mu(0,T)$ in terms of Banach spaces for the data $f$ and $u_0$. 

However, it is a classical result, that under the assumption of maximal $L^p$-regularity the operator $A$ needs to be the generator of a holomorphic semigroup. It can be shown that this is not the case for the operator $-(-\Delta_v)^\frac{\beta}{2}- v \cdot \nabla_x$ in $L^q(\R^{2n})$, see \cite{niebel_kinetic_nodate}. Therefore, the concept of maximal $L^p$-regularity with base space $X = L^q(\R^{2n})$ is not applicable in the context of kinetic equations. 

 In this article, we introduce and investigate the concept of kinetic maximal $L^p$-regularity, which can be viewed as a suitable replacement of maximal $L^p$-regularity in the kinetic setting. Based on the inequality \eqref{eq:intineq} it seems to be natural to define the concept of kinetic maximal $L^p_\mu$-regularity as follows. We say that the fractional Laplacian $-(-\Delta_v)^\frac{\beta}{2}$ admits kinetic maximal $L^p_\mu(L^q)$-regularity if for all $f \in L^p_\mu((0,T);L^q(\R^{2n}))$ there exists a unique distributional solution $u \in L^p_\mu((0,T);L^q(\R^{2n}))$ of equation \eqref{eq:intkol} with $u(0) = 0$ such that 
$$\partial_t u + v \cdot \nabla_x u, (-\Delta_v)^\frac{\beta}{2} u \in L^p_\mu((0,T);L^q(\R^{2n})).$$
In this case we write $u \in \BE_\mu(0,T) := \T_\mu^p((0,T);L^q(\R^{2n})) \cap L^p_\mu((0,T);H_v^{\beta,q}(\R^{2n}))$, where $\T_\mu^p$ is used to describe the regularity of the kinetic term. See below for a precise definition of these spaces.

Obviously, the first question is whether the property of kinetic maximal $L^p_\mu(L^q)$-regularity is satisfied at all. In the special case $p = q = 2$, $\mu \in (1/2,1]$ this has been shown recently in \cite{niebel_kinetic_nodate} by the authors. For $p = q \in (1,\infty)$ and $\mu = 1$ the corresponding key estimate has been known since at least the findings in \cite{chen_lp-maximal_2018,folland_estimates_1974}. On the basis of this estimate we establish kinetic maximal $L^p(L^p)$-regularity and extend this result by proving kinetic maximal $L^p_\mu(L^q)$-regularity for $p,q \in (1,\infty)$ and $\mu \in (1/p,1]$. The proof of the results in \cite{folland_estimates_1974} relies on the structure of the homogeneous group on $\R^{2n+1}$ underlying the kinetic term $\partial_t + v \cdot \nabla_x$, which does not prioritize any of the variables $(t,x,v)$. Kinetic maximal $L^p_\mu(L^q)$-regularity explicitly splits the integrability with respect to the time variable $t$ and the space variables $x,v$. 

The most important tool in our article is the operator
\begin{equation*}
	[\Gamma u](t,x,v) = u(t,x+tv,v)
\end{equation*}
defined on functions $u \colon \R \times \R^n \times \R^n \to \R$. Applying the operator $\Gamma$ to a solution $u$ of the Kolmogorov equation leads to the non-autonomous problem
\begin{equation*}
\begin{cases}
	\partial_t w = -\Gamma (-\Delta_v)^\frac{\beta}{2} \Gamma^{-1} w + \Gamma f \\
	w(0) = u_0
	\end{cases}
\end{equation*} 
where $w = \Gamma u$. This observation is one of the key points in our proof of kinetic maximal $L^p(L^q)$-regularity for the (fractional) Kolmogorov equation. Similary as in the classical case \cite{pruss_maximal_2004}, we show that kinetic maximal $L^p_\mu(L^q)$-regularity is independent of the parameter $\mu \in (1/p,1]$, and thus follows from the unweighted case. 

It is natural to ask in which sense the initial value $u(0) = 0$ is attained. As the space $\BE_\mu(0,T)$ of kinetic maximal $L^p_\mu(L^q)$-regularity does not allow to control the regularity of the term $\partial_t u$, this cannot be answered by using standard techniques. However, using several operator theoretic properties of the operator $\Gamma$ including the crucial identity $\partial_t \Gamma u = \Gamma(\partial_t u + v \cdot \nabla_x u)$, which relates $\T_\mu^p$ with $W^{1,p}_\mu$, we show that
\begin{equation*}
	\T_\mu^p((0,T);L^q(\R^{2n})) \hookrightarrow C([0,T];L^q(\R^{2n})).
\end{equation*}
 This allows to define the trace space $X_{\gamma,\mu} = \tr( \BE_\mu(0,T))$ and to deduce that every function in $\BE_\mu(0,T)$ is continuous with values in the trace space $X_{\gamma,\mu}$ by an abstract argument. In the classical theory of maximal $L^p$-regularity the trace space is given by a real interpolation space of the base space $X$ and the domain $D(A)$. For example, in the case of the fractional heat equation on the space $L^q(\R^{2n})$ this leads to the Besov space $(L^q(\R^{2n}),H^{\beta,q}(\R^{2n}))_{\mu-1/p,p} = B^{\beta(\mu-1/p)}_{qp}(\R^{2n})$. This is no longer true in the kinetic setting in general. In this work we succeed in directly characterizing the trace space of $\BE_\mu(0,T)$ in terms of anisotropic Besov spaces. Motivated by the so-called kinetic regularization, formalized by Bouchut in \cite{bouchut_hypoelliptic_2002} for $L^p(L^p)$-spaces, it is natural to expect that the trace space needs to control some regularity in the $x$-variable, too. We generalize this kinetic regularization result to $L^p_\mu(L^q)$-spaces. It turns out that the trace space is isomorphic to the anisotropic kinetic Besov space ${^{\mathrm{kin}}B}_{qp}^{\mu-1/p,\beta}(\R^{2n})$, which can be viewed as the intersection of a Besov space of order $\frac{\beta}{\beta+1}(\mu-1/p)$ in the $x$-variable and a Besov space of order $\beta(\mu-1/p)$ in the $v$-variable. This allows us to fully characterize strong $L^p_\mu(L^q)$-solutions of the Cauchy problem for the (fractional) Kolmogorov equation in terms of the data $f$ and $u_0$.

The kinetic first order term $\partial_t + v \cdot \nabla_x$ limits the choice of suitable base spaces. So far we have restricted ourselves to the base space $L^q(\R^{2n})$. In the present article we do not only consider kinetic maximal $L^p_\mu$-regularity in the base space $L^q(\R^{2n})$ but even for the scale of spaces
\begin{equation*}
	X_\beta^{s,q} =  \left\{ f \in \S'(\R^{2n}) \; \colon \; \F^{-1} \left( \left(  (1+\abs{k}^2)^\frac{\beta}{2(\beta+1)}+(1+\abs{\xi}^{2})^\frac{\beta}{2} \right)^{s} \F(f)(k,\xi) \right) \in L^q(\R^{2n})   \right\}.
\end{equation*}
The definition of this space is again motivated by the kinetic regularization result. Observe that $X_\beta^{0,q} = L^q(\R^{2n})$. Most of the abstract properties like the embedding into some space of continuous functions in time hold when $L^q(\R^{2n})$ is replaced by $X^{s,q}_\beta$, too. We show that the (fractional) Kolmogorov equation admits kinetic maximal $L^p_\mu(X^{s,q}_\beta)$-regularity for all $s \ge 0$. This corresponds to strong solutions with possibly higher regularity. Together with the characterization of the trace space we are able to show that the $L^p_\mu(L^q)$-solution of the (fractional) Kolmogorov equation with initial datum in ${^{\mathrm{kin}}B}_{qp}^{\mu-1/p,\beta}(\R^{2n})$ regularizes instantaneously, i.e. is an element of $C^\infty((0,\infty) \times \R^{2n})$.

Concerning weak solutions, kinetic maximal $L^2_\mu$-regularity with base space $X^{-1/2,2}_\beta$ has been studied recently by the authors in \cite{niebel_kinetic_nodate}. We extend this result to kinetic maximal $L^p_\mu(X^{-1/2,2}_\beta)$-regularity for $p \in (1,\infty)$ and $\mu \in (1/p,1]$. However, kinetic maximal $L^p_\mu(X^{-1/2,q}_\beta)$-regularity for $q \neq 2$ remains an interesting open problem. 

An important application of maximal $L^p$-regularity is the treatment of quasilinear problems by linearization. For example one can prove local well-posedness. Indeed, one can show that under the assumptions
\begin{equation*}
	(A,F) \in C^{1-}_{\mathrm{loc}}(X_{\gamma,\mu} ; \B(D(A),X)\times X),
\end{equation*}
$u_0 \in X_{\gamma,\mu}$ and that $A(u_0)$ admits maximal $L^p_\mu$-regularity, the Cauchy problem
\begin{equation*}
	\begin{cases}
		\partial_t u = A(u) u + F(u) \\
		u(0) = u_0
	\end{cases}
\end{equation*}
admits a unique solution $u \in \BE_\mu(0,T)$ for some $T>0$ with $\BE_\mu(0,T)$ as in equation \eqref{eq:mrspaceclass}, see \cite{pruss_moving_2016}. We transfer this result to a wide class of quasilinear kinetic equations. This allows us to prove short time existence of $L^p_\mu(L^q)$-solutions to the kinetic quasilinear diffusion problem 
\begin{equation*}
	\begin{cases}
			\partial_t u + v \cdot \nabla_x u = \nabla_v  \cdot (a(u) \nabla_v u) \\ 
			u(0) = u_0,
	\end{cases}
\end{equation*}
where the nonlinearity $a \in C^2_b(\R; \mathrm{Sym}(n))$. In order to apply our general result we extend kinetic maximal $L^p_\mu(L^q)$-regularity of the Laplacian $\Delta_v$ to second order differential operators in non-divergence form $a \colon \nabla_v^2$ with variable coefficient $a$. For $a = a(t,x,v)$ our analysis shows that it is natural to impose that $a$ is bounded and that $\Gamma a $ is uniformly continuous. In particular, we do not only study the concept of maximal $L^p_\mu$-regularity for the fractional Laplacian in velocity but for a general class of possibly non-autonomous operators.

The plan of the present article is as follows. We introduce the concept of kinetic maximal $L^p_\mu$-regularity in Section \ref{sec:kinmaxreg}. Here, we also provide a deeper study of the corresponding function spaces and investigate some abstract consequences of the kinetic maximal $L^p_\mu$-regularity property. Moreover, we state how existing results regarding global $L^p$-estimates for the (fractional) Kolmogorov equation fit into our setting. In Section \ref{sec:regtransfer} we extend the principle of regularity transfer for kinetic equations to the $L^p_\mu(L^q)$-setting. Next, we extend the well-known $L^p(L^p)$-estimates for the (fractional) Kolmogorov equation to the case of different integrability exponents in time and space. In particular, we prove in Section \ref{sec:LpLq} that the Kolmogorov equation admits kinetic maximal $L^p(L^q)$-regularity. In Section \ref{sec:kinmaxlpmulq} we prove the equivalence of kinetic maximal $L^p$-regularity and kinetic maximal $L^p$-regularity with certain temporal weights. We apply this theorem to the (fractional) Kolmogorov equation. Section \ref{sec:kolcoeff} extends kinetic maximal $L^p_\mu(L^q)$-regularity of the Kolmogorov equation to related operators with time- and space dependent coefficients. An abstract result on how kinetic maximal $L^p_\mu$-regularity can be used to solve quasilinear kinetic equations is proven in Section \ref{sec:absquasi}. In the last Section \ref{sec:quasilindiff} we apply the latter abstract result to prove short-time existence for a quasilinear kinetic diffusion equation. In the appendix we have collected some technical results concerning Fourier multipliers and Besov spaces. Moreover, we prove $L^p_\mu(L^q)$-regularity estimates for an initial value problem, which will be used to characterize the initial value regularity of strong solutions of the (fractional) Kolmogorov equation.
 
 In our calculations we denote by the letter $c$ a universal positive constant which can change from line to line. Whenever we write $\lesssim$ we estimate only by a universal constant. Estimates are proven always for smooth functions first and the general case follows by an approximation argument. Furthermore, let us fix some notation. For $s \in \R$ and $q \in (1,\infty)$ we define one Bessel potential space to measure the regularity in the position variable $x$,
\begin{equation*}
	H_x^{s,q}(\R^{2n}) = \left\{ f \in \S' \; \colon \; \F^{-1}((1+\abs{k}^{2})^\frac{s}{2} \F(f)(k,\xi)) \in L^q(\R^{2n})   \right\},
\end{equation*}
and one for the regularity in the velocity variable $v$,
\begin{equation*}
	H_v^{s,q}(\R^{2n}) = \left\{ f \in \S' \; \colon \; \F^{-1}((1+\abs{\xi}^{2})^\frac{s}{2} \F(f)(k,\xi)) \in L^q(\R^{2n})   \right\}.
\end{equation*}
Both of these spaces are Banach spaces when equipped with the obvious norm. The operators $D_x^s := (-\Delta_x)^{s/2} \colon H_x^{s,q}(\R^{2n}) \to L^q(\R^{2n})$ and $D_v^s := (-\Delta_v)^{s/2} \colon H_v^{s,q}(\R^{2n}) \to L^q(\R^{2n})$ are well-defined and bounded. The vector space of symmetric matrices $a \in \R^{n \times n}$ will be denoted by $\mathrm{Sym}(n)$. For two matrixes $a,b \in \R^{n \times n}$ we write $a \colon b = \tr(AB)$.

\section{Kinetic Maximal $L^p_\mu$-regularity with temporal weights} \label{sec:kinmaxreg}

In this section we define the concept of kinetic maximal $L^p_\mu$-regularity. Let $\beta \in (0,2]$, $s \in \R$, $p,q \in (1, \infty)$, $\mu \in \left( \frac{1}{p},1 \right]$ and $T \in (0,\infty)$. We introduce the space
\begin{equation*}
	X_\beta^{s,q} =  \left\{ f \in \S'(\R^{2n}) \; \colon \; \F^{-1} \left( \left(  (1+\abs{k}^2)^\frac{\beta}{2(\beta+1)}+(1+\abs{\xi}^{2})^\frac{\beta}{2} \right)^{s} \F(f)(k,\xi) \right) \in L^q(\R^{2n})   \right\}
\end{equation*}
equipped with the respective norm $\norm{\cdot}_{X_\beta^{s,q}}$. In particular, $X_\beta^{0,q} = L^q(\R^{2n})$ and $X_\beta^{s,q} = H_x^{s\frac{\beta}{\beta+1},q}(\R^{2n}) \cap H_v^{s\beta,q}(\R^{2n})$ for all $s \ge 0$. 

 Given any Banach space $X$, $p \in (1,\infty)$, $T \in (0,\infty]$ and $\mu \in (1/p,1]$ we define the time-weighted $L_\mu^p(X)$ space as
\begin{equation*}
	L_\mu^p((0,T);X) = \{ f \colon (0,T) \to X \colon f \text{ measurable and } \int_0^T t^{p-p\mu} \norm{f(t)}_X^p \dx t < \infty \}.
\end{equation*}
Equipped with the norm $\norm{f}_{p,\mu,X}^p = \int_0^T t^{p-p\mu} \norm{f(t)}_X^p \dx t$ the vector space $L_\mu^p((0,T);X)$ is a Banach space. Here, we will mainly consider the Lebesgue space $L^p_\mu ((0,T);X_\beta^{s,q})$ and additionally define the vector space
\begin{equation*}
	\T^{p}_\mu((0,T);X_\beta^{s,q}) = \{ f \in L^p_\mu ((0,T);X_\beta^{s,q}) \;  \colon \; \partial_t f +v \cdot \nabla_x f \in L^p_\mu ((0,T);X_\beta^{s,q}) \}
\end{equation*}
equipped with the norm $\norm{u}_{\T^p_\mu(X_\beta^{s,q})} = \norm{u}_{p,\mu,X_\beta^{s,q}}+ \norm{\partial_t u +v\cdot \nabla_x u}_{p,\mu,X_\beta^{s,q}}$. If $\mu = 1$ we drop the subscript in our notation. 

We introduce the mapping $\Gamma$ acting on functions $u \colon \R \times \R^{2n} \to \R$ given by $[\Gamma u](t,x,v) = u(t,x+tv,v)$ for all $t \in \R$ and any $x,v \in \R^n$. For functions $u = u(x,v)$ we write $[\Gamma(t)u](x,v) = u(x+tv,v)$ for $t \in \R$ and all $x,v \in \R^n$. The following results are the natural generalization of the results given in \cite{niebel_kinetic_nodate} for the case $p = q = 2$.  

\begin{prop} \label{prop:GamLpisom}
	For all $p,q \in (1,\infty)$, any $\mu \in (1/p,1]$ and all $T \in (0,\infty]$ the mapping
	\begin{equation*}
		\Gamma \colon L^p_\mu((0,T); L^q(\R^{2n})) \to L^p_\mu((0,T); L^q(\R^{2n})) 
	\end{equation*}
	is a well-defined isometric isomorphism. The inverse operator is given by $[\Gamma^{-1}u](t,x,v) = u(t,x-tv,v) $. Moreover, the group $(\Gamma(t))_{t \in \R} \colon L^q(\R^{2n}) \to L^q(\R^{2n})$ of isometries is strongly continuous. 
\end{prop}

\begin{proof}
	The first property follows immediately by substitution in the $x$-integral. The strong continuity follows by an approximation argument and a proof can be found in \cite[Proposition 2.2]{metafune_lp-spectrum_2001}. 
\end{proof}

Given any Banach space $X$ we define
\begin{equation*}
	W_\mu^{1,p}((0,T);X) = \{ u \colon (0,T) \to X \; \colon \; u, \partial_t u \in L^p_\mu((0,T);X)) \}.
\end{equation*}

\begin{prop} \label{prop:TWtoWW}
	For all $p,q \in (1,\infty)$, $\mu \in (1/p,1]$, $T \in (0,\infty)$ and $k \in \Z$ the mapping
	\begin{equation*}
		\Gamma \colon \T_\mu^{p}((0,T);W^{k,q}(\R^{2n})) \to W^{1,p}_\mu((0,T); W^{k,q}(\R^{2n}))
	\end{equation*}
	is a well-defined isomorphism. Moreover, we have 
	\begin{equation*}
		\partial_t \Gamma u = \Gamma(\partial_t u + v \cdot \nabla_x u)
	\end{equation*}
	in the sense of distributions for all $u \in \T_\mu^{p}((0,T);W^{k,q}(\R^{2n}))$. Finally, the group 
	$$(\Gamma(t))_{t \in \R} \colon W^{k,q}(\R^{2n}) \to W^{k,q}(\R^{2n})$$ 
	is strongly continuous and the same holds true for the inverse group. 
\end{prop}

\begin{proof}
	The proof follows along the lines of the proof of \cite[Proposition 5.3]{niebel_kinetic_nodate}.
\end{proof}

\begin{lemma} \label{lem:WintoC}
	Let $p \in (1,\infty)$, $\mu \in (1/p,1]$, $T \in (0,\infty]$ and let $X$ be any Banach space. The trace at $t= 0$ of every function $u \in W^{1,p}_\mu((0,T); X)$ is well-defined. Moreover, the embedding
	\begin{equation*}
		W^{1,p}_\mu((0,T);X) \hookrightarrow C([0,T];X)
	\end{equation*} 
	holds continuously. 
\end{lemma}

\begin{proof}
	\cite[Proposition 2.1]{pruss_maximal_2004}
\end{proof}

We introduce the subspace $_0W_\mu^{1,p}((0,T);X) = \{ u \in W_\mu^{1,p}((0,T);X) \; | \; u(0) = 0 \}$, where $X$ denotes any Banach space. 	

\begin{prop} \label{prop:TcontinLq}
	For all $p,q \in (1,\infty)$, $\mu \in (1/p,1]$, $T \in (0,\infty)$ and any $k \in \Z$ the trace at $t= 0$ of every function $u \in \T^p_\mu((0,T); W^{k,q}(\R^{2n}))$ is well-defined. Moreover, we have
	\begin{equation*}
		\T_\mu^p((0,T);W^{k,q}(\R^{2n})) \hookrightarrow C([0,T];W^{k,q}(\R^{2n}))
	\end{equation*}
	continuously. In particular, if $k \ge 0$ the embedding
	\begin{equation*}
		\T^p_\mu((0,T);W^{k,q}(\R^{2n})) \hookrightarrow C([0,T];L^q(\R^{2n}))
	\end{equation*}
	holds, too. 
\end{prop}

\begin{proof}
	The proof follows analogously to the proof of \cite[Theorem 5.4]{niebel_kinetic_nodate} using Proposition \ref{prop:TWtoWW} and Lemma \ref{lem:WintoC}. 
\end{proof}

\begin{theorem} \label{thm:Xbsqemb}
	Let $\beta \in (0,2]$, $s \in \R$, $p,q \in (1,\infty)$, $\mu \in (1/p,1]$ and $T \in (0,\infty)$. Then, there exists $k \in \Z $ such that the embedding
	\begin{equation*}
		\T_{\mu}^p((0,T);{X}_\beta^{s,q}) \hookrightarrow C([0,T];W^{k,q}(\R^{2n}))
	\end{equation*}
	holds continuously. Moreover, if $s \ge 0$, then 
	\begin{equation*}
		\T_{\mu}^p((0,T);{X}_\beta^{s,q}) \hookrightarrow C([0,T];L^q(\R^{2n})).
	\end{equation*}
\end{theorem}

\begin{proof}
	For $s \ge -1/2$ we have $X_\beta^{s,q} \subset W^{-1,q}(\R^{2n})$, which shows the first embedding. The general case follows as for all $k \in \N_0$ and any $s \in [-k/2,\infty)$ we have $X_\beta^{s,q} \subset W^{-k,q}(\R^{2n})$. The second embedding follows from the fact that $X_\beta^{s,q} \subset L^q(\R^{2n})$ for all $s \ge 0$.
\end{proof}

We write $_0\T^p_\mu((0,T);Y) = \{ u \in \T^p_\mu((0,T);Y) \; | \; u(0) = 0 \}$ for $Y \in \{{X}_\beta^{s,q},W^{k,q}(\R^{2n}) \}$. For any Banach space $X$ we introduce the mapping $\Phi_\mu \colon L^p_\mu((0,T);X) \to L^p((0,T);X) $ defined as $[\Phi_\mu u](t) = t^{1-\mu}u(t)$ for $t>0$. This operator is very useful to relate weighted and unweighted $L^p$-spaces. We collect some properties of $\Phi_\mu$.

\begin{prop} Let $X$ be a Banach space. \label{prop:Phimuprop}
	\begin{enumerate}[(i)]
		\item $\Phi_\mu \colon  L^p_\mu((0,T);X) \to L^p((0,T);X) $  is an isometric isomorphism.
		\item $\Phi_\mu \colon _0W_\mu^{1,p}((0,T);X) \to  {_0 W}^{1,p}((0,T);X) $  is an isomorphism.
	\end{enumerate}
\end{prop} 

\begin{proof}
	See \cite[Proposition 2.2]{pruss_maximal_2004}.
\end{proof}

\begin{coro} \label{cor:TmuT}
For all $p,q \in (1,\infty)$, $\mu \in (1/p,1]$, $T \in (0,\infty)$ and all $k \in \Z$ the following diagram commutes.
	\[ \psset{arrows=->, arrowinset=0.25, linewidth=0.6pt, nodesep=3pt, labelsep=2pt, rowsep=0.7cm, colsep = 1.1cm, shortput =tablr}
 \everypsbox{\scriptstyle}
 \begin{psmatrix}
 _0\T_\mu^{p}((0,T);W^{k,q}(\R^{2n})) & _0\T^{p}((0,T);W^{k,q}(\R^{2n})) \\
 _0W_\mu^{1,p}((0,T);W^{k,q}(\R^{2n})) & _0W^{1,p}((0,T);W^{k,q}(\R^{2n})).
 \ncline{1,1}{1,2}^{\Phi_\mu} 
 \ncline{1,1}{2,1} <{\Gamma }
 \ncline{1,2}{2,2} > {\Gamma}
 \ncline{2,1}{2,2}^{\Phi_\mu}
 \end{psmatrix}
 \]
 In particular, each of the operators in the diagram is a well-defined isomorphism.  
\end{coro}

\begin{coro} \label{coro:hardykin}
	Let all $p,q \in (1,\infty)$, $\mu \in (1/p,1]$, $T \in (0,\infty)$ and $k \in \Z$. For all $u \in {_0\T}_\mu^{p}((0,T);W^{k,q}(\R^{2n}))$ we have $t^{-\mu}u(t) \in L^p((0,T);W^{k,q}(\R^{2n}))$.
\end{coro}

\begin{proof}
	A proof for $u \in {_0W}^{1,p}((0,T);W^{k,q}(\R^{2n}))$, i.e. $\mu = 1$, can be found in the proof of \cite[Proposition 2.2]{pruss_maximal_2004}. Using the mapping $\Phi_\mu$ one can transfer this to the weighted case and then using the mapping $\Gamma$ to the case that $u \in {_0\T}_\mu^{p}((0,T);W^{k,q}(\R^{2n}))$.
\end{proof}

We are now able to define what kinetic maximal $L^p_\mu(X_\beta^{s,q})$-regularity means for a family of operators $A(t) = (A(t))_{t \in [0,T]}$ with constant domain $D(A) \subset X_\beta^{s,q}$. We always assume that $D(A)$ is equipped with a norm equivalent to the graph norm of $A(0)$. We introduce 
$${\BE}_\mu(0,T) := {\BE}_\mu((0,T);X_\beta^{s,q}) := {\T}^{p}_\mu((0,T);X_\beta^{s,q}) \cap L_\mu^p((0,T);D(A)).$$
As $\T_\mu^{p}((0,T);X_\beta^{s,q}) \hookrightarrow C([0,T];W^{k,q})$ for some $k \in \Z$ the trace space 
$$X_{\gamma,\mu} := \tr(\BE_\mu((0,T);X_\beta^{s,q})) $$
is well-defined. As usual, we are going to equip the trace space with the norm 
\begin{equation*}
	\norm{g}_{X_{\gamma,\mu}} := \inf \{ \norm{u}_{\BE_\mu((0,T);X_\beta^{s,q})} \colon u(0) =  g, u \in \BE_\mu((0,T);X_\beta^{s,q}) \}.
\end{equation*}

\begin{definition}
	Let $\beta \in (0,2]$, $s \in \R$, $p,q \in (1,\infty)$, $\mu \in (\frac{1}{p},1]$ and $T \in (0,\infty)$. We assume $A(t) = (A(t))_{t \in [0,T]} \colon D(A) \to X_\beta^{s,q}$, to be a family of operators acting on functions $u \in D(A) 
	\subset X_\beta^{s,q} $ such that 
	\begin{equation*}
		t \mapsto A(t) \in L^1((0,T);\B(D(A);X_\beta^{s,q})) \cap \B(L^p_\mu((0,T);D(A));L^p_\mu((0,T);X^{s,q}_\beta)).
	\end{equation*}
	 We say that the family of operators $A(t) = (A(t))_{t \in [0,T]}$ satisfies the \textit{kinetic maximal $L^p_\mu(X_\beta^{s,q})$-regularity property} if for all $f \in L^p_\mu((0,T);X_\beta^{s,q})$ there exists a unique distributional solution $u \in  {_0 \BE}_\mu(0,T) := {_0 \BE}_\mu((0,T);X_\beta^{s,q}) := {_0\T}^{p}_\mu((0,T);X_\beta^{s,q}) \cap L_\mu^p((0,T);D(A))$ of the equation
		\begin{equation} \label{eq:kinacp}
		\begin{cases}
			\partial_t u + v\cdot \nabla_x u - A(t)u = f, \quad t \in (0,T) \\
			u(0) = 0.
		\end{cases}
	\end{equation}
\end{definition}

\begin{remark}
	\begin{enumerate}
		\item The condition on $f$ is also necessary. Indeed, if $u \in \BE_\mu(0,T)$ is a solution of the Kolmogorov equation, then $A(t) u \in L^p_\mu((0,T);X_\beta^{s,q}) $, whence $f \in L^p_\mu((0,T);X_\beta^{s,q})$. 
		\item Clearly, one can also choose a different scale of spaces than $X_\beta^{s,q}$ in the definition of kinetic maximal $L^p_\mu$-regularity. The choice of $X_\beta^{s,q}$ is motivated by the fractional Kolmogorov equation, where the operator acts only in the velocity variable. However, due to the kinetic regularity transfer one gains also some regularity in the $x$-variable, which explains the definition of $X_\beta^{s,q}$. This phenomenon will be investigated more closely in Section \ref{sec:regtransfer}.
		\item We recall that for all $s \in \R$ there exists $k \in \Z$ such that ${_0 \BE}_\mu((0,T);X_\beta^{s,q}) \hookrightarrow C([0,T];W^{k,q}(\R^{2n}))$, whence every solution $u$ of equation \eqref{eq:kinacp} attains the initial value zero in the sense of convergence in $W^{k,q}(\R^{2n})$.
		\item In the present article we concentrate on the case $s \ge 0$. However, for the fractional Kolmogorov equation we sketch a result of kinetic maximal $L^p_\mu(X_\beta^{s,2})$-regularity for all $s \ge -1/2$, i.e. when $q = 2$. The remaining case $q \neq 2$ and, in particular, the case of weak solutions, i.e. $s = -1/2$ is an interesting open problem.
	\end{enumerate}

\end{remark}

 Under additional assumptions on the operators $A(t)$ one can show that uniqueness of solutions to equation \eqref{eq:kinacp} in the space ${_0 \BE}_\mu((0,T);X_\beta^{s,q})$ is automatic.  
 
\begin{theorem} \label{thm:uniqueweak}
	Let $\beta \in (0,2]$, $s \in \R$, $p,q \in (1,\infty)$, $\mu \in (\frac{1}{p},1]$ and $T \in (0,\infty)$. We assume that there exists $k \in \Z$ such that for almost all $t  \in [0,T]$ the operator $\Gamma(t) A(t) \Gamma^{-1}(t)$ is dissipative in $W^{k,q}(\R^{2n})$ and that $X_\beta^{s,q} \subset W^{k,q}(\R^{2n})$. Then, any distributional solution $u \in {_0\BE}_\mu((0,T);X_\beta^{s,q})$ of equation \eqref{eq:kinacp} with $f = 0$ is equal to zero. 
\end{theorem}

\begin{proof}
Using the mapping $\Gamma$ we can transfer the problem of uniqueness of equation \eqref{eq:kinacp} to the uniqueness of the following nonautonomous problem
	\begin{equation} \label{eq:nonautokol}
 		\begin{cases}
 			\partial_t w = \Gamma(t) A(t) \Gamma^{-1}(t) w, \quad t>0 \\
 			w(0) = 0, 
 		\end{cases}
 	\end{equation}
 	in the class of functions $w \in W^{1,p}_\mu((0,T);W^{k,q}(\R^{2n}))$ such that 
 	$$\Gamma(t) A(t) \Gamma^{-1}(t) w \in L^p_\mu((0,T);W^{k,q}(\R^{2n})).$$
 	Indeed, let $u \in {_0\BE}_\mu((0,T);X_\beta^{s,q})$ be a solution of equation \eqref{eq:kinacp} with $f = 0$, then $ u \in {_0\T}^{p}_\mu((0,T);W^{k,q}(\R^{2n})) \cap L_\mu^p((0,T);D(A))$, whence $\Gamma u \in W^{1,p}_\mu((0,T);W^{k,q}(\R^{2n}))$ as a consequence of Proposition \ref{prop:TWtoWW}. A direct calculation shows that the function $\Gamma u$ satisfies equation \eqref{eq:nonautokol} and $\Gamma(t) A(t) \Gamma^{-1}(t) \Gamma(t) u \in L^p_\mu((0,T);W^{k,q}(\R^{2n}))$, hence uniqueness of equation \eqref{eq:nonautokol} would imply that $\Gamma u = 0$ and in particular this would imply that $u = 0$ by Proposition \ref{prop:TWtoWW}.
 	  
 	 	The calculations in the proof of \cite[Proposition 3.2]{arendt_lp-maximal_2007} hold true in our setting, too, and prove the uniqueness of solutions to equation \eqref{eq:nonautokol} as a consequence of the dissipativity of $\Gamma(t) A(t) \Gamma^{-1}(t)$ in $W^{k,q}(\R^{2n})$. More details in the case $p = q = 2$ can be found in \cite{niebel_kinetic_nodate}.
\end{proof}

\begin{coro} \label{cor:dissoflapl}
	Let $\beta \in (0,2]$, $s \in \R$, $p,q \in (1,\infty)$, $\mu \in (1/p,1]$, $T \in (0,\infty)$. Then, any solution $u \in {_0\BE}_\mu((0,T);X_\beta^{s,q})$ of the (fractional) Kolmogorov equation, i.e. equation \eqref{eq:kinacp} with $A(t) = A = -(-\Delta_v)^\frac{\beta}{2}$, is equal to zero almost everywhere.
\end{coro}

\begin{proof}
	This follows from the well-known fact that $ -(-\Delta_v)^\frac{\beta}{2}$ is dissipative in $W^{k,q}(\R^{2n})$ for any $k \in \Z$.
\end{proof}

\begin{coro}
	Let $\beta \in (0,2]$, $s \ge 0$, $p,q \in (1,\infty)$, $\mu \in (1/p,1]$, $T \in (0,\infty)$. If the family of operators $A(t)$ is dissipative in $L^q(\R^{2n})$, then any solution $u \in {_0\BE}_\mu((0,T);X_\beta^{s,q})$ of equation \eqref{eq:kinacp} is equal to zero almost everywhere.
\end{coro}

\begin{proof}
	We can choose $k = 0$ in Theorem \ref{thm:uniqueweak}. Indeed, if $A(t)$ is dissipative in $L^q(\R^{2n})$, then so is $\Gamma(t) A(t) \Gamma^{-1}(t)$.
\end{proof}

The following lemma gives a useful criterion to check whether a family of operators satisfies the kinetic maximal $L^p_\mu(X_\beta^{s,q})$-regularity property.

\begin{lemma} \label{lem:criteriamaxkinreg}
	Let $\beta \in (0,2]$, $s \in \R$, $p,q \in (1,\infty)$, $\mu \in (\frac{1}{p},1]$, $T \in (0,\infty)$ and $A(t) \colon D(A) \subset X_\beta^{s,q} \to X_\beta^{s,q}$ be a family of operators such that for all functions $f$ in a dense subset $D \subset L^p_\mu((0,T);X_\beta^{s,q})$, there exists a solution $u \in   {_0 \BE}_\mu((0,T);X_\beta^{s,q})$ of equation \eqref{eq:kinacp} and the estimate
	\begin{equation*}
		\norm{u}_{p,\mu,X_\beta^{s,q}}+\norm{\partial_t u+v \cdot \nabla_x u}_{p,\mu,X_\beta^{s,q}}+\norm{A(\cdot)u}_{p,\mu,X_\beta^{s,q}} \le c \norm{f}_{p,\mu,X_\beta^{s,q}}
	\end{equation*}
	is satisfied for some constant $c>0$. If $\Gamma(t) A(t) \Gamma^{-1}(t)$ is dissipative in some $W^{k,q}(\R^{2n}) \supset X_\beta^{s,q}$ for almost all $t \in [0,T]$, then $(A(t))_{t \in [0,T]}$ satisfies the kinetic maximal $L^p_\mu(X_\beta^{s,q})$-regularity property. 
\end{lemma}

\begin{proof}
	The claim follows using a standard approximation argument and the uniqueness proven in Theorem \ref{thm:uniqueweak}.
\end{proof}

\begin{theorem} \label{thm:kinmaxLpLp}
	Let $\beta \in (0,2]$ and $T \in (0,\infty)$, then, the operator $A(t) = A = -(-\Delta_v)^{\beta/2}$ satisfies the kinetic maximal $L^p(L^p)$-regularity property for all $p \in (1,\infty)$. Moreover for every $T_0 \in (0,\infty)$, the constant $C$ in the estimate
	\begin{equation*}
		\norm{u}_{{_0\BE}((0,T);L^p(\R^{2n}))} \le C\norm{\partial_t u+v \cdot \nabla_x u +(-\Delta_v)^{\beta/2} }_{p,}
	\end{equation*}
	can be chosen independently of $T \in (0,T_0]$.
\end{theorem}

\begin{proof}
	The $L^p(L^p)$-regularity estimates for the operator $\Delta_v$ have been known for quite some time. The result can be proven for example using the singular integral theory on homogeneous groups as stated in \cite{folland_estimates_1974}. A proof of this result for a general class of Ornstein-Uhlenbeck operators can also be found in \cite{bramanti_global_2009}. The desired result in the case $\beta = 2$ is a special case of \cite[Theorem 3]{bramanti_global_2009}.
	
	The regularity estimates in the nonlocal case $\beta \in (0,2)$ and the existence of solutions for sufficiently nice $f$ have been proven recently in \cite{chen_lp-maximal_2018}.
	
	In particular, the regularity estimate, i.e. 
	\begin{equation*}
		\norm{(-\Delta_v)^\frac{\beta}{2}u}_{L^p((0,\infty);L^p(\R^{2n}))} \le C \norm{\partial_t u+v \cdot \nabla_x u +(-\Delta_v)^{\beta/2} }_{L^p((0,\infty);L^p(\R^{2n}))}
	\end{equation*}
	holds for all $\beta \in (0,2]$. 
	
	It remains to show that the solution $u$ is an element of $L^p((0,T);L^p(\R^{2n}))$. This follows directly from the representation of solutions in terms of the fundamental solution. However, we only get 
	\begin{equation*}
		\norm{u}_{p} \le C T \norm{\partial_t u+v \cdot \nabla_x u +(-\Delta_v)^{\beta/2} }_{p}.
	\end{equation*}
	Finally, by Corollary \ref{cor:dissoflapl} we may apply Lemma \ref{lem:criteriamaxkinreg} to obtain the claim.
\end{proof}

\begin{theorem} \label{thm:kinmaxweakL2}
	Let $\beta \in (0,2]$, $s \ge -1/2$ and $\mu \in (1/2,1]$. Then, $A = -(-\Delta_v)^\frac{\beta}{2}$ satisfies the kinetic maximal $L^2_\mu(X_\beta^{s,2})$-regularity property.
\end{theorem}

\begin{proof}
	\cite[Theorem 5.13]{niebel_kinetic_nodate}
\end{proof}

\begin{theorem} \label{thm:trace}
	For all $p,q \in [1,\infty)$, any $\mu \in (\frac{1}{p},1]$ and every $T  \in (0,\infty)$ the embedding  
	\begin{equation*}
		\BE_\mu((0,T);X_\beta^{s,q}) \hookrightarrow C([0,T]; X_{\gamma,\mu})
	\end{equation*}
	holds continuously. 
\end{theorem} 

\begin{proof}
This follows similar as in the proof of \cite[Theorem 5.9]{niebel_kinetic_nodate} which is based on the proof of \cite[Proposition 1.4.2]{amann_linear_1995}.  
\end{proof}

An advantage of the introduction of temporal weights is that they allow to see nicely the instantaneous regularization. Indeed, let $0<\delta<T$, then 
\begin{equation*}
	\BE_\mu(0,T) \hookrightarrow \BE_1(\delta,T) \hookrightarrow C([\delta,T];X_{\gamma,1}).
\end{equation*}
This implies
\begin{equation*}
	\BE_\mu((0,T);X_\beta^{s,q}) \hookrightarrow C((0,T];X_{\gamma,1}).
\end{equation*}
In the classical theory the characterization of the nontrivial initial-value problem follows immediately from the maximal $L^p$-regularity property. A similar result holds true in the kinetic setting.

\begin{theorem} \label{thm:acpwiv}
		Let $\beta \in (0,2]$, $s \in \R$, $p,q \in (1,\infty)$, $\mu \in (\frac{1}{p},1]$ and $T \in (0,\infty)$. Suppose that the family of operators $(A(t))_{t \in [0,T]}$ satisfies the kinetic maximal $L^p_\mu(X_\beta^{s,q})$-regularity property on the interval $(0,T)$. Then, for all $f \in L^p_\mu((0,T);X_\beta^{s,q})$ and every $g \in X_{\gamma,\mu}$ there exists a unique solution $u \in   \BE_\mu((0,T);X_\beta^{s,q}) \cap C([0,T];X_{\gamma,\mu})$ of the Cauchy problem
		\begin{equation*}
		\begin{cases}
			\partial_t u + v\cdot \nabla_x u -A(t)u=  f, \quad t \in (0,T) \\
			u(0) = g.
		\end{cases}
	\end{equation*}
	In particular, the solution $u$ attains the initial value in the sense of convergence in the trace space $X_{\gamma,\mu}$.
\end{theorem} 

\begin{proof}
	Let us prove uniqueness first. If $u,w$ were to be two solutions to the Cauchy problem with the same data, then $u-w$ is a solution to the Cauchy problem with vanishing data. Thus, $u-w$ must be zero by the kinetic maximal $L^p_\mu(X_\beta^{s,q})$-regularity property for $(A(t))_{t \in [0,T]}$. To prove existence, choose first $u \in \BE_\mu((0,T);X_\beta^{s,q})$ such that $ u(0) = g$, then there exists a function $w \in { _0\BE}_\mu((0,T);X_\beta^{s,q})$ such that $\partial_tw + v\cdot \nabla_x w = A(t)w -\partial_t u- v \cdot \nabla_x u+ A(t)u +f $. Consequently the function $u+w$ solves the Cauchy problem with inhomogeneity $f$ and initial datum $g$.
\end{proof}
 
 \begin{theorem}
	Let $\beta \in (0,2]$, $s \in \R$, $p,q \in (1,\infty)$, $\mu \in (\frac{1}{p},1]$ and $T \in (0,\infty)$. In the autonomous case $A(t) = A$ the kinetic maximal $L^p_\mu(X_\beta^{s,q})$-regularity property is independent of the choice of $T$. Under the assumption that $A$ satisfies kinetic maximal $L^p_\mu(X_\beta^{s,q})$-regularity on some interval $(0,T)$ with $T < \infty$, we have that $A$ admits kinetic maximal $L^p_\mu(X_\beta^{s,q})$-regularity on every interval $(0,T')$ with $T'<\infty$.  
\end{theorem}

\begin{proof}
	If $T'<T$ this follows by extending $f$ by $0$ on $(T',T)$. If $T>T'$ we divide the interval $(0,T')$ in subintervals of length smaller or equal than $T$ and apply the kinetic maximal $L^p_\mu(X_\beta^{s,q})$-regularity on every subinterval. To conclude we simply glue the solutions together using Theorem \ref{thm:acpwiv}. 
\end{proof}

\begin{remark} \begin{enumerate}
		\item Kinetic maximal $L^p_\mu(L^q)$-regularity for the Kolmogorov equation \linebreak makes only sense for $T<\infty$ due to the fact that the Kolmogorov semigroup is not exponentially stable.
		\item Kinetic maximal $L^p_\mu(X_\beta^{s,q})$-regularity of the family of operators $(A(t))_{t \in [0,T]}$ is equivalent to the isomorphism property of the bounded linear operator
		\begin{equation*}
			P \colon {_0\BE}_\mu((0,T);X_\beta^{s,q}) \to L^p_\mu((0,T);X_\beta^{s,q}), \; Pu= \partial_t u + v \cdot \nabla_x u-A(t)u.
		\end{equation*}
		\item The ansatz of defining kinetic maximal $L^p_\mu(L^q)$-regularity can also be applied to different settings, where $\partial_t$ is replaced by some first order differential operator and maximal $L^p$-regularity is not available. The only requirements for the embedding into the space $C(L^q)$ are that the  characteristics of the first order differential operator generate a strongly continuous group in $L^q(\R^{2n})$ and that the respective operator is an isomorphism of $L^p((0,T);L^q(\R^{2n}))$. This applies for example to the more general Kolmogorov type operators $\partial_t + \langle Bx , \nabla \rangle$ considered in \cite{bramanti_global_2013,metafune_lp-spectrum_2001}. To treat also weak solutions, one needs to understand the equivalent of kinetic regularization to define a suitable replacement for the space $X_\beta^{s,q}$.
	\end{enumerate}
\end{remark}

\begin{lemma} \label{lem:multmaxkinreg}
	Let $\beta \in (0,2]$, $s \in \R$, $p,q \in (1,\infty)$, $\mu \in (1/p,1]$ and $T \in (0,\infty)$. For $\gamma >0$ define $S_\gamma u (t,x,v) = u(\gamma t,\gamma x, v)$. Let $A(t) = A$ be an operator satisfying kinetic maximal $L^p_\mu(X_\beta^{s,q})$-regularity such that $S_\gamma (Au) = A(S_\gamma u)$ for all $u \in L^p_\mu((0,T);D(A))$. Then, for all $\alpha >0$ the operator $\alpha A$ satisfies the kinetic maximal $L^p_\mu(X_\beta^{s,q})$-regularity property.
\end{lemma}

\begin{proof}
	This follows directly by applying the coordinate transformation $S_\gamma$. 
\end{proof}

\section{Kinetic Regularization in $L^p_\mu(L^q)$}
\label{sec:regtransfer}

When considering the operator $A = -(-\Delta_v)^\frac{\beta}{2}$, which acts only in the velocity variable, it is natural to ask why we also consider $x$-regularity in the definition of the base space $X_\beta^{s,q}$ and not only the $v$-regularity. This is due to the so-called kinetic regularization phenomenon. For $p = q$ and $\mu = 1$ it is well-known that regularity in the $v$-variable implies regularity in the $x$-variable for solutions of kinetic equations. This has been proven in the seminal article by Bouchut \cite{bouchut_hypoelliptic_2002}. We will show that this extends to $p \neq q$ and $\mu \in (\frac{1}{p},1]$. Employing the mapping $\Phi_\mu$ the weighted case follows immediately from the case $\mu = 1$. Moreover, the case $p \neq q$ can be proven by essentially the same method as in \cite{bouchut_hypoelliptic_2002} using only some refined arguments. For the readers convenience we will give the proof anyway. We note that we have collected several useful tools regarding Fourier multipliers in the appendix.

\begin{prop} \label{prop:advbouchut}
		Let $\beta \ge 0$, $p,q \in (1,\infty)$, $\mu \in (1/p,1]$, $T \in (0,\infty]$. If 
		\begin{equation*}
			u \in \T^{p}_\mu((0,T);L^q(\R^{2n})) \text{ with } D_v^\beta u \in L^p_\mu((0,T);L^q(\R^{2n})),
		\end{equation*}
		then $ D_x^\frac{\beta}{\beta+1} u \in L^p_\mu((0,T);L^q(\R^{2n}))$.
	\end{prop}
	
	\begin{proof}
		  Let $p,q \in (1,\infty)$, $\beta \in (0,\infty)$, $\mu = 1$ and define $s:= \frac{\beta}{\beta+1}$. Without loss of generality we may assume $T = \infty$. Let $\varphi \in C_c^\infty(\R^{n})$ such that $\int_{\R^n} \varphi \dx v = 1$ and $\int_{\R^n} v^\alpha \varphi(v) \dx v = 0$ for all $\alpha \in \N_0^n$ with $1\le \abs{\alpha}<  \beta$. We define $\varphi_\epsilon(v) = \frac{1}{\epsilon^n} \varphi(v/\epsilon)$ and introduce the operator 
		\begin{equation*}
			[Ru](t,x,v) = \F^{-1}_{t,x}([\varphi_{\epsilon(k)} \ast \F_{t,x}(u)](v))(t,x),
		\end{equation*}
		where
		\begin{equation*}
			\epsilon(k) = \epsilon_0 \abs{k}^{-\frac{1}{\beta+1}}
		\end{equation*}
		for some constant $\epsilon_0 >0$. We have
		\begin{equation*}
			\F_{t,x,v}(Ru)(\omega,k,\xi) = \hat{\varphi}(\epsilon(k)\xi) \F_{t,x,v}(u)(\omega,k,\xi).  
		\end{equation*}
		As in \cite{bouchut_hypoelliptic_2002} the function 
		\begin{equation*}
			\Upsilon(k,\xi) = \frac{1-\hat{\varphi}(\epsilon(k)\xi)}{\epsilon(k)^\beta \abs{\xi}^\beta}
		\end{equation*}
		defines an $L^q(\R^{2n})$ bounded Fourier multiplier. This shows 
		\begin{equation*}
			\norm{\epsilon_0^{-\beta} D_x^s(u-Ru)}_{p,q} \le C_1\norm{D_v^\beta u}_{p,q},
		\end{equation*}
		where the constant $C_1$ does not depend on $\epsilon_0$. Writing $u = u-Ru + Ru$ it remains to estimate $\norm{D_x^sRu}_{p,q}$. As $u \in \T^{p}((0,\infty);L^q(\R^{2n}))$ the function
		\begin{equation*}
			 f:= \partial_t u + v\cdot \nabla_x u
		\end{equation*} 
		belongs to $L^p((0,\infty);L^q(\R^{2n}))$. Applying the Fourier transformation in $(t,x)$ gives
		\begin{equation*}
			i(\omega+ \langle v,k \rangle) \F_{t,x} u = \F_{t,x} f.
		\end{equation*}
		We introduce the interpolation parameter $\lambda(k) = \lambda_0 \abs{k}^s$ for some constant $\lambda_0 >0$. We write
		\begin{equation*}
			\F_{t,x} (u) = \frac{\lambda(k)}{\lambda(k)+i(\omega+\langle v, k\rangle)} \F_{t,x}( u) + \frac{1}{\lambda(k)+i(\omega+\langle v, k\rangle)} \F_{t,x}( f)
		\end{equation*}
		and deduce
		\begin{align*}
			\F_{t,x}([Ru](\cdot,v))(\omega,k) &=\int_{\R^n} \frac{\lambda(k)}{\lambda(k)+i(\omega+\langle \eta, k\rangle)} \F_{t,x}(u(\cdot,\eta))(\omega,k) \varphi_{\epsilon(k)}(v-\eta) \dx \eta \\
			&+ \int_{\R^n}  \frac{1}{\lambda(k)+i(\omega+\langle \eta, k\rangle)} \F_{t,x}(f(\cdot,\eta))(\omega,k) \varphi_{\epsilon(k)}(v-\eta) \dx \eta \\
			&=: \F_{t,x}([Qu](\cdot,v))(\omega,k) +\F_{t,x}([Wf](\cdot,v))(\omega,k).
		\end{align*}
		The operator $Q$ can be written as $Q = RM$, where
		\begin{equation*}
			\F_{t,x}([Mu](\cdot,v))(\omega,k) = \frac{\lambda(k)}{\lambda(k)+i(\omega+\langle v, k\rangle)} \F_{t,x}(u(\cdot,v))(\omega,k)
		\end{equation*}
		and 
		\begin{equation*}
			\F_{t,x}(\Gamma[Mu])(\omega,k,v) = \frac{\lambda(k)}{\lambda(k) + i\omega}  [\F_{t,x}(\Gamma u)](\omega,k,v).
		\end{equation*}
		It is well-known that $(\omega,k) \mapsto \frac{\lambda(k)}{\lambda(k) + i\omega}$ defines a Fourier multiplier with constant independent of $\lambda_0$. Indeed, by Theorem \ref{thm:kinmkihlin} it defines a bounded $L^p((0,\infty);L^q(\R^{2n}))$ operator and the isometry property of $\Gamma$ gives $
			\norm{M u}_{p,q} \le C_2 \norm{u}_{p,q} $, where $C_2$ is independent of $\lambda_0$. Moreover, as in \cite{bouchut_hypoelliptic_2002} the symbol defining $R$, i.e. $(k,\xi) \mapsto \hat{\varphi}(\epsilon(k)\xi) $
		is a Fourier multiplier and hence defines a bounded $L^q(\R^{2n})$ operator with constant $C_3$ independent of $\epsilon_0>0$. We deduce
		\begin{equation*}
			\norm{Qu}_{p,q} = \norm{RMu}_{p,q} \le C_3\norm{Mu}_{p,q} \le  C_2C_3\norm{u}_{p,q}
		\end{equation*}
		for all $1 < p,q <\infty$. In \cite[Proposition 1.1]{bouchut_hypoelliptic_2002} it is shown that
		\begin{equation*}
			\norm{Qu}_{2,2} \lesssim C_4 \left(\frac{\lambda_0}{\epsilon_0} \right)^\frac{1}{2} \norm{u}_{2,2}
		\end{equation*}
		for yet another constant $C_4>0$ independent of $\epsilon_0,\lambda_0$. Interpolation of the mixed norm Lebesgue spaces in the latter two inequalities gives
		\begin{equation*}
			\norm{Qu}_{p,q} \le C_5 \left(\frac{\lambda_0}{\epsilon_0} \right)^{C_6} \norm{u}_{p,q}.
		\end{equation*}
		for some positive constants $C_5,C_6$ independent of $\epsilon_0,\lambda_0$. As $D_x^s$ commutes with $Q$ we deduce
		\begin{equation*}
			\norm{D_x^sQu}_{p,q} \le C_5 \left(\frac{\lambda_0}{\epsilon_0} \right)^{C_6} \norm{D_x^s u}_{p,q}.
		\end{equation*}
	We note that $\lambda_0 D_x^s Wf = Qf$, hence 
	\begin{equation*}
			\norm{D_x^s Wf}_{p,q} \le C_5 \frac{1}{\lambda_0} \left(\frac{\lambda_0}{\epsilon_0} \right)^{C_6} \norm{f}_{p,q}.
	\end{equation*} 
	Recalling that $u = Ru+u-Ru = Qu+Qf+u-Ru$ we conclude
	\begin{equation*}
		\norm{D_x^s u}_{p,q} \le C_5 \left(\frac{\lambda_0}{\epsilon_0} \right)^{C_6} \norm{D_x^s u}_{p,q} + C_5 \frac{1}{\lambda_0} \left(\frac{\lambda_0}{\epsilon_0} \right)^{C_6} \norm{f}_{p,q} + C_1 \epsilon_0^\beta \norm{D_v^\beta u}_{p,q}
	\end{equation*}
	and choosing $\lambda_0/\epsilon_0$ small enough we deduce the claim. 
	\end{proof}

\begin{remark}
	The proof given here holds true for the $L^{\bar{q}}$-space with mixed norm respective to the integrability exponents $\bar{q} = (q_1,\dots,q_{2n})$, too. 
\end{remark}

\section{Kinetic maximal $L^p(L^q)$-regularity for the Kolmogorov equation}
\label{sec:LpLq}
When dealing with the Kolmogorov equation it proves to be very useful to employ the representation formulae of solutions in terms of the fundamental solution. Let us consider a function $u \colon [0,\infty) \times \R^{2n} \to \R$, $u = u(t,x,v)$, which is a solution of the (fractional) Kolmogorov equation
\begin{equation*}
	\begin{cases}
		\partial_t u +v \cdot \nabla_x u = -(-\Delta_v)^{\frac{\beta}{2}} u +f \\
		u(0) = g
	\end{cases}
\end{equation*}
where $\beta \in (0,2]$ and with sufficiently smooth data $f$ and $g$. The Fourier transform in $(x,v)$  with respective Fourier variables $(k,\xi)$ of a function $u$ will be denoted by $\hat{u}$. At least formally applying the Fourier transform gives
\begin{equation*}
	\begin{cases}
		\partial_t \hat{u} - k  \cdot \nabla_\xi \hat{u} = -\abs{\xi}^\beta \hat{u} + \hat{f} \\
		\hat{u}(0) = \hat{g}.
	\end{cases}
\end{equation*}
In the following we will use the function $e_\beta \colon [0,\infty) \times \R^{2n} \to \R$ defined as
\begin{equation*}
	e_\beta(t,k,\xi) = \exp( -\int_0^t \abs{\xi+(t-\tau) k}^\beta \dx \tau)
\end{equation*}
for $\beta \in (0,2]$. A solution to the Fourier transformed Kolmogorov equation can be explicitly given by means of the method of characteristics as
\begin{equation} \label{eq:fouriersol}
	\hat{u}(t,k,\xi) = \hat{g}(k,\xi + tk) e_\beta(t,k,\xi) + \int_0^t e_\beta(t-s,k,\xi) \hat{f}(s,k,\xi+(t-s)k) \dx s
\end{equation}
for sufficiently nice functions $f$ and $g$. For the operator $A = A_\beta = -(-\Delta_v)^{\beta/2}$ with domain $D(A) = H^{\beta,q}_v(\R^{2n})$ we recall that 
\begin{align*}
	\BE_\mu((0,T);L^q(\R^{2n})) &=\T^{p}_\mu((0,T);L^q(\R^{2n})) \cap L^p_\mu((0,T);H^{\beta,q}_v(\R^{2n})) \\
	 &= \T^{p}_\mu((0,T);L^q(\R^{2n})) \cap L^p_\mu((0,T);X_\beta^{1,q})
\end{align*}
by Proposition \ref{prop:advbouchut}.

\begin{theorem} \label{thm:kinmaxLpLq}
	For all $p,q \in (1,\infty)$ the fractional Laplacian in velocity $A = -(-\Delta_v)^\frac{\beta}{2}$ satisfies the kinetic maximal $L^p(L^q)$-regularity property. Again, for every $T_0 \in (0,\infty)$, the constant $C$ in the estimate
	\begin{equation*}
		\norm{u}_{{_0\BE}((0,T);L^q(\R^{2n}))} \le C\norm{\partial_t u+v \cdot \nabla_x u +(-\Delta_v)^{\beta/2} }_{p,L^q(\R^{2n})}
	\end{equation*}
	can be chosen independently of $T \in (0,T_0]$.
\end{theorem}

\begin{proof}
First we observe the following property of $\Gamma(t)$ in connection with the Fourier transformation. For suitable $g \colon \R^{2n} \to \R $ we have
\begin{equation*}
	\F(\Gamma(t)g)(k,\xi) = \F(g)(k,\xi-tk)
\end{equation*}
for all $t \in \R$. The transformation on the right-hand side defines an operator similar to $\Gamma$.

	It suffices to show that the operator 
	\begin{align*}
		&S \colon L^p((0,T);L^q(\R^{2n})) \to L^p((0,T);L^q(\R^{2n})), \\
		&[Sf](t) = \F^{-1}\left(\abs{\xi}^\beta \int_0^t e_\beta(t-s,k,\xi) \hat{f}(s,k,\xi+(t-s)k)  \dx s\right)
	\end{align*}
	is bounded or equivalently that the operator
	\begin{align*}
		&\tilde{S} \colon L^p((0,T);L^q(\R^{2n})) \to L^p((0,T);L^q(\R^{2n})), \\ 
		&[\tilde{S}f](t) = \Gamma S f =  \int_0^t \F^{-1}\left(\abs{\xi-tk}^\beta  e_\beta(t-s,k,\xi-tk) \hat{f}(s,k,\xi) \right) \dx s
	\end{align*}
	is bounded. This is due to the isometry property of $\Gamma$ in $L^q(\R^{2n})$. 
	
	We introduce the Fourier multiplication operators $k(t,s) \colon L^q(\R^{2n}) \to L^q(\R^{2n})$ defined as
	\begin{equation*}
		\kappa(t,s)g = \F^{-1}(\abs{\xi-tk}^\beta  e_\beta(t-s,k,\xi-tk) \hat{g}(k,\xi)) \mathds{1}_{0 \le s \le t \le T},
	\end{equation*}
	then $\tilde{S}f = \int_0^Tk(t,s)f(s)\dx s$. According to Theorem \ref{thm:kinmaxLpLp} the operator $\tilde{S}$ is bounded for $p = q$. In spirit of \cite[Theorem 2]{rubio_de_francia_calderon-zygmund_1986} it suffices to show that
	\begin{equation} \label{eq:hoer1}
		\sup_{s,s' \in (0,T)} \int_{2\abs{s-s'}\le \abs{t-s}} \norm{\kappa(t,s)- \kappa(t,s')}_{\B(L^q(\R^{2n}))} \dx t < \infty
	\end{equation}
	and
	\begin{equation} \label{eq:hoer2}
		\sup_{s,s' \in (0,T)} \int_{2\abs{s-s'}\le \abs{t-s}} \norm{\kappa(s,t)- \kappa(s',t)}_{\B(L^q(\R^{2n}))} \dx t < \infty.
	\end{equation}
	Actually the results in \cite{rubio_de_francia_calderon-zygmund_1986} are formulated for operators on vector-valued $L^p(\R)$ spaces. However, they continue to hold in the case where $\R$ is replaced by $(0,T)$ as stated in \cite[p. 15]{rubio_de_francia_calderon-zygmund_1986}.
	
	We first prove that the condition in equation \eqref{eq:hoer1} is satisfied. 
	Let $s,s' \in (0,T)$ and $g$ a smooth function, then
	\begin{align*}
		&\F \left(\kappa(t,s)\hat{g}-\kappa(t,s')\hat{g} \right)(k,\xi) \\
		&= \abs{\xi-tk}^\beta \left[ e_\beta(t-s,k,\xi-tk)-e_\beta(t-s',k,\xi-tk) \right]\hat{g}(k,\xi) \\
		&= \abs{\xi-tk}^\beta \int_{s'}^s \partial_r e_\beta(t-r,k,\xi-tk) \hat{g}(k,\xi) \dx r  \\
		&= \int_{s'}^s \left[ \abs{\xi-tk}^{2\beta}-\beta  \abs{\xi-tk}^\beta \int_0^{t-r} \abs{\xi-(w+r)k}^{\beta-2} \langle \xi-(w+r)k,k \rangle \dx w \right] \\\
		& \quad \quad \quad e_\beta(t-r,k,\xi-tk) \hat{g}(k,\xi) \dx r \\
		&=: \int_{s'}^s m_1(t,r,k,\xi) \hat{g}(k,\xi) \dx r
	\end{align*}
	and
	\begin{align*}
		&\F(\kappa(s,t)\hat{g}-\kappa(s',t)\hat{g})(k,\xi) \\
		&=  \left[\abs{\xi-sk}^\beta e_\beta(s-t,k,\xi-sk)-\abs{\xi-s'k}^\beta e_\beta(s'-t,k,\xi-s'k) \right]\hat{g}(k,\xi) \\
		&=  \int_{s'}^s \partial_r \left( \abs{\xi-rk}^\beta  e_\beta(r-t,k,\xi-rk) \right) \hat{g}(k,\xi) \dx r  \\
		&=  \int_{s'}^s \left[-\abs{\xi-rk}^{\beta-2} \langle \xi-rk,k \rangle + \abs{\xi-rk}^{2\beta} \right] e_\beta(r-t,k,\xi-rk) \hat{g}(k,\xi) \dx r \\
		&=: \int_{s'}^s \left[ m_2(t,r,k,\xi) +m_3(t,r,k,\xi) \right] \hat{g}(k,\xi)  \dx r.
	\end{align*}
	It suffices to show that each of the above two integrands defines a bounded Fourier multiplier with constant $C(t-r)^{-2}$. It is then an easy calculation to show that the conditions in \cite[Theorem 2]{benedek_convolution_1962} are satisfied, see for example \cite[Lemma 11]{haak_maxregnonauto_2015}.

	Let us consider the multiplier $m_1$ first. By the considerations in Lemma \ref{lem:dilkinmik} and Lemma \ref{lem:gammamult} it suffices to study the following multiplier where we have substituted $\tau = t-r$, $\xi = \xi-tk$ and $k = \tau k$
	\begin{equation*}
		\left[ \abs{\xi}^{2\beta}-\beta \abs{\xi}^\beta \int_0^1 \abs{\xi+wk}^{\beta-2} \langle \xi+wk,k \rangle \dx w \right] \exp(-\tau \int_0^1 \abs{\xi+wk}^\beta \dx w),
	\end{equation*}	
	since
	\begin{align*}
		e_\beta(\tau,\frac{k}{\tau},\xi) &= \exp( -\int_0^\tau \abs{\xi+(\tau-r) \frac{1}{\tau} k}^\beta \dx r) = \exp( -\int_0^\tau \abs{\xi+w \frac{1}{\tau} k}^\beta \dx w)\\
		&=\exp(-\tau \int_0^1 \abs{\xi+wk}^\beta \dx w).
	\end{align*}
	Moreover, the prefactor of the exponential function is $2\beta$-homogeneous and hence it suffices to consider the multiplier
	\begin{equation*}
		\left[ \abs{\xi}^{2\beta}+ \abs{k}^{2\beta}\right] \exp(-\tau \int_1^2 \abs{\xi+wk}^\beta \dx w).
	\end{equation*}	
	Using Lemma \ref{lem:itsez} we can further reduce to problem and see that it suffices to study the classical Fourier multiplier
	\begin{equation*}
		\left[ \abs{\xi}^{2\beta}+ \abs{k}^{2\beta}\right] \exp(-\frac{c}{2}\tau \abs{\xi}^\beta-\frac{c}{2}\tau \abs{k}^\beta),
	\end{equation*}
	which is a bounded $L^q$-multiplier with constant $C(\beta,n)\tau^{-2}$. Combining the latter observations together with the fact that product of Fourier multipliers are Fourier multipliers we deduce the desired estimate for $m_1$. For the boundedness of $m_3$ one can argue exactly the same. For $m_2$ we note that the substitutions $\tau=r-t$, $\xi = \xi - rk$ and $k = kt$ lead to the multiplier
	\begin{equation*}
		\tau^{-1}\abs{\xi}^\frac{\beta}{2}\langle \xi,k \rangle \exp(-\tau \int_0^1 \abs{\xi+wk}^\beta \dx w),
	\end{equation*}
	which is by the same argument as above can be bounded by $C(n,\beta)\tau^{-2}$. We have shown that the operator $\tilde{S}$ satisfies the condition in Lemma \ref{lem:criteriamaxkinreg} and hence, we conclude that it is bounded. Finally, for all $f \in L^p((0,T);L^q(\R^{2n}))$ the solution $u$ of the Kolmogorov equation satisfies $u \in L^p((0,T);H_v^{\beta,q}(\R^{2n}))$. Moreover, as the Kolmogorov semigroup is contractive it follows immediately that $u \in L^p((0,T);L^q(\R^{2n}))$ as $T< \infty$ with the same estimate as in the proof of Theorem \ref{thm:kinmaxLpLp}, which shows that $u \in \BE((0,T);L^q(\R^{2n}))$.
\end{proof}

\begin{remark}
	It is classical that maximal $L^p$-regularity for autonomous operators is independent of $p$. Here we extend this principle to the kinetic case but only for the special case of the fractional Laplacian. Moreover, we make use of the fundamental solution. We have used a technique which was also applied in \cite{
	haak_maxregnonauto_2015,hieber_pseudo-differential_2000} to show the same result for maximal $L^p$-regularity in the setting of nonautonomous operators. 
\end{remark}

\begin{theorem}
	For all $p \in (1,\infty)$ and any $s \ge -1/2$ the fractional Laplacian $A = -(-\Delta_v)^\frac{\beta}{2}$ satisfies the kinetic maximal $L^p(X_\beta^{s,2})$-regularity property. 
\end{theorem}

\begin{proof}
	Again, it suffices to show that the operator 
	\begin{align*}
		&{S} \colon L^p((0,T);X_\beta^{s,2}) \to L^p((0,T);X_\beta^{s,2}), \\
		&[{S}f](t,k,\xi) = \F^{-1}\left(\abs{\xi}^\beta \int_0^t \hat{f}(r,k,\xi+(t-r)k) e_\beta(t-r,k,\xi) \dx r\right)
	\end{align*}
	is bounded or equivalently that the operator
	\begin{align*}
		&\tilde{S} \colon L^p((0,T);L^2(\R^{2n})) \to L^p((0,T);L^2(\R^{2n})), \\ 
		&[\tilde{S}f](t) = \int_0^t \F^{-1}\left(\abs{\xi-tk}^\beta  e_\beta(t-r,k,\xi-tk) \frac{(1+\abs{k}^\frac{2\beta}{\beta+1}+\abs{\xi-tk}^{2\beta})^\frac{s}{2}}{(1+\abs{k}^\frac{2\beta}{\beta+1}+\abs{\xi-rk}^{2\beta})^\frac{s}{2}} \hat{f}(r,k,\xi)\right) \dx r \\
		&= \int_0^T \kappa(t,r)f(r) \dx r,
	\end{align*}
	is bounded. As the boundedness of $S$ is known for $p = 2$ by Theorem \ref{thm:kinmaxweakL2} the present theorem can be proven using the exact same technique as in the proof of Theorem \ref{thm:kinmaxLpLq}. However, due to the theorem of Plancherel one only needs to show the $L^\infty$ bound on the multipliers, which simplifies the proof. We omit the details. 
\end{proof}

\section{Kinetic maximal $L^p(X_\beta^{s,q})$-regularity with temporal weights}
\label{sec:kinmaxlpmulq}
An important result in the theory of maximal $L^p$-regularity with temporal weights is that it is equivalent to unweighted maximal $L^p$-regularity. Inspired by this result and the proof given in \cite{pruss_maximal_2004} we prove a generalization of this theorem to the case of kinetic maximal $L^p_\mu(X_\beta^{s,q})$-regularity. To use the ideas of the proof given in \cite{pruss_maximal_2004} we need a suitable representation of solutions. Such a representation can be given for example in terms of the fundamental solution. We note that as in the classical case we can only treat autonomous case $A(t) = A$. Let us first collect some useful tools. 

\begin{coro} \label{cor:EmuE}
	Let $\beta \in (0,2]$, $s \in \R$, $p,q \in (1,\infty)$, $\mu \in (1/p,1]$ and $T \in (0,\infty)$. The mapping $\Phi_\mu \colon \BE_\mu((0,T);X_\beta^{s,q}) \to \BE((0,T);X_\beta^{s,q}) $ is an isomorphism.
\end{coro}

\begin{proof}
	This follows directly from Proposition \ref{prop:Phimuprop} and Corollary \ref{cor:TmuT}.
\end{proof}

\begin{prop} \label{prop:integralopest}
	Let $\beta \in (0,2]$, $s \in \R$, $p,q \in (1,\infty)$, $\mu \in \left( \frac{1}{p},1 \right]$, $T>0$ and $A \colon D(A) \subset X_\beta^{s,q} \to X_\beta^{s,q}$. We assume $K \in L^1_\mathrm{loc}((0,T);\mathcal{B}(X_\beta^{s,q}))$ to be such that $K(t)f \in D(A) $ for all $f \in X_\beta^{s,q}$ and $\norm{A K(t)}_{\B({X_\beta^{s,q}})} \le Mt^{-1}$ for some $M>0$ and almost all $t \in (0,T)$. Then, the linear operator $R \colon L^p((0,T);X_\beta^{s,q}) \to L^p((0,T);X_\beta^{s,q})$ defined by
	\begin{align*}
		[Rf](t) = \int_0^t K(t-s) \left[ \left( \frac{t}{s} \right)^{1-\mu}-1 \right] f(s) \dx s
	\end{align*}
	on functions $f \in L^p((0,T);X_\beta^{s,q})$ satisfies the estimate $\norm{A R f}_{p,X_\beta^{s,q}} \le c \norm{f}_{p,X_\beta^{s,q}}$ for all  $f \in L^p((0,T);X_\beta^{s,q})$.
	\end{prop}

\begin{proof}
	This is a special case of \cite[Proposition 2.3]{pruss_maximal_2004}. 
\end{proof}

\begin{theorem} \label{thm:1impliesmu}
	Let $s \in \R$, $p,q \in (1,\infty)$ and $\mu \in (\frac{1}{p},1]$. Assume that the dissipative operator $A \colon D(A) \subset X_\beta^{s,q}$ satisfies the following properties: 
	\begin{enumerate}
		\item The operator $A-v\cdot \nabla_x$, with domain
		\begin{equation*}
			\{ u \in X_\beta^{s,q} \colon Au, v\cdot \nabla_x u \in X_\beta^{s,q} \}
		\end{equation*}
		generates a strongly continuous bounded $C_0$-semigroup $T(t)$ in $X_\beta^{s,q}$.
		\item For all inhomogeneities $f \in C_c^\infty((0,T)\times \R^{2n})$ the solution of the Cauchy problem with zero initial datum can be given as
		\begin{equation*}
			u(t) = \int_0^t T(t-s)f(s) \dx s.
		\end{equation*}
		\item Assume that $T(t)g \in D(A)$ for all $g \in X_\beta^{s,q}$ and almost all $t>0$ and that the estimate
		\begin{equation*}
			\norm{AT(t)g}_{X_\beta^{s,q}} \lesssim t^{-1} \norm{g}_{X_\beta^{s,q}}
		\end{equation*}
		is satisfied. 
	\end{enumerate}
	Under these assumptions $A$ satisfies kinetic maximal $L^p(X_\beta^{s,q})$-regularity if and only if $A$ satisfies kinetic maximal $L^p_\mu(X_\beta^{s,q})$-regularity. 
\end{theorem}

\begin{proof}
	Suppose that $A$ satisfies the kinetic maximal $L^p(X_\beta^{s,q})$-regularity property. We note that $C_c^\infty((0,T) \times \R^{2n})$ is dense in $L^p_\mu((0,T);X_\beta^{s,q})$. Let $\mu \in (\frac{1}{p},1]$ and $f \in C_c^\infty((0,T)\times \R^{2n})$. We define
	\begin{equation*}
		u(t) = \int_0^t T(t-s)f(s) \dx s
	\end{equation*}
	for any $t \in (0,T)$. Extending $f$ by zero for $t > T$ and using Hardy's inequality we deduce
	\begin{align*}
		&\int_0^\infty \norm{t^{1-\mu} \int_0^t T(t-s) f(s) \dx s  }^p_{X_\beta^{s,q}} \dx t \\
		&\le \int_0^\infty \left( t^{1-\mu} \int_0^t \norm{f(s)}_{X_\beta^{s,q}} \dx s \right)^p \dx t \le c \int_0^\infty  t^{2p-\mu p}  \norm{f(t)}^p_{X_\beta^{s,q}} \dx t \le cT^p \norm{f}_{p,\mu,X_\beta^{s,q}}^p,
	\end{align*}
	whence $u \in L^p_\mu((0,T);X^{s,q}_\beta)$. As in \cite{pruss_maximal_2004} we write
	\begin{align*}
		u(t) &= t^{\mu-1} \int_0^t T(t-s)s^{1-\mu} f(s) \dx s \\
		&+t^{\mu -1} \int_0^t T(t-s) \left[ \left( \frac{t}{s}\right)^{1-\mu} -1 \right]s^{1-\mu} f(s) \dx s \\
		&= \Phi_\mu^{-1} \left[ \left( \partial_t + v \cdot \nabla_x - A \right)^{-1} \Phi_\mu f + R \Phi_\mu f \right] \\
		&=:\Phi_\mu^{-1} \left[ w_1+w_2 \right],
	\end{align*}
	with $R$ as in Proposition \ref{prop:integralopest}, where $K(\cdot) = AT(\cdot)$. As $\Phi_\mu f \in L^p((0,T);X_\beta^{s,q})$ we deduce that $w_1 \in {_0 \BE}((0,T);X_\beta^{s,q})$ by the assumed kinetic maximal $L^p(X_\beta^{s,q})$-regularity. 
	
	\noindent Next, we prove that $w_2 \in {_0 \BE}((0,T);X_\beta^{s,q})$. By assumption (3) we can apply Proposition \ref{prop:integralopest} to deduce that $A w_2 \in L^p((0,T);X_\beta^{s,q})$. Furthermore, we have 
	\begin{equation*} 
		(\partial_t + v \cdot \nabla_x) w_2 = Aw_2 + (1-\mu)t^{-\mu }\int_{0}^t T(t-s)f(s) \dx s,
	\end{equation*}
	with 
	\begin{align*}
		&\int_0^\infty \norm{t^{-\mu} \int_0^t T(t-s) f(s) \dx s  }^p_{X_\beta^{s,q}} \dx t \\
		&\le c\int_0^\infty \left( t^{-\mu} \int_0^t \norm{f(s)}_{X_\beta^{s,q}} \dx s \right)^p \dx t \le c\int_0^\infty  t^{p-\mu p}  \norm{f(t)}^p_{X_\beta^{s,q}} \dx t = c \norm{f}_{p,\mu,X_\beta^{s,q}}^p
	\end{align*}
	as the consequence of another application of Hardy's inequality. We conclude that $w_2=R\Phi_\mu f \in {_0\T_\mu^p}((0,T);X_\beta^{s,q})$, whence we deduce $w_2 \in  {_0 \BE}((0,T);X_\beta^{s,q})$. By Corollary \ref{cor:EmuE} we conclude $u \in {_0 \BE}_\mu((0,T);X_\beta^{s,q}) $. Consequently, $u$ is a solution to the Kolmogorov equation in the space of weighted kinetic maximal $L^p(X_\beta^{s,q})$-regularity and satisfies
	\begin{equation*}
		\norm{u}_{{_0 \BE}_\mu((0,T);X_\beta^{s,q})} \lesssim \norm{f}_{p,\mu,X_\beta^{s,q}}
	\end{equation*}
	by Corollary \ref{prop:Phimuprop}. Using Lemma \ref{lem:criteriamaxkinreg} it follows that $A$ satisfies the kinetic maximal $L^p_\mu(X_\beta^{s,q})$-regularity property. 
	
	The converse implication can be shown by writing 
	\begin{equation*}
		u = \Phi_\mu (\partial_t+v\cdot \nabla_x -A)^{-1} \Phi_\mu^{-1}f - \Phi_\mu R\Phi_\mu^{-1}f
	\end{equation*}
	and proceeding as above.    
\end{proof}
	
	\begin{remark}
		In the classical theory of maximal $L^p$-regularity one can show that every $L^p$-solution is a mild solution, hence is given by $u(t) = \int_0^t T(t-s)f(s) \dx s$. It is not known to the authors, whether one can show a similar result in the kinetic setting, too. Such a result could be used to drop assumption (2) in Theorem \ref{thm:1impliesmu}.
	\end{remark}
	
	Let us now prove that the Kolmogorov equation satisfies the property of kinetic maximal $L^p_\mu(L^q)$-regularity. By Theorem \ref{thm:kinmaxLpLq} it satisfies the kinetic maximal $L^p(L^q)$-regularity property. Therefore, it remains to check the assumptions of Theorem \ref{thm:1impliesmu}. We recall the kernel representation for the semigroup corresponding to the (fractional) Kolmogorov equation. We denote by $G_\beta$ the fundamental solution of the Kolmogorov equation. If $\beta = 2$ we know that
	\begin{equation*}
		G_2(t,x,v) = \frac{\gamma_n}{t^{2n}} \exp\left(- \frac{1}{t}\abs{v}^2+ \frac{3}{t^2} \langle {v},x\rangle - \frac{3}{t^3} \abs{x}^2 \right),
	\end{equation*}
where $\gamma_n = \sqrt{3}^n(2\pi)^{-n}$ and if $\beta \neq 2$ only a representation in Fourier variables is known. Henceforth, we are going to denote by $(T_\beta(t))_{t \ge 0}$ the Kolmogorov semigroup in $L^q(\R^{2n})$, i.e. 
\begin{align}
	[T_\beta(t)g](x,v) &= \int_{\R^{n}}\int_{\R^{n}} G_\beta(t,x-\tilde{x}-t\tilde{v},v-\tilde{v})g(\tilde{x},\tilde{v})\dx \tilde{v} \dx \tilde{x} \label{eq:repKol}
\end{align}
for all $g \in L^q(\R^{2n})$. More information on the Kolmogorov semigroup for $\beta = 2$ can be found in \cite[Chapter 10]{lorenzi_analytical_2017} and in \cite{chen_lp-maximal_2018} for $\beta \in (0,2)$. Using the results in \cite{chen_lp-maximal_2018} one can show that for $\beta  \in (0,2)$ the Kolmogorov equation induces a strongly continuous semigroup in $L^q(\R^{2n})$. Using this notation we can write every solution of the Kolmogorov equation for sufficiently smooth data as
\begin{equation*}
	u(t) = T_\beta(t)g + \int_0^t T_\beta (t-s)f(s) \dx s.
\end{equation*}
The following decay estimates show that assumption (3) of Theorem \ref{thm:1impliesmu} is satisfied for $\beta = 2$.

\begin{lemma} \label{lem:decayestimsemi}
	Let $q \in (1,\infty)$, $g \in L^q(\R^{2n})$, then $\Delta_v T_2(t)g \in L^q(\R^{2n})$ for all $t>0$ and
	\begin{equation*}
		\norm{\Delta_v T_2(t)g}_{q} \le Ct^{-1} \norm{g}_{q}.
	\end{equation*}
\end{lemma}

\begin{proof}
	We have
	\begin{align*}
		[\Delta_v T_2(t)g](x,v) &= \int_{\R^{n}} \int_{\R^{n}} G_2(t,x-\tilde{x}-t\tilde{v},v-\tilde{v}) \left[- \frac{4}{t^2}\abs{v-\tilde{v}}^2+ \frac{9}{t^4}\abs{x- \tilde{x}-t \tilde{v}}^2  \right.  \\
		&\hphantom{+\int_{\R^{n}} \int_{\R^{n}} G_2(t,x-\tilde{x}-t\tilde{v},v-\tilde{v}) +} \left.- \frac{12}{t^3} \langle v- \tilde{v},x-\tilde{x}-t\tilde{v} \rangle - \frac{2n}{t}  \right] g(\tilde{x},\tilde{v}) \dx \tilde{v} \dx \tilde{v}
	\end{align*}
	Using the generalized Young inequality from Lemma \ref{lem:genYoung} we conclude
	\begin{equation*}
		\norm{\Delta_v T_2(t)f}_{q,\R^{2n}} \le \int_{\R^n} \int_{\R^n} G_2(t,x,v) \left| -\frac{2}{t^2} \abs{v}^2-\frac{6}{t^3} \langle v,x \rangle + \frac{3}{t^4} \abs{x}^2 - \frac{2n}{t} \right| \dx x \dx v \norm{g}_{q,\R^{2n}}.
	\end{equation*}
	The Cauchy-Schwarz and Young's inequality imply
	\begin{equation*}
		\abs{G_2(t,x,v)} \le \frac{\gamma_n}{t^{2n}} \exp(-\frac{c}{t} \abs{v}^2- \frac{c}{t^3}\abs{x}^2)
	\end{equation*}
	for all $x,v \in \R^n$ and every $t>0$ for some constant $c>0$. We have
	\begin{align*}
		&\int_{\R^n} \int_{\R^n} G_2(t,x,v) \left|  -\frac{2}{t^2} \abs{v}^2-\frac{6}{t^3} \langle v,x \rangle + \frac{3}{t^4} \abs{x}^2 - \frac{2n}{t} \right| \dx x \dx v \\ & \le \frac{c}{t^2} \int_{\R^n} \int_{\R^n} G_2(t,x,v) |v|^2 \dx x \dx v+ \frac{c}{t^4} \int_{\R^n} \int_{\R^n} G_2(t,x,v) |x|^2 \dx x \dx v + \frac{2n}{t}= I_1+I_2+\frac{2n}{t}.
	\end{align*}
	The first integral can be estimated as 
	\begin{align*}
		I_1 &= \frac{c}{t^2} \int_{\R^n} \int_{\R^n} G_2(t,x,v) |v|^2 \dx x \dx v \le \frac{c}{t^{2+2n}} \int_{\R^n} \int_{\R^n} \exp(-\frac{c}{t} \abs{v}^2- \frac{c}{t^3}\abs{x}^2) |v|^2 \dx x \dx v \\
		&= \frac{c}{t^{2+2n}} \int_{\R^n} \exp(- \frac{c}{t^3}\abs{x}^2) \dx x \int_{\R^n} |v|^2 \exp(-\frac{c}{t} \abs{v}^2)  \dx v \\
		&= \frac{c}{t^{2+2n}} \int_0^\infty r^{n-1} \exp(-c\frac{r^2}{t^3}) \dx r \int_0^\infty r^{n+1} \exp(-c \frac{r^2}{t}) \dx r = ct^{-1}.
	\end{align*}
	Similarly the second integral can be estimated as $I_2 \le ct^{-1}$. This shows $\norm{ \Delta_v T_2(t)g}_{q,\R^{2n}} \le C t^{-1} \norm{g}_{q,\R^{2n}}$. 
\end{proof}

\begin{lemma}
	For all $q \in (1,\infty)$, $\beta \in (0,2)$ we have
	\begin{equation*}
		\norm{(-\Delta_v)^{\beta/2}T_\beta(t)g}_{q} \lesssim t^{-1} \norm{g}_{q}
	\end{equation*}
	for all $t>0$ and any $g \in L^q(\R^{2n})$.
\end{lemma}

\begin{proof}
	Similarly as in the proof of Lemma \ref{lem:decayestimsemi} one can show that
	\begin{equation*}
		\norm{\Delta_v T_\beta (t)g}_q \lesssim t^{-\frac{2}{\beta}} \norm{g}_q,
	\end{equation*}
	by using \cite[Remark 2.6]{chen_lp-maximal_2018}. The desired estimate follows then by interpolation of the latter estimate and the contractivity estimate for the Kolmogorov semigroup. 
\end{proof}

\begin{coro} \label{cor:kinmaxLpmuLqkol}
	For all $p,q \in (1,\infty)$, any $ \mu \in (\frac{1}{p},1]$ and all $\beta \in (0,2]$ the operator $A = -(-\Delta_v)^\frac{\beta}{2}$ satisfies the kinetic maximal $L^p_\mu(L^q)$-regularity property. 
\end{coro}

\begin{coro} \label{cor:coordlaplace}
	For all $p,q \in (1,\infty)$, any $ \mu \in (\frac{1}{p},1]$ and every positive definite symmetric matrix $a \in \R^{n\times n}$ the operator $Au = a \colon \nabla_v^2 u$ with domain $D(A) = H^{2,q}_v(\R^{2n})$ satisfies the kinetic maximal $L^p_\mu(L^q)$-regularity property. In particular for all $T>0$ and any $u \in {_0 \BE}_\mu((0,T);L^q(\R^{2n}))$ we have
	\begin{equation*}
		\norm{u}_{{_0\BE}_\mu(0,T)} \le C(p,q,\mu,T) \lambda^{-1} \norm{\partial_t u + v \cdot \nabla_x u - Au}_{p,\mu,L^q(\R^{2n})},
	\end{equation*}
	where $\lambda>0$ is the smallest eigenvalue of $a$. Given any finite interval $[0,T_0]$ the constant $C=C(T_0)$ can be chosen independently of $T \in (0,T_0]$. 
\end{coro}

\begin{proof}
	This follows using the coordinate transformation $(t,x,v) \mapsto (t,A^{-\frac{1}{2}}x,A^{-\frac{1}{2}}v)$. 
\end{proof}

\section{Characterization of the trace space of strong solutions to the Kolmogorov equation}

The trace space for strong solutions of the Kolmogorov equation in the case $p = q = 2$ and $\mu \in (1/2,1]$ can be given in terms of anisotropic fractional Sobolev spaces. This has been shown in the article \cite{niebel_kinetic_nodate}. In this section we give a characterization of the trace space in terms of anisotropic Besov spaces in the general case.

 Let $\beta \in (0,2]$, $s \ge 0 $, $p,q \in (1, \infty)$, $\mu \in \left( \frac{1}{p},1 \right]$ and $T \in (0,\infty)$. For the Kolmogorov equation we recall that 
 \begin{align*}
 	\BE_\mu(0,T) =\BE_\mu((0,T);X_\beta^{s,q}) &= \T^{p}_\mu((0,T);X_\beta^{s,q}) \cap L^p_\mu((0,T);X_\beta^{s+1,q}) \\
 	&= \T^{p}_\mu((0,T);H^{s\beta,q}_v(\R^{2n})) \cap L^p_\mu((0,T);H^{(s+1)\beta,q}_v(\R^{2n}))
 \end{align*}
 by Proposition \ref{prop:advbouchut} for $s = 0$ and an induction argument, whence by interpolation for all $s \ge 0$.

First, we present one of the equivalent definitions of anisotropic Besov spaces as stated in \cite[Section 5]{dachkovski_anisotropic_2003}. Let $n \in \N$, $p,q \in (1,\infty)$, $\sigma \in \R$ and $\alpha \in [0,\infty)^{2n} $. Given $r \in \R$ and $z \in \R^{2n}$ we write $t^\alpha z = (t^{\alpha_1}z_1,\dots,t^{\alpha_{2n}}z_{2n})$ and $t^{r\alpha}z = (t^{r})^\alpha z$ for $t>0$. We denote by $\abs{(x,v)}_\alpha$ a homogeneous norm of a vector $(x,v) \in \R^{2n}$ with respect to the dilation induced by $\alpha$. For example we can choose
\begin{equation*}
	\abs{(x,v)}_\alpha = \sum_{i = 1}^{n} \abs{x_i}^{1/\alpha_i} + \abs{v_i}^{1/\alpha_{n+i}}.
\end{equation*}
Note however, that this is only quasi-distance but not a norm in the proper sense. Let us choose any functions $\varphi_0,\varphi \in \S(\R^{2n})$ satisfying the conditions
\begin{enumerate}
	\item $\varphi_0(\cdot) = \varphi_0(-\cdot)$ and $\varphi(\cdot) = \varphi(- \cdot)$,
	\item ${\hat{\varphi}_0(k,\xi)}>0 $ if $\abs{(k,\xi)}_\alpha < 2$ and ${\hat{\varphi}_0(k,\xi)} = 0 $ for $\abs{(k,\xi)}_\alpha > 2$,
	\item ${\hat{\varphi}(k,\xi)} >0 $ if $ \frac{1}{2} < \abs{(k,\xi)}_\alpha < 2$ and ${\hat{\varphi}(k,\xi)}  = 0$ else.
\end{enumerate}

We assume that $\varphi,\varphi_0$ are Schwartz, thus we can drop the additional assumptions given in \cite{dachkovski_anisotropic_2003}, where they only assume that $\varphi,\varphi_0 \in C^\infty$.
We define
	\begin{equation*}
		\norm{u | B_{qp}^{\sigma,\alpha}(\R^{2n})} =   \norm{\varphi_0 \ast u}_q + \left( \int_0^C \left( \frac{\norm{\varphi_t \ast u}_q}{t^{\sigma}}  \right)^p \frac{\dx t}{t}\right)^\frac{1}{p},
	\end{equation*}
	where we write $\varphi_t(x,v) = t^{-2n}\varphi(t^{-\alpha}(x,v))$ and $C>0$ is some constant. The anisotropic Besov space is defined as
	\begin{equation*}
	B_{qp}^{\sigma,\alpha}(\R^{2n}) = \{ f \in \S'(\R^{2n}) \colon \norm{f | B_{qp}^{\sigma,\alpha}(\R^{2n})}<\infty \}.
\end{equation*}
As in the isotropic setting this definition is independent of the choice of $\varphi,\varphi_0$ and $C$.

Clearly, the trace space of $\BE_\mu((0,T);X_\beta^{s,q})$ should also control some regularity in the $x$ variable. In view of the results in \cite{niebel_kinetic_nodate} a natural choice for the trace space seems to be
	\begin{equation*}
		\bigcap_{r = 1}^n B_{qp,x_r}^{\frac{\beta}{\beta+1}a} \cap B_{qp,v_r}^{\beta a}
	\end{equation*}
	for a suitable choice of $a>0$. Choosing $\alpha = (\frac{2(\beta+1)}{\beta+2},\dots,\frac{2(\beta+1)}{\beta+2},\frac{2}{\beta+2},\dots,\frac{2}{\beta+2})$ and $\sigma = \frac{2a\beta}{\beta+2}$ we have that 
	\begin{equation*}
		B_{qp}^{\sigma,\alpha} \cong \bigcap_{r = 1}^n B_{qp,x_r}^{\frac{\beta}{\beta+1}a} \cap B_{qp,v_r}^{\beta a}
	\end{equation*}
	as stated in \cite[Section 5.1.3]{triebel_theory_2006}. We introduce the kinetic Besov space of order $a$ and regularity scale corresponding to $\beta$ as
	\begin{equation*}
		{^{\mathrm{kin}}B}_{qp}^{a,\beta}(\R^{2n}) := B_{qp}^{\sigma,\alpha}(\R^{2n}).
	\end{equation*}
	Of particular interest is the case $\beta = 2$, where $\alpha = (\frac{3}{2},\dots,\frac{3}{2},\frac{1}{2},\dots,\frac{1}{2})$ and $\sigma = a$. The goal of this section is to prove the following theorem.
	
\begin{theorem} \label{thm:chartrace}
	For all $\beta \in (0,2]$, $s \ge 0$, any $p,q \in (1,\infty)$, every $\mu \in (\frac{1}{p},1]$, and all $T \in (0,\infty)$ we have 
	\begin{equation*}
		{^{\mathrm{kin}}B}_{qp}^{s+\mu-1/p,\beta}(\R^{2n}) \cong \tr \left( \BE_\mu((0,T);X_\beta^{s,q}) \right).
	\end{equation*}
	Moreover, there exists a constant $c>0$ such that 
	\begin{equation*}
		c^{-1}\norm{u \left| {^{\mathrm{kin}}B}_{qp}^{s+\mu-1/p,\beta}(\R^{2n}) \right.} \le \norm{u \left| \tr \left( \BE_\mu((0,T);X_\beta^{s,q}) \right) \right.} \le c\norm{u \left| {^{\mathrm{kin}}B}_{qp}^{s+\mu-1/p,\beta}(\R^{2n}) \right.}
	\end{equation*} 
	for all $u \in \tr \left( \BE_\mu((0,T);X_\beta^{s,q} \right)$.
\end{theorem}

\begin{prop} \label{prop:necivp}
	Let $\beta \in (0,2]$, $s \ge 0$, $p,q \in (1,\infty)$, $\mu \in (\frac{1}{p},1]$, $T \in (0,\infty)$ and let $u \in {\BE}_\mu((0,T);X_\beta^{s,q})$ be a solution of the homogeneous Kolmogorov equation given by equation \eqref{eq:fouriersol}, then $u(0) \in {^{\mathrm{kin}}B}_{qp}^{s+\mu-1/p,\beta}(\R^{2n})$. 
\end{prop}
	
\begin{proof}[Proof of Proposition \ref{prop:necivp}]
	 Basically, we use the fact that a function closely related to the fundamental solution of the Kolmogorov equation can be used as the function $\varphi$ in above equivalent norm of the anisotropic Besov space $B^{\sigma,\alpha}_{qp}$ with $\sigma = \frac{2\beta (s+\mu-1/p)}{\beta+2}$, i.e. on ${^{\mathrm{kin}}B}_{qp}^{s+\mu-1/p,\beta}(\R^{2n})$. We have
	\begin{align*}
		&\left( \int_0^T t^{p-p\mu}\norm{[D_x^{\frac{\beta}{\beta+1}(s+1)}+(-\Delta_v)^{\beta(s+1)/2}] u(t)}_q^p \dx t \right)^\frac{1}{p} \\
		&=  \left( \int_0^T t^{p-\mu p}\norm{\F^{-1}([\abs{k}^{\frac{\beta}{\beta+1}(s+1)}+ \abs{\xi-tk}^{\beta(s+1)}]e_\beta(t,k,\xi-tk)\hat{g}(k,\xi))}_q^p {\dx t} \right)^\frac{1}{p} \\
		&= c\left( \int_0^{T^\frac{\beta+2}{2\beta}} \frac{\norm{\F^{-1}(\hat{\Phi}(t,k,\xi)\hat{g}(k,\xi))}_q^p}{t^{\frac{2\beta}{\beta+2}(sp+\mu p-1) }} \frac{\dx t}{t} \right)^\frac{1}{p},
	\end{align*}
	with 
	\begin{equation*}
		\hat{\Phi}(t,k,\xi) = [t^{\frac{2\beta}{\beta+2}(s+1)}\abs{k}^{\frac{\beta}{\beta+1}(s+1)}+ t^{\frac{2\beta}{\beta+2}(s+1)}\abs{\xi-t^{\frac{2\beta}{\beta+2}}k}^{\beta(s+1)}]e_\beta(t^{\frac{2\beta}{\beta+2}},k,\xi-t^{\frac{2\beta}{\beta+2}}k).
	\end{equation*}
	From
	\begin{equation*}
		e_\beta(t,k,\xi-tk) = \exp(- \int_0^t \abs{\xi-rk}^\beta \dx r) = \exp(-t\int_0^1 \abs{\xi-rtk}^\beta \dx r)
	\end{equation*}
	we infer that $\hat{\Phi}(t,k,\xi) = \hat{\Phi}(1,t^\alpha (k,\xi))$, whence $\Phi(t,x,v) = t^{-2n} \Phi(1,t^{-\alpha}(x,v))$. However, $\Phi(1,\cdot,\cdot)$ cannot be used as a function in the definition of the equivalent Besov norm directly. We therefore choose a function $\beta \in \S(\R^{2n})$ such that $\beta(\cdot) = \beta(-\cdot)$ and such that the Fourier transform of $\beta $ satisfies $\supp \hat{\beta} = \{ k,\xi \in \R^{n} \colon \frac{1}{2} \le \abs{(k,\xi)}_\alpha \le 2 \}$ and $\hat{\beta}(\xi)>0$ for $\frac{1}{2} < \abs{(k,\xi)}_\alpha < 2$. Then, the function $\varphi = \beta \ast \Phi(1,\cdot)$ satisfies the conditions (1)-(3) above and
	\begin{equation*}
		\varphi_t = \beta_t \ast \Phi(t,\cdot,\cdot)
	\end{equation*}
	holds for all $t>0$. Using Young's inequality we deduce 
	\begin{equation*}
		\left\|\varphi_{t} * g\right\|_{{q}}=\left\|\beta_{t} * \Phi(t,\cdot,\cdot) * g\right\|_{{q}} \leq\left\|\beta_{t}\right\|_{{1}}\left\|\Phi(t,\cdot,\cdot) * g\right\|_{q}=\|\beta\|_{{1}}\left\|\Phi(t,\cdot,\cdot) * g\right\|_{{q}}.
	\end{equation*}
	Let $\varphi_0 \in \S$ be any function satisfying the conditions (1) and (2), then 
	\begin{align*}
		\norm{g |{^{\mathrm{kin}}B}_{qp}^{\mu-1/p,\beta}(\R^{2n})} &= \norm{\varphi_0 \ast g}_q + \left( \int_0^{T^\frac{\beta+2}{2\beta}} \left( \frac{\norm{\varphi_t \ast g}_q}{t^{\sigma}}  \right)^p \frac{\dx t}{t}\right)^\frac{1}{p} \\ 
		&\le \norm{\varphi_0}_1\norm{g}_q + \norm{\beta}_1\left( \int_0^{T^\frac{\beta+2}{2\beta}} \frac{\norm{\Phi(t,\cdot,\cdot) \ast g}_q^p}{t^{\sigma p}} \frac{\dx t}{t} \right)^\frac{1}{p} \\
		&\lesssim \norm{u}_{\BE_\mu(0,T)}.
	\end{align*}  
\end{proof}

\begin{prop} \label{prop:suffiv}
	Let $\beta \in (0,2]$, $s \ge 0$, $p,q \in (1,\infty)$, $\mu \in (\frac{1}{p},1]$, $T \in (0,\infty)$ and let $g \in {^{\mathrm{kin}}B}_{qp}^{s+\mu-1/p,\beta}(\R^{2n})$. Then, the solution of the homogeneous Kolmogorov equation given by equation \eqref{eq:fouriersol} satisfies $u \in {\BE}_\mu((0,T);X_\beta^{s,q})$.
\end{prop}

\begin{proof}[Proof]
	By Proposition \ref{prop:advbouchut} it suffices to show that $(-\Delta_v)^{\frac{\beta}{2}(s+1)}u \in L^p_\mu((0,T);L^q(\R^{2n}))$ and that $u \in L^p_\mu((0,T);L^q(\R^{2n}))$. We have
	\begin{align*}
		\int_0^Tt^{p-\mu p} \norm{(-\Delta_v)^{\frac{\beta}{2}(s+1)}u(t)}^p_q \dx t &= \int_0^T t^{p-\mu p} \norm{\Gamma(t)[(-\Delta_v)^{\frac{\beta}{2}(s+1)}u(t)]}^p_q \dx t \\
		&= \int_0^T t^{p-\mu p} \norm{\F^{-1}(\abs{\xi-tk}^{\beta(s+1)}e_\beta(t,k,\xi-tk)\hat{g}(k,\xi))}_q^p\dx t.
	\end{align*}
	Using that 
	\begin{equation*}
		m(k,\xi) = \frac{\abs{\xi-tk}^{\beta(s+1)}}{\abs{\xi}^{\beta(s+1)}+t^{\beta(s+1)} \abs{k}^{\beta(s+1)}}
	\end{equation*}
	is a bounded $L^q(\R^{2n})$-Fourier multiplier with constant independent of $t$, and employing Lemma \ref{lem:dilkinmik} and Lemma \ref{lem:itsez} we deduce
	\begin{align*}
		&\int_0^T t^{p-\mu p} \norm{(-\Delta_v)^{\frac{\beta}{2}(s+1)}u(t)}^p_q \dx t \\
		&\lesssim \int_0^T t^{p-\mu p} \norm{\F^{-1}([\abs{\xi}^{\beta(s+1)}+t^{\beta(s+1)} \abs{k}^{\beta(s+1)}]\exp(-ct\abs{\xi}^\beta/2-ct^{\beta+1}\abs{k}^\beta/2) \hat{g}(k,\xi))}_q^p \dx t
	\end{align*} 
	with $c$ as in Lemma \ref{lem:itsez}. But the latter expression is bounded by the ${^{\mathrm{kin}}B}_{qp}^{s+\mu-1/p,\beta}(\R^{2n})$-norm of $g$ as a consequence of Proposition \ref{prop:helpprob}. Finally, it is clear that 
	\begin{equation*}
		\norm{u}_{p,\mu,L^q(\R^{2n})} \le C(T,\mu) \norm{g}_{L^q(\R^{2n})} \lesssim \norm{g | {^{\mathrm{kin}}B}_{qp}^{s+\mu-1/p,\beta}(\R^{2n})},
	\end{equation*}
	due to the contractivity of the Kolmogorov semigroup and by Lemma \ref{lem:besovLq}.
\end{proof}

\begin{proof}[Proof of Theorem \ref{thm:chartrace}]
	The considerations in Proposition \ref{prop:necivp}, the linearity and uniqueness of solutions to the Kolmogorov equation show that the inclusion mapping
	\begin{equation*}
		\iota \colon \tr \left(\BE_\mu(0,T)\right) \to  {^{\mathrm{kin}}B}_{qp}^{s+\mu -1/p,\beta}(\R^{2n})
	\end{equation*}
	is a well-defined linear and bounded operator. Indeed, let $\epsilon>0$, then there is $u \in \BE_\mu(0,T)$ such that 
		$$\norm{u}_{\BE_\mu(0,T)} < \norm{g}_{\tr \left(\BE_\mu(0,T) \right)} + \epsilon.$$
	 Writing $\partial_t u +v \cdot \nabla_x u = -(-\Delta_v )^\frac{\beta}{2}u +f$ with $u(0) = g$, where $f = \partial_t u +v \cdot \nabla_x u +(-\Delta_v )^\frac{\beta}{2}u$ we have that $u$ is a solution of the Kolmogorov equation. For the homogeneous part $w$ of the solution we deduce
	 \begin{align*}
	 	\norm{g|{^{\mathrm{kin}}B}_{qp}^{s+\mu-1/p,\beta}(\R^{2n})} &\lesssim \norm{w}_{\BE_\mu(0,T)} \lesssim \norm{w-u}_{\BE_\mu(0,T)} +\norm{u}_{\BE_\mu(0,T)} \\
	 	&\lesssim \norm{f}_{p,\mu,X_\beta^{s,q}} +\norm{u}_{\BE_\mu(0,T)} \le  2\norm{u}_{\BE_\mu(0,T)} \\
	 	&\lesssim \norm{g}_{\tr \left(\BE_\mu(0,T) \right)} + \epsilon
	 \end{align*} 
	 by Proposition \ref{prop:necivp} and an approximation argument.
	
	Next, we are going to show that $\iota$ is surjective. To do so, given $g \in {^{\mathrm{kin}}B}_{qp}^{s+\mu-1/p,\beta}(\R^{2n})$ we need to construct a function $u \in \BE_\mu(0,T) $ such that $u(0) = g$. Here we can choose $u$ to be the solution of the homogeneous Kolmogorov equation with initial value $g$ as a consequence of Proposition \ref{prop:suffiv}. 
	This shows that there exists $u \in \BE_\mu(0,T)$ with $u(0) = g$.  We conclude the proof of Theorem \ref{thm:chartrace} by the open mapping Theorem.
\end{proof}

\begin{remark}
	Let us go back to the classical maximal $L^p$-regularity result that the trace space is given by a real interpolation space. We have shown that
	\begin{equation*}
		{^{\mathrm{kin}}B}_{qp}^{\mu-1/p,\beta}(\R^{2n}) \cong \tr \left( \BE_\mu((0,T);L^q(\R^{2n})) \right)
	\end{equation*} 
	and the space on the left-hand side can be viewed as the real interpolation space
	\begin{equation*}
\left(L^q(\R^{2n}),H_x^{\frac{\beta}{\beta+1},q}(\R^{2n}) \cap H_v^{\beta,q}(\R^{2n}) \right)_{\mu-1/p,p}.
	\end{equation*}
\end{remark}

\begin{coro} \label{cor:continbesovxv}
	Let $\beta \in (0,2]$, $s \ge 0$, $p,q \in (1,\infty)$, $ \mu \in (1/p,1]$ and $T \in (0,\infty)$. The embedding 
	\begin{equation*}
		\BE_\mu((0,T);X_\beta^{s,q}) \hookrightarrow C([0,T], {^{\mathrm{kin}}B}_{qp}^{s+\mu-1/p,\beta}(\R^{2n})) 
	\end{equation*}
	holds continuously.
\end{coro}

\section{Kinetic maximal $L^p_\mu$-regularity for the (fractional) Kolmogorov equation in spaces of higher regularity}

In this section we will deduce that the (fractional) Kolmogorov equation satisfies the kinetic maximal $L^p_\mu(X_\beta^{s,q})$-regularity property for all $s \ge 0$. By differentiation of the Kolmogorov equation we are able to show the following theorem. 

\begin{theorem}
	For all $s \in \N_0$, $p,q \in (1,\infty)$ and any $\mu \in (1/p,1]$ the operator $A = -(-\Delta_v)^\frac{\beta}{2}$ satisfies the kinetic maximal $L^p_\mu(X_2^{s/2,q})$- regularity property. 
\end{theorem}

\begin{proof}

	Let $s  \ge 0$, then $X_\beta^{s/\beta,q} = H_x^{s\frac{1}{\beta+1},q}(\R^{2n}) \cap H_v^{s,q}(\R^{2n})$. For $s = 0$ the result has already been proven. We are going to prove the general case by induction. Let $f \in L^p((0,T);X_\beta^{(s+1)/\beta,q})$, then the assumed kinetic maximal $L^p_\mu(X_\beta^{s/\beta,q})$-regularity property implies the existence of a unique solution $u \in {_0 \BE}_\mu((0,T);X_\beta^{s/\beta,q})$ of the Kolmogorov equation with right hand-side $f$ and zero initial datum. 
	
	It remains to show that $u \in \BE_\mu((0,T);X_\beta^{(s+1)/\beta,q})$. Differentiating the Kolmogorov equation with respect to $x$ yields
	\begin{equation*}
		\partial_t D_x^\frac{1}{\beta+1}u +v \cdot \nabla_x D_x^\frac{1}{\beta+1}u = \Delta_v D_x^\frac{1}{\beta+1}u + D_x^\frac{1}{\beta+1} f,
	\end{equation*}
	which, by the kinetic maximal $L^p_\mu(X_\beta^{s/\beta,q})$-regularity property, shows that 
	$$D_x^\frac{1}{\beta+1}u \in \BE_\mu((0,T);X_\beta^{(s+1)/\beta,q}),$$
	whence $D_x u = D_x^\frac{1}{\beta+1}D_x^\frac{\beta}{\beta+1}u \in L^p((0,T);X_\beta^{s/\beta,q})$. Next, we differentiate with respect to $v$ to the result
	\begin{equation*}
		\partial_t \partial_{v_i} u + v \cdot \nabla_x \partial_{v_i} u = \Delta_v \partial_{v_i} u - \partial_{x_i} u + \partial_{v_i} f.
	\end{equation*}
	As $\partial_{x_i} u \in \BE_\mu((0,T);X_\beta^{s/\beta,q})$ the kinetic maximal $L^p_\mu(X_\beta^{s/\beta,q})$-regularity property implies $u \in \BE_\mu((0,T);X_\beta^{(s+1)/\beta,q})$.
\end{proof}

\begin{coro} \label{coro:kinmaxsge0}
	For all $\beta \in (0,2]$, $s \ge 0$, $p,q \in (1,\infty)$ and any $\mu \in (1/p,1]$ the operator $A = -(-\Delta_v)^\frac{\beta}{2}$ satisfies the kinetic maximal $L^p_\mu(X_\beta^{s,q})$-regularity property. 
\end{coro}

\begin{proof}
	This follows by interpolation and the latter theorem. 
\end{proof}

\begin{coro} 
	Let $\beta \in (0,2]$, $s \ge 0$, $p,q \in (1,\infty)$, $\mu \in (1/p,1]$ and $g \in {^{\mathrm{kin}}B}_{qp}^{s+\mu-1/p,\beta}(\R^{2n})$. Then the solution $u$ to the homogeneous (fractional) Kolmogorov equation with initial value $g$ regularizes instantaneously, i.e. we have $u \in C^\infty((0,\infty)\times \R^{2n}) $. 
\end{coro}

\begin{proof}
	The proof follows along the lines of the proof of \cite[Theorem 6.4]{niebel_kinetic_nodate} using Corollary \ref{coro:kinmaxsge0} and Theorem \ref{thm:chartrace}.
\end{proof}

It is an interesting open question whether one can extend this technique also to the case where the base space is a Sobolev space of negative order, i.e. the case of weak solutions. However, the commutator term cannot be controlled as easily as in the above situation. A natural approach would be to apply the operator $J^s = (\id+D_x^\frac{2}{3}+D_v^2)^\frac{s}{2}$ to the equation, which results in
\begin{equation*}
	\partial_t J_v^s u + v \cdot \nabla_x J_v^s u = \Delta_v J_v^su -s\mathcal{F}^{-1}(\langle k, \xi \rangle (1+ \abs{k}^\frac{2}{3}+ \abs{\xi}^2)^{-1} \widehat{J_v^s u}) + J_v^s f.
\end{equation*}
The commutator is not of lower order which causes several problems. We emphasize that kinetic maximal $L^p_\mu(X_\beta^{s,q})$-regularity of $A = -(-\Delta_v)^\frac{\beta}{2}$ for $s<0$ is a very interesting open problem.

\section{The Kolmogorov equation with bounded and uniformly continuous coefficients}
\label{sec:kolcoeff} 
In this section we are going to show that the Kolmogorov equation with uniformly continuous and bounded coefficients satisfies the kinetic maximal $L^p_\mu(L^q)$-regularity property. A similar result for degenerate Ornstein-Uhlenbeck operators has already been proven in \cite{bramanti_global_2013} but only for $p = q$ and $\mu = 1$ and no lower order terms. We compare their result to ours in the remark after the following theorem.  
 
\begin{theorem} \label{thm:kinmaxLpmuLqvarcoeff}
 Let $T>0$, $a  \in L^\infty([0,T] \times \R^{2n}; \mathrm{Sym}(n))$ with $\Gamma a \in BUC([0,T] \times \R^{2n}; \mathrm{Sym}(n))$ such that $a(t,x,v) \ge \lambda \id $ for all $(t,x,v) \in [0,T] \times \R^{2n}$ in the sense of positive definite matrices for some constant $\lambda >0$. Furthermore, let $b \in L^\infty([0,T]\times \R^{2n};\R^{n})$ and $c \in L^\infty([0,T]\times \R^{2n};\R) $. For all $p,q \in (1,\infty)$ and $\mu \in (\frac{1}{p},1]$, the Cauchy problem \begin{equation*}
	\begin{cases}
		\partial_t u +v \cdot \nabla_x u = a \colon \nabla_v^2 u + b \cdot \nabla_v u + cu  +f, \\
		u(0) = g
	\end{cases}
\end{equation*}
admits a unique solution $u \in \BE_\mu(0,T) = \T^{p}_\mu((0,T);L^q(\R^{2n})) \cap L^p_\mu((0,T);H^{2/3,q}_x(\R^{2n}) \cap H^{2,q}_v(\R^{2n}))$ if and only if
\begin{enumerate}
	\item $f \in L^p_\mu((0,T);L^q(\R^{2n}))$ and
	\item $g \in {^{\mathrm{kin}}B}_{qp}^{\mu-1/p,2}(\R^{2n})$.
\end{enumerate}
In other words the operator $u \mapsto  a \colon \nabla_v^2 u + b \cdot \nabla_v u + cu$ satisfies the kinetic maximal $L^p_\mu(L^q)$-regularity property.
\end{theorem}

\begin{proof} 
For the sake of clearness of  the following argumentation we introduce the spaces $X = L^p_\mu((0,\delta);L^q(\R^{2n}))$ with the norm $\norm{\cdot}_X = \norm{\cdot}_{p,\mu,L^q(\R^{2n})}$ and $ Z = {_0 \BE}_\mu(0,\delta)$ equipped with the respective norm $\norm{\cdot}_{Z}$ for some $\delta >0$, which will be chosen at a later point. 

We are going to show that the operator
\begin{equation*}
		P \colon Z \to X, \; \partial_t + v \cdot \nabla_x - a \colon \nabla_v^2 - b \cdot \nabla_v u - cu
\end{equation*}
is an isomorphism for some $\delta >0$. As a consequence of Theorem \ref{thm:acpwiv} the result follows for non-zero initial value, too. As the constant $\delta>0$ will depend only on the modulus of continuity of $\Gamma a$ and other universal constants, by translation, we can iterate this argument to deduce the claim on the intervals $[\delta,2\delta]$, $[2\delta,3\delta], \dots$ and conclude the claim on the interval $[0,T]$ by gluing the separate solutions together.

First, let us show that $P$ satisfies $\norm{Pu}_X \ge C \norm{u}_Z$ for some constant $C = C(\lambda,T,p,q,\mu)$, which, in particular, implies that $P$ is injective. Let $\epsilon>0$, due to uniform continuity of $\Gamma a$ there exist $z_1=(x_1,v_1),z_2=(x_2,v_2),\dots \hphantom{} \in \R^{2n}$ and $\delta_0>0$ such that
	\begin{equation*}
		\abs{a(t,x+tv,v)-a(0,z_k)} \le \epsilon
	\end{equation*}
	for all $0 \le t\le\delta \le \delta_0$ and any $z \in \R^{2n}$ with $\abs{z-z_k} < \delta_0$. Furthermore $\delta_0$ should be chosen such that the sets $U_k = B_{\delta_0}(z_k)$, $k \in \N$ are a covering of $\R^{2n}$ with the property that $U_k \cap U_j \neq 0$ for at most a fixed number $M = M(n) \in \N$ of indices $j,k \in \N$. Let $(\eta_k)_{k \in \N} \subset C^\infty(\R^{2n})$ be a partition of unity such that $\sum_{k = 1}^\infty \eta_k = 1$, $0 \le \eta_k \le 1$ with $\supp \eta_k \subset U_k$. Additionally, we assume that $\norm{\nabla \eta_k}_\infty,\norm{\nabla \eta_k^2}_\infty \le C_1$. We define $\varphi_k(t,x,v) = [\Gamma(-t) \eta_k](x,v) =  \eta_k(x-tv,v)$, so that $\partial_t \varphi_k + v \cdot \nabla_x \varphi_k = 0$. Clearly, $(\varphi_k(t,\cdot))_{k \in \N}$ is still a partition of unity of $\R^{2n}$ for all $t \in [0,T]$ and $\norm{\nabla \varphi_k}_\infty,\norm{\nabla \varphi_k^2}_\infty \le C_1(T,n,\delta_0)$.  Furthermore, we have
	\begin{equation*}
		1 \le \sum_{k = 1}^\infty \mathds{1}_{(0,1]}(\varphi_k(t,x,v)) \le M
	\end{equation*}
	for all $(t,x,v) \in [0,T] \times \R^{2n}$.
	
 To simplify notation we introduce the following differential operators 
	\begin{equation*}
		Au = a(t,x,v) \colon \nabla_v^2u, \; A_k u = a(0,z_k) \colon \nabla_v^2 u \text{ and } Lu = b \cdot \nabla_v u + cu.
	\end{equation*}
	We localize the equation $Pu =: f$ by multiplying with the test function $\varphi_k$ 
	\begin{align} \label{eq:loceq}
		\partial_t (\varphi_k u) + v \cdot \nabla_x (\varphi_k u) &= A_k(\varphi_k u) + \varphi_k f \\
		&+ \varphi_k L u -2\langle a(0,z_k) \nabla_v \varphi_k ,\nabla u \rangle -(a(0,z_k) \colon \nabla_v^2 \varphi_k)u \nonumber \\ 
		&+\varphi_k \left( A( u)-A_k(u) \right). \nonumber
	\end{align}
	Let us investigate the frozen equation more closely. For all $k \in \N$, by Corollary \ref{cor:coordlaplace}, the equation 
	\begin{equation*}
		\begin{cases}
			\partial_t w + v \cdot \nabla_x w = A_kw + f \\
			w(0) = 0
		\end{cases}
	\end{equation*}
	possesses a unique solution $\mathcal{L}_k(f) = w \in Z$ if and only if $f \in X$. Moreover, there exists a constant $C_2=C_2(p,q,\mu,\lambda,T)>0$ independent of $k$ and $\delta \in [0,T]$ such that 
	\begin{equation*}
		\norm{\mathcal{L}_k(f)}_{Z} \le C_2 \norm{f}_X.
	\end{equation*}
	Here, the independence of the constant $C_2$ of $k$ follows from the uniform ellipticity assumption on the coefficient matrix $a$. 
	
	Applying the solution operator $\L_k \colon X \to Z$ to equation \eqref{eq:loceq} gives
	\begin{equation*}
		\norm{\varphi_k u}_Z \le  \norm{ \L_k(\varphi_k f)}_Z + \norm{\L_k(R_k(u))}_Z + \norm{ S_k(\varphi_k u)}_Z ,  
	\end{equation*}
	where 
	\begin{equation*}
		R_k(u) = \varphi_k L u -2\langle A\nabla_v \varphi_k ,\nabla u \rangle -(A \colon \nabla_v^2 \varphi_k)u
	\end{equation*}
	and 
	\begin{equation*}
		S_k(u) = \L_k\left[\varphi_k (A(u)-A_k(u))\right]. 
	\end{equation*}
	Let us first estimate
	\begin{align*}
		\norm{S_k (u)}_Z &= \norm{\L_k(\varphi_k(A(u)-A_k(u)))}_Z \le C_2 \norm{\varphi_k(A(u)-A_k(u))}_X \\
		&= C_2 \norm{\varphi_k (a(\cdot,\cdot)-a(0,z_k)) \colon \nabla_v^2 u}_X \le C_2\epsilon \norm{\varphi_k \nabla_v^2 u}_X, 
	\end{align*}
	where $C_3$ depends only on the bound $C_1$ on the derivatives of $\varphi_k$. Here, we have used uniform continuity of $\Gamma a$. Next, we estimate
	\begin{equation*}
		\sum_{k = 1}^\infty \norm{R_k(u)}_X \le C_ 4 \norm{\nabla_v u}_X + C_4\norm{u}_X
	\end{equation*}
	with $C_4 = C_4(\norm{a}_\infty,\norm{b}_\infty,\norm{c}_\infty,n,\delta_0,p)$. Consequently,
	\begin{align*}
		\norm{u}_Z &\le  \sum_{k = 1}^\infty {\norm{\varphi_k u}_Z} \le C_2C_3M\epsilon \norm{\nabla_v^2 u}_X+C_2 \sum_{k = 1}^\infty \norm{\varphi_k f}_{X} + \norm{R_k(u)}_X  \\
		&\le C_2C_3M\epsilon \norm{u}_Z+2C_2M \norm{f}_{X} +C_ 4 \norm{\nabla_v u}_X + C_4\norm{u}_X.
	\end{align*}
	By interpolation of Sobolev spaces we have
	\begin{equation*}
		 \norm{\nabla_v u}_X \le \epsilon \norm{u}_Z + C_5\norm{u}_X
	\end{equation*}
	with $C_5 = C_5(\epsilon) >0$. Combining these estimates we conclude
	\begin{align*}
		\norm{u}_Z &\le  2C_2M \norm{f}_{X} + (C_2C_3M\epsilon+C_ 4 \epsilon) \norm{u}_Z + (C_4 + C_4C_5) \norm{u}_X.
	\end{align*}	
	It can be easily shown that for all $w \in {_0 W}^{1,p}_\mu((0,\delta);L^p(\R^{2n}))$ we have
	\begin{equation*}
		\norm{w}_{p,\mu,L^q(\R^{2n})} \le \delta \norm{\partial_t w}_{p,\mu,L^q(\R^{2n})}.
	\end{equation*}
	Consequently, by using the mapping property of $\Gamma$ it follows
	\begin{equation*}
		\norm{w}_{X}= \norm{w}_{p,\mu,L^q(\R^{2n})} \le \delta \norm{w}_{{_0\T}^{p}_\mu((0,\delta);L^q(\R^{2n}))} \le \delta \norm{w}_Z
	\end{equation*}
	for all $w \in Z$. Choosing first $\epsilon \le \frac{1}{4}(C_4+C_2C_3M)^{-1}$ and then $\delta \le \min \{ \frac{1}{4} (C_2C_4(\epsilon))^{-1}, \delta_0 \}$ we deduce
	\begin{equation*}
		\norm{u}_Z \le 2C_2 \norm{f}_X = 2C_2\norm{Pu}.
	\end{equation*}		
	It remains to verify that $P$ is surjective. We are going to use the method of continuity to prove this. For $s \in [0,1]$ we define
	\begin{equation*}
		P(s) \colon Z \to X, \; u \mapsto (1-s) P+s(\partial_t +v \cdot \nabla_x u-\lambda \Delta_v u)
	\end{equation*}
	A direct estimation shows that $[0,1] \to \B(Z,X)$, $s \mapsto P(s)$ is well-defined and norm continuous. Moreover, as the ellipticity constant of $P(s)$ is still given by $\lambda$ we deduce that there exists a constant $C = C(p,q,\mu,\lambda,T,\delta_0)$ such that
	\begin{equation*}
		C\norm{u}_Z \le  \norm{P(s)u}_X
	\end{equation*}
	for all $s \in [0,1]$ and any $u \in Z$. Consequently, by the method of continuity, $P=P(0)$ must be surjective as it is already known from Theorem \ref{cor:kinmaxLpmuLqkol} and Lemma \ref{lem:multmaxkinreg} that $P(1)$ is surjective.  
\end{proof}

\begin{remark}
	Let us compare our result to the one given in \cite{bramanti_global_2013}. In \cite{bramanti_global_2013} it is only assumed that $a \in BUC([0,T]\times \R^{2n})$ and not that $\Gamma a \in BUC([0,T]\times \R^{2n})$. However, there seems to be an inconsistency hidden in the proof given in \cite{bramanti_global_2013}. To be more precise, we note that in the proof of \cite[Proposition 3.2]{bramanti_global_2013} they use the module of continuity of $a$ with respect to the euclidean norm. However, in the proof of their main result \cite[Theorem 3.1]{bramanti_global_2013} they need the local estimate of \cite[Proposition 3.2]{bramanti_global_2013} on balls with respect to the local quasi-symmetric quasi-distance $d(z,\xi) = \norm{\xi^{-1} \circ z}$ and hence one should use the module of continuity with respect to $d$ in the proof of \cite[Proposition 3.2]{bramanti_global_2013}, which is in our case equivalent to the assumption $\Gamma a \in BUC$. 
	
	Moreover, we note that while in \cite{bramanti_global_2013} they use a deep result on the covering of $[0,T] \times \R^{2n}$ with balls with respect to the quasi distance $d$, we directly give such a covering and our ansatz even extends to the more general case considered in \cite{bramanti_global_2013}.

Let us finally note that the conditions $\Gamma a \in BUC([0,T]\times \R^{2n})$ and $ a \in BUC([0,T]\times \R^{2n})$ are not equivalent. In fact, let $n = 1$, $a(t,x,v) = \max \{ \min\{ \abs{x},1 \}, 1/2 \}$, then $a \in BUC([0,T]\times \R^{2n})$. We assume that $\Gamma a \in BUC([0,T]\times \R^{2n})$. Then, there exists $\delta >0$ such that
\begin{equation*}
	\frac{1}{4}> \abs{a(0,0,v)-a(0,0+\delta/2 v,v)} = 1/2- \max \{ \min \{ \delta/2 \abs{v},1\}, 1/2 \}
\end{equation*} 
for all $v \in \R$. In the limit $v \to \infty$ this leads to a contradiction, whence $\Gamma a \notin BUC([0,T]\times \R^{2n})$. Noting that $a(t,x,v) = a(t,x+tv-tv,v)$ it is clear how to construct an example of a function $a$ such that $\Gamma a \in BUC([0,T]\times \R^{2n})$ and $a \notin BUC([0,T]\times \R^{2n})$.
\end{remark}

\begin{remark}
	As a next step one would like to prove a similar result for kernel operators of the form 
	\begin{equation*}
		[A(t) u](x,v) = \pv \int_{\R^n} ({u(x,v+w)-u(x,v)}) \frac{a(t,x,v,w)}{\abs{w}^{n+\beta}} \dx w.
	\end{equation*}
	If $a(t,x,v,w) = a(t,x,v)$ then this is rather straightforward. Here, one can prove that if $\Gamma a$ is uniformly continuous then kinetic maximal $L^p_\mu(L^q)$-regularity of $-(-\Delta_v)^{\frac{\beta}{2}}$ transfers to the family of operators $(A(t))_{t \in [0,T]}$. If $a(t,x,v,w)=a(w)$, then under the assumption that $a \in L^\infty(\R^{n};(\lambda,K))$, where $0 <\lambda < K$, is symmetric maximal $L^p_\mu(L^p)$-regularity follows by the work in \cite{chen_lp-maximal_2018} and the considerations in Section \ref{sec:kinmaxlpmulq}. We plan to treat the general case in the near future. 
\end{remark}

\section{Quasilinear Kolmogorov equations}
\label{sec:absquasi}
Let us now consider quasilinear kinetic partial differential equations of the type 
\begin{equation} \label{eq:quasilinkin}
	\begin{cases}
		\partial_t u + v \cdot \nabla_x u = A(u)u + F(u), t>0 \\
		u(0) = u_0.
	\end{cases}
\end{equation}
Let $\beta \in (0,2]$, $s \in \R$, $p,q \in (1,\infty)$, $\mu \in (1/p,1]$, $T \in (0,\infty]$ and $D$ be a Banach space, densely embedded in $X_\beta^{s,q}$. We are interested mainly in $L^p_\mu(X_\beta^{s,q})$ solutions, i.e. functions 
\begin{equation*}
	u \in \BE_\mu(0,T) := \T^{p}_\mu((0,T);X_\beta^{s,q}) \cap L^p_\mu((0,T);D).
\end{equation*}
 We recall $X_{\gamma,\mu} = \tr(\BE_\mu(0,T))$ and assume $V_\mu \subset X_{\gamma,\mu}$ to be an open subset. Furthermore, suppose that $(A,F) \colon V_\mu \to \B(D;X_\beta^{s,q})\times X_\beta^{s,q}$ and $u_0 \in V_\mu$.

 We are going to develop a solution theory for equations of this type which is similar to the theory of ordinary differential equations. This ansatz has proven to be useful in the context of maximal $L^p$-regularity. The theorem presented in this section and its proof are mainly based on the corresponding maximal $L^p$-regularity theory as given in \cite{pruss_moving_2016}. 

\begin{theorem} \label{thm:quasishorttime}
	Let $\beta \in (0,2]$, $s \in \R$, $p,q \in (1,\infty)$ and $\mu \in (\frac{1}{p},1]$. Let
	\begin{equation*}
		(A,F) \in C^{1-}_{\mathrm{loc}}(V_\mu ; \B(D,X_\beta^{s,q})\times X_\beta^{s,q})
	\end{equation*} 
	and $u_0 \in V_\mu$ such that $A(u_0)$ satisfies the kinetic maximal $L^p_\mu(X_\beta^{s,q})$-regularity property. There exists a time $T = T(u_0)>0$ and a radius $\epsilon = \epsilon(u_0) >0 $ with $\bar{B}_\epsilon(u_0)= \bar{B}^{X_{\gamma,\mu}}_\epsilon(u_0) \subset V_\mu$ such that the Cauchy problem \eqref{eq:quasilinkin} admits a unique solution 
	\begin{equation*}
		u(\cdot;u_1) \in \BE_\mu(0,T) \cap C([0,T];V_\mu)
	\end{equation*}
	on $[0,T]$ for any initial value $u_1 \in \bar{B}_\epsilon(u_0)$. Furthermore, there exists a constant $C = C(u_0)$ such that for all $u_1,u_2 \in \bar{B}_\epsilon(u_0)$ we have
	\begin{equation*}
		\norm{u(\cdot;u_1)-u(\cdot;u_2)}_{\BE_\mu(0,T)} \le C \norm{u_1-u_2}_{X_{\gamma,\mu}},
	\end{equation*}
	i.e. the solutions depend continuously on the initial data. Finally, the solution regularizes instantaneously, i.e. for all $\delta  \in (0,T)$ we have $u \in \BE_\mu(\delta,T) \hookrightarrow C([\delta,T];X_{\gamma,1})$.
\end{theorem}

\begin{proof}
 	The proof follows closely the argument in \cite[Section II.5]{pruss_moving_2016} and only some small changes are needed in the kinetic setting. Let $u_0 \in V_\mu$, then, there exist $\epsilon_0>0$ and a Lipschitz constant $L>0$ such that $\bar{B}_{\epsilon_0}(u_0) \subset V_\mu$ and that the estimates
	\begin{equation*}
		\norm{A(w_1)v-A(w_2)v}_{X_\beta^{s,q}} \le L \norm{w_1-w_2}_{X_{\gamma,\mu}} \norm{v}_D
	\end{equation*}
	and
	\begin{equation*}
		\norm{F(w_1)-F(w_2)}_{X_\beta^{s,q}} \le L \norm{w_1-w_2}_{X_{\gamma,\mu}} 
	\end{equation*}
	hold for all $w_1,w_2 \in \bar{B}_{\epsilon_0}(u_0)$ and any $v \in D$. By assumption, the frozen operator $A(u_0)$ satisfies the kinetic maximal $L^p_\mu(X_\beta^{s,q})$-regularity property. In particular for all $T>0$ there exists a constant $C_1=C_1(T,u_0)>0$ such that 
	\begin{equation*}
		\norm{u}_{\BE_\mu(0,T)} \le C_1 \norm{\partial_t u + v \cdot \nabla_x u - A(u_0)u}_{p,\mu,X_\beta^{s,q}} + C_1 \norm{u(0)}_{X_{\gamma,\mu}}
	\end{equation*}
	for all $u \in \BE_\mu(0,T)$. Moreover, the trace Theorem \ref{thm:trace} together with a standard extension argument implies the existence of a constant $C_2>0$ independent of $T$ such that
	\begin{equation*}
		\norm{u}_{\infty,X_{\gamma,\mu}} \le C_2 \norm{u}_{{\BE}_{\mu}(0,T)}
	\end{equation*}
	for all $u \in {_0\BE}_{\mu}(0,T)$. 
	
	Let $T>0$ and $u_0^*  \in \BE_\mu(0,T)	$ be the unique solution of the linear problem $\partial_t w + v \cdot \nabla_x w = A(u_0) w$, $w(0) = u_0$. For $u_1 \in X_{\gamma,\mu}$ and $r \in (0,1]$ we define
	\begin{equation*}
		\BB_{r,T,u_1} = \{ w \in \BE_\mu(0,T) \colon w(0) = u_1 \text{ and } \norm{w-u_{0}^*}_{\BE_\mu(0,T)} \le r\}. 
	\end{equation*}
	\textbf{Claim 1:} Let $\epsilon \in (0,\epsilon_0]$ and $u_1 \in \bar{B}_{\epsilon}(u_0)$. Provided that $r,T,\epsilon>0$ are chosen sufficiently small, it follows that for all $w \in \BB_{r,T,u_1}$ we have $w(t) \in \bar{B}_{\epsilon_0}(u_0)$ for any time $t \in [0,T]$. \newline
	
	Let $u_1^*$ be the unique solution of the linear problem $\partial_t w +v \cdot \nabla_x w = A(u_0)w$ with initial value $w(0) = u_1$. Let $w \in \BB_{r,T,u_1}$, then 
	\begin{equation*}
		\norm{w-u_0}_{\infty,X_{\gamma,\mu}} \le \norm{w-u_1^*}_{\infty,X_{\gamma,\mu}} +  \norm{u_1^*-u_0^*}_{\infty,X_{\gamma,\mu}} +  \norm{u_0^*-u_0}_{\infty,X_{\gamma,\mu}}.
	\end{equation*} 
	By the continuity of $u_0^*$ in the trace space, there exists $T_0 >0$ such that $ \norm{u_0^*-u_0}_{\infty,X_{\gamma,\mu}} \le \epsilon_0/3$ for all $T \le T_0$. Furthermore, there exists a constant $C_\gamma = C_\gamma(T_0)$ such that 
	\begin{equation*}
		\norm{u}_{\infty,X_{\gamma,\mu}} + \norm{u}_{\BE_\mu(0,T)} \le C_\gamma \norm{u(0)}_{X_{\gamma,\mu}}
	\end{equation*}
	for all $u \in \BE_\mu(0,T)$ with $\partial_tu +v \cdot \nabla_x u = A(u_0)u$ and any $T \in (0,T_0]$.
	
	 Combining this information with the kinetic maximal $L^p_\mu(X_\beta^{s,q})$-regularity estimate, we deduce
	\begin{align} \label{eq:quasi1}
		\norm{w-u_0}_{\infty,X_{\gamma,\mu}} &\le C_2 \norm{w-u_1^*}_{\BE_\mu(0,T)} +C_\gamma\norm{u_1-u_0}_{X_{\gamma,\mu}} +  \norm{u_0^*-u_0}_{\infty,X_{\gamma,\mu}} \\ \nonumber
		&\le C_2( \norm{w-u_0^*}_{\BE_\mu(0,T)}+\norm{u_0^*-u_1^*}_{\BE_\mu(0,T)}) + C_\gamma \epsilon + \norm{u_0^*-u_0}_{\infty,X_{\gamma,\mu}} \\ \nonumber
		&\le C_2 r + C_\gamma(C_2+1) \epsilon + \norm{u_0^*-u_0}_{\infty,X_{\gamma,\mu}}.
	\end{align}
	Choosing $r \le r_0:= \epsilon_0/(3C_2) $, $T \le T_0$ and $\epsilon \le  \epsilon_1 = (3C_\gamma(C_2+1))^{-1} \epsilon_0 $ we deduce \linebreak $\norm{w(t)-u_0}_{X_{\gamma,\mu}} \le \epsilon_0$ for all $t \in [0,T]$, i.e. the validity of Claim 1. \newline
	
	From now on we assume $r \le r_0$, $T \in (0,T_0]$ and $\epsilon \le \epsilon_1(T)$. For $u_1 \in \bar{B}_\epsilon(u_0)$ we are going to show that the mapping 
	\begin{equation*}
		G_{u_1} \colon \BB_{r,T,u_1} \mapsto \BE_{\mu}(0,T), \; G_{u_1}(w) := u, 
	\end{equation*}
	where $u$ is the unique solution of the linear problem 
	\begin{equation*}
	\begin{cases}
		\partial_t u + v \cdot \nabla_x u = A(u_0 ) u + F(w) + (A(w)-A(u_0))w \\
		u(0) = u_1,
	\end{cases}
	\end{equation*}
	admits a unique fixed point under the assumption that $r,T,\epsilon>0$ are chosen sufficiently small.
	\newline
	
	\textbf{Claim 2:} Provided that the parameters $r,T,\epsilon>0$ are chosen sufficiently small we claim that $G_{u_1}$ is a self-mapping, i.e. $G_{u_1}(\BB_{r,T,u_1}) \subset \BB_{r,T,u_1}$. \newline
	
	Let $w \in \BB_{r,T,u_1}$, then $G_{u_1}(w) - u_0^*$ is a solution to 
	\begin{equation*}
		\begin{cases}
			\partial_t (G_{u_1}(w) - u_0^*) + v \cdot \nabla_x (G_{u_1}(w) - u_0^*) = A(u_0) (G_{u_1}(w) - u_0^*) + F(w) + (A(w)-A(u_0))w \\
			(G_{u_1}(w) - u_0^*)(0) = u_1-u_0,
		\end{cases}
	\end{equation*}
	whence
	\begin{equation*}
		\norm{G_{u_1}(w)-u_0^*}_{\BE_{\mu}(0,T)} \le C_1\norm{u_1-u_0}_{X_{\gamma,\mu}} + C_1\norm{F(w) + (A(w)-A(u_0))w}_{p,\mu,X_\beta^{s,q}}.
	\end{equation*}
	We have
	\begin{align*}
		\norm{[A(w)-A(u_0)]w}_{p,\mu,X_\beta^{s,q}} &\le L \norm{w-u_0}_{\infty,X_{\gamma,\mu}} \norm{w}_{\BE_\mu(0,T)} \\
		&\le L \norm{w-u_0}_{\infty,X_{\gamma,\mu}} (r+ \norm{u_0^*}_{\BE_\mu(0,T)})
	\end{align*}
	and
	\begin{align*}
		\norm{F(w)}_{p,\mu,X_\beta^{s,q}} &\le \norm{F(w)-F(u_0)}_{p,\mu,X_\beta^{s,q}} + \norm{F(u_0)}_{p,\mu,X_\beta^{s,q}} \\
		&\le C_3(T) (L \norm{w-u_0}_{\infty,X_{\gamma,\mu}} + \norm{F(u_0)}_{X_\beta^{s,q}}),
	\end{align*}
	where $C_3(T,p) = (1+(1-\mu)p)^{-1/p}T^{1/p+1-\mu}$. We recall the estimate in equation \eqref{eq:quasi1} 
	\begin{align*}
		\norm{w-u_0}_{\infty,X_{\gamma,\mu}} \le C_2r+C_\gamma(C_2+1)\epsilon+\norm{u_0^*-u_0}_{\infty,X_{\gamma,\mu}} \le \epsilon_0.
	\end{align*}
	 We deduce
	\begin{align*}
		\norm{G_{u_1}(w)-u_0^*}_{\BE_{\mu}(0,T)} &\le C_1\norm{u_1-u_0}_{X_{\gamma,\mu}} +L \norm{w-u_0}_{\infty,X_{\gamma,\mu}} \left(r+ \norm{u_0^*}_{\BE_\mu(0,T)}\right) \\
		&\hphantom{=}+C_3(T) (L \norm{w-u_0}_{\infty,X_{\gamma,\mu}} + \norm{F(u_0)}_{X_\beta^{s,q}}) \\
		&\le C_1\epsilon +L \left(C_2r+C_\gamma(C_2+1)\epsilon+\norm{u_0^*-u_0}_{\infty,X_{\gamma,\mu}} \right) \left(r+ \norm{u_0^*}_{\BE_\mu(0,T)} \right) \\
		&\hphantom{=}+C_3(T) (L \epsilon_0 + \norm{F(u_0)}_{X_\beta^{s,q}}). 
	\end{align*}
	Choosing first $r \in (0,r_0]$ such that $LC_2r \le 1/8$, then $T=T(r) \in (0,T_0]$ such that
	\begin{equation*}
		\norm{u_0^*-u_0}_{\infty,X_{\gamma,\mu}} \le (8L)^{-1}, \; C_3(T) \le (8L \epsilon_0 + 8\norm{F(u_0)}_{X_\beta^{s,q}})^{-1} \text{ and } \norm{u_0^*}_{\BE_\mu(0,T)} \le r
	\end{equation*}
	and finally $\epsilon = \epsilon(r,T) \in (0,\epsilon_1]$ such that $ \epsilon \le (8C_1(T))^{-1}r$ and $\epsilon \le (8LC_\gamma(C_2+1))^{-1}$ we conclude
	\begin{equation*}
		\norm{G_{u_1}(w)-u_0^*}_{\BE_{\mu}(0,T)} \le r,
	\end{equation*}
	hence $G_{u_1}(w) \in \BB_{r,T,u}$. This shows Claim 2. \newline  
	
	\textbf{Claim 3:} We can choose $r,T,\epsilon >0$ small enough such that for all $u_1,u_2 \in \bar{B}_\epsilon(u_0)$ and any $w_1 \in \BB_{r,T,u_1}$, $w_2 \in \BB_{r,T,u_2}$ we have
	\begin{equation*}
		\norm{G_{u_1}(w_1)-G_{u_2}(w_2)}_{\BE_{\mu}(0,T)} \le \frac{1}{2} \norm{w_1-w_2}_{\BE_\mu(0,T)} + c \norm{u_1-u_2}_{X_{\gamma,\mu}}.
	\end{equation*}\newline
	
	The function $z:= G_{u_1}(w_1)-G_{u_2}(w_2)$ solves the linear problem 
	\begin{equation*}
		\begin{cases}
			\partial_t z + v \cdot \nabla_x z = &A(u_0)z + (F(w_1)-F(w_2)) \\
			&+ (A(w_1)-A(u_0))(w_1-w_2)+(A(w_1)-A(w_2))w_2 \\
			z(0) = u_1-u_2
		\end{cases}
	\end{equation*}
	and hence
	\begin{align*}
		\norm{G_{u_1}(w_1)-G_{u_2}(w_2)}_{\BE_{\mu}(0,T)} &\le C_1 \norm{u_1-u_2}_{X_{\gamma,\mu}} + C_1\norm{F(w_1)-F(w_2)}_{p,\mu,X_\beta^{s,q}} \\
		&+ C_1\norm{(A(w_1)-A(u_0))(w_1-w_2)}_{p,\mu,X_\beta^{s,q}} \\
		&+ C_1\norm{(A(w_1)-A(w_2))w_2}_{p,\mu,X_\beta^{s,q}}.
	\end{align*}
	The estimate
	\begin{equation*}
		\norm{F(w_1)-F(w_2)}_{p,\mu,X_\beta^{s,q}} \le C_3(T)L \norm{w_1-w_2}_{\infty,X_{\gamma,\mu}}
	\end{equation*}
	combined with 
	\begin{align*}
		\norm{w_1-w_2}_{\infty,X_{\gamma,\mu}} &\le \norm{w_1-w_2-(u_1^*-u_2^*)}_{\infty,X_{\gamma,\mu}} +\norm{u_1^*-u_2^*}_{\infty,X_{\gamma,\mu}} \\
		&\le C_2\norm{w_1-w_2-(u_1^*-u_2^*)}_{\BE_\mu(0,T)} + C_\gamma\norm{u_1-u_2}_{X_{\gamma,\mu}} \\
		&\le C_2\norm{w_1-w_2}_{\BE_\mu(0,T)} + 2C_\gamma\norm{u_1-u_2}_{X_{\gamma,\mu}}
	\end{align*}
	yields
	\begin{equation*}
		\norm{F(w_1)-F(w_2)}_{p,\mu,X_\beta^{s,q}} \le C_3(T)L \left(C_2\norm{w_1-w_2}_{\BE_\mu(0,T)} + 2C_\gamma\norm{u_1-u_2}_{X_{\gamma,\mu}} \right).
	\end{equation*}

	An application of the local Lipschitz estimate for $A$ together with Claim 1 and the estimate in equation \eqref{eq:quasi1} results in 
	\begin{align*}
		&\norm{(A(w_1)-A(u_0))(w_1-w_2)}_{p,\mu,X_\beta^{s,q}} + \norm{(A(w_1)-A(w_2))w_2}_{p,\mu,X_\beta^{s,q}} \\
		&\le L(\norm{w_1-u_0}_{\infty,X_{\gamma,\mu}} \norm{w_1-w_2}_{\BE_\mu(0,T)} + \norm{w_1-w_2}_{\infty, X_{\gamma,\mu}}\norm{w_2}_{\BE_\mu(0,T)}) \\
		&\le (LC_2 r + LC_\gamma(C_2+1) \epsilon+L\norm{u_0^*-u_0}_{\infty,X_{\gamma,\mu}}+C_2Lr+C_2L\norm{u_0^*}_{\BE_\mu(0,T)})\norm{w_1-w_2}_{\BE_\mu(0,T)} \\
		&\hphantom{=}+2LC_\gamma(r+\norm{u_0^*}_{\BE_\mu(0,T)})\norm{u_1-u_2}_{X_{\gamma,\mu}} ,
	\end{align*}
	where we have estimated
	\begin{equation*}
		\norm{w_2}_{\BE_\mu(0,T)} \le \norm{w_2-u_0^*}_{\BE_\mu(0,T)} + \norm{u_0^*}_{\BE_\mu(0,T)} \le r + \norm{u_0^*}_{\BE_\mu(0,T)}.
	\end{equation*}
	Let us now reduce $r>0$ even more, so that $r\le (12C_2L)^{-1}$, let $T >0$ be so small that $C_3(T) \le (12LC_2)^{-1}$, $ \norm{u_0^*}_{\BE_\mu(0,T)} \le (12C_2L)^{-1}$, $ \norm{u_0^*-u_0}_{\infty,X_{\gamma,\mu}} \le (12L)^{-1}$ and choose $\epsilon >0$ with $\epsilon \le (12LC_\gamma(C_2+1))^{-1}$, then for this particular choice of $r,T,\epsilon$ we deduce that
	\begin{equation} \label{eq:gencontquasil}
		\norm{G_{u_1}w_1-G_{u_2}w_2}_{\BE_{\mu}(0,T)} \le \frac{1}{2} \norm{w_1-w_2}_{\BE_\mu(0,T)}+ (C_\gamma/(2C_2)+C_1) \norm{u_1-u_2}_{X_{\gamma,\mu}}.
	\end{equation}
	This completes the proof of Claim 3. We note that we can choose $r,T,\epsilon>0$ such that the conditions imposed on $r,T,\epsilon$ in Claim 1, Claim 2 and Claim 3 hold true simultaneously. This is due to the fact that we always choose them in the same order with the same dependence.  \newline 
	
	 Choosing $u_2 = u_1$ in the statement of Claim 3 it follows that the mapping $G_{u_1}$ is a contractive self-mapping of closed subsets of a Banach space. Applying Banach's fixed point theorem we obtain a unique fixed point $u \in \BB_{r,T,u_1}$ of $G_{u_1}$, i.e. $G_{u_1}(u) = u$. In particular, it follows that $u$ is the unique local solution of equation \eqref{eq:quasilinkin}. The second assertion of the theorem follows by applying the estimate in equation \eqref{eq:gencontquasil} to the solutions $u(\cdot;u_1),u(\cdot;u_2)$ with respective initial values $u_1,u_2 \in \bar{B}_{\epsilon}(u_0)$. The instantaneous regularization has already been proven to hold in Section \ref{sec:kinmaxreg}. 
\end{proof}

\begin{lemma} \label{lem:maxivexis}
	Under the assumptions of Theorem \ref{thm:quasishorttime} additionally suppose that for all $w \in V_\mu$ the operator $A(w)$ satisfies the kinetic maximal $L^p_\mu(X_\beta^{s,q})$-regularity property. Then, the solution $u(t)$ of equation \eqref{eq:quasilinkin} has a maximal interval of existence $[0,t^+(u_0))$, which is characterized by either one of the following conditions:
	\begin{enumerate}
		\item global existence, i.e. $t^+(u_0) = \infty$
		\item $t^+(u_0)<\infty$, $\liminf\limits_{t \to t^+(u_0)} \mathrm{dist}_{X_{\gamma,\mu}}(u(t),\partial V_\mu) = 0$
		\item $t^+(u_0)<\infty$, $\liminf\limits_{t \to t^+(u_0)} \mathrm{dist}_{X_{\gamma,\mu}}(u(t),\partial V_\mu) > 0$ and $\lim\limits_{t \to t^+(u_0)} u(t)$ does not exist in $X_{\gamma,\mu}$.
	\end{enumerate} 
\end{lemma}

\begin{proof}
	For $u_0 \in V_\mu$ we define
	\begin{equation*}
		t^+(u_0) = \sup \{ \delta>0 \colon \eqref{eq:quasilinkin} \text{ admits a unique solution on } [0,\delta ] \}.
	\end{equation*}
	We suppose that $t^+(u_0) < \infty$, $\mathrm{dist}_{X_{\gamma,\mu}}(u(t),\partial V_\mu)                                                                                                                                                                                                                                                              \ge \eta$ for some $\eta >0$ and all $t \in [0,t^+(u_0))$ and assume that $u(t)$ converges to $u_1 := u(t^+(u_0)) \in V_\mu$ as $t \to t^{+}(u_0)$. The set $u([0,t^+(u_0)]) \subset V_\mu$ is compact in $X_{\gamma, \mu}$, whence by Theorem \ref{thm:quasishorttime} and a compactness argument we find a $\delta >0$ such that for all $s \in [0,t^+(u_0)]$ the problem $\partial_t w + v \cdot \nabla_x w = A(w)w+F(w)$, $w(0) = u(s)$ admits a unique solution $w \in \BE_\mu(0,\delta)$. For $s_0 = t^+(u_0)-\delta/2 $ the corresponding solution coincides with $u$ on $[s_0,t^+(u_0)]$ and extends $u$ beyond $t^+(u_0)$, a contradiction. This shows that only one of the conditions (1)-(3) can hold as they are mutually exclusive. 
\end{proof}

\begin{coro} \label{cor:critglobal}
	Under the assumptions of Theorem \ref{thm:quasishorttime} we suppose that $V_\mu = X_{\gamma,\mu}$ and that for all $w \in V_\mu$ the operator $A(w)$ satisfies the kinetic maximal $L^p_\mu(X_\beta^{s,q})$-regularity property. Assume that $ \norm{u}_{\BE_\mu(0,t^+(u_0))}< \infty$, then $t^+(u_0) = \infty$.
\end{coro}

\begin{proof}
	Suppose that $t^{+}(u_0) < \infty$. The local solution $u \in \BE_\mu(0,t^+(u_0))$ satisfies
	\begin{equation*}
	\begin{cases}
		\partial_t u +v \cdot \nabla_x u = A(u_0)u + f \\
		u(0) = u_0,
	\end{cases}
	\end{equation*}
	where $f(t) = (A(u(t))-A(u_0))u(t) + F(u(t))$. Due to the assumption on $u$, we deduce $f \in L^p_\mu((0,t^+(u_0));X_\beta^{s,q})$. By the kinetic maximal $L^p_\mu(X_\beta^{s,q})$-regularity property, we conclude $u \in C([0,t^+(u_0)],X_{\gamma,\mu})$, whence Lemma \ref{lem:maxivexis} implies $t^+(u_0) = \infty$.  
\end{proof}

\section{Local existence of strong $L^p_\mu(L^q)$-solutions to a quasilinear kinetic diffusion equation}
\label{sec:quasilindiff}

	In this section we are going to study local existence of strong $L^p_\mu(L^q)$ solutions of a quasilinear kinetic diffusion equation. Let $p,q \in (1,\infty)$, $\mu \in (\frac{1}{p},1]$ and let $\kappa \in C^2_b(\R;\mathrm{Sym}(n))$ such that $\langle \kappa(z) \xi,\xi \rangle \ge \lambda \abs{\xi}^2$ for all $z \in \R, \xi \in \R^n$ and a constant $\lambda>0$. We are interested in the existence of local $L^p_\mu(L^q)$-solutions of the quasilinear kinetic diffusion problem
	\begin{equation} \label{eq:quasilinkol}
		\begin{cases}
			\partial_t u + v \cdot \nabla_x u = \nabla_v \cdot (\kappa(u) \nabla_v u)  \\
			u(0) = u_0.
		\end{cases}
	\end{equation}
	The right space for strong $L^p_\mu(L^q)$-solutions of equation \eqref{eq:quasilinkol} is obviously 
	$$\BE_\mu(0,T) = \T^{p}_\mu((0,T);L^q(\R^{2n})) \cap L^p_\mu((0,T);H^{2,q}_v(\R^{2n})).$$
	We choose $V_\mu = X_{\gamma,\mu} = {^{\mathrm{kin}}B}_{qp}^{\mu-1/p,2}(\R^{2n}) $ and consider the operators
	\begin{equation*}
		A \colon X_{\gamma,\mu} \mapsto \B(H^{2,q}_v(\R^{2n});L^q(\R^{2n})), \; A(u)w = \kappa(u) \colon \nabla_v^2 w
	\end{equation*}
	and $F \colon X_{\gamma,\mu} \to L^q(\R^{2n})$ given by 
	\begin{equation*}
		F(u) = \langle \kappa'(u) \nabla_v u , \nabla_v u \rangle, 
	\end{equation*}
	which shows that equation \eqref{eq:quasilinkol} is of the form of equations considered in Section \ref{sec:absquasi}. We define the spaces
	\begin{equation*}
		{C}_0(\R^{2n}) := \{ u \in C(\R^{2n}) \colon \forall \epsilon>0 \colon \exists K \subset \R^{2n} \text{ compact } \colon \abs{u(x)} \le \epsilon \; \forall x \in K^c \}
	\end{equation*}
	and
	\begin{equation*}
		{_vC}_0^1(\R^{2n}) := \{ u \in C^1_v(\R^{2n}) \colon u, \partial_{v_1}u, \dots, \partial_{v_n} u \in C_0(\R^{2n}) \}.
	\end{equation*}

	 To assure, that $A$ and $F$ are well-defined and locally Lipschitz continuous we furthermore assume that
	\begin{equation*}
		\mu-\frac{1}{p}>\frac{1}{2}+\frac{2n}{q} >\frac{2n}{q}>0
	\end{equation*}
 	Under this assumption the following embedding
	\begin{equation*}
		X_{\gamma,\mu} = {^{\mathrm{kin}}B}_{qp}^{\mu-1/p,2}(\R^{2n}) \hookrightarrow {_vC}_0^1(\R^{2n})
	\end{equation*}
	holds continuously by Lemma \ref{lem:contembed} and Corollary \ref{cor:gradvembed}. Moreover, we have the embedding $X_{\gamma,\mu} \hookrightarrow H^{1,q}_v(\R^{2n})$ by Lemma \ref{lem:besovLq} and Lemma \ref{lem:laplacianbesov}.

	 Let us first check that $A$ is well-defined and locally Lipschitz continuous. As $u \in X_{\gamma,\mu} \subset C_0(\R^{2n})$ the image of $u$ is the subset of a compact set. Hence, $\kappa(u)$ is bounded as a consequence of the continuity of $\kappa$. This shows that $A$ is well-defined. Let $u_0 \in X_{\gamma,\mu}$ and $\epsilon>0$, then there exists a compact set $K = K(\epsilon,u_0) \subset \R$ such that $u(x) \in K$ for any $x \in \R^{2n}$ and all $u \in B^{X_{\gamma,\mu}}_\epsilon(u_0)$. In particular $\abs{\kappa'(u(x))} \le C(u_0,\epsilon)$ for all $u \in B^{X_{\gamma,\mu}}_\epsilon(u_0)$ and all $x \in \R^{2n}$. We deduce
	\begin{align*}
		\norm{A(u)w-A(v)w}_q &\le \norm{\kappa(u)-\kappa(v)}_\infty \norm{\nabla_v^2 w}_q \le C(u_0,\epsilon) \norm{u-v}_\infty \norm{w}_{H^{2,q}_v} \\
		&\le C \norm{u-v}_{X_{\gamma,\mu}} \norm{w}_{H^{2,q}_v}  
	\end{align*}
	as a consequence of the mean value theorem. Next, we consider the mapping $F$. We have
	\begin{equation*}
		\norm{F(u)}_q \le \norm{\kappa'(u)}_\infty \norm{\nabla_v u}_\infty \norm{\nabla_v u}_q  \le C(\kappa')  \norm{u}_{X_{\gamma,\mu}}^2.
	\end{equation*}
	and 
	\begin{align*}
		\norm{F(u)-F(w)}_q &\le \norm{\kappa'(u)-\kappa'(w)}_q \norm{\nabla_v u}_\infty^2 + \norm{\kappa'(w)}_\infty \norm{\nabla_v u+\nabla_v w}_\infty \norm{\nabla_v u-\nabla_v w}_q \\
		&\le C(u_0,\epsilon,\kappa'') \norm{u-w}_{q} + C(u_0,\epsilon,\kappa') \norm{\nabla_v u-\nabla_v w}_q \\
		&\le C(u_0,\epsilon,\kappa',\kappa'') \norm{u-w}_{X_{\gamma,\mu}}
	\end{align*} 
	for all $u,w \in B^{X_{\gamma,\mu}}_\epsilon(u_0)$.
	If $u_0 \in {^{\mathrm{kin}}B}_{qp}^{\mu-1/p,2}(\R^{2n})$, then it follows $\kappa(u_0(x+tv,v)) \in BUC([0,T] \times \R^{2n})$ by Lemma \ref{lem:C0gamabuc} as a consequence of the fact that $u_0 \in C_0(\R^{2n})$ together with the uniform continuity of $\kappa$. Consequently, by Theorem \ref{thm:kinmaxLpmuLqvarcoeff} the frozen operator $A(u_0)$ admits kinetic maximal $L^p_\mu(L^q)$-regularity. We conclude by Theorem \ref{thm:quasishorttime} that there exists a $T = T(u_0)$ such that the Cauchy problem \eqref{eq:quasilinkol} admits a unique solution $u \in \BE_\mu(0,T)$. Moreover, we have $u \in BUC([0,T] \times \R^{2n})$. 
	
\appendix
\section{Appendix}

\begin{lemma} \label{lem:genYoung}
	Let $f,g \colon \R^{2n} \to \R$ be measurable functions such that $f \in L^1(\R^{2n})$ and $g \in L^q(\R^{2n})$ for some $q \in [1,\infty)$. Given any $t \ge 0$, we define
	\begin{align*}
		(f \ast_t g)(x,v) &= \int_{\R^{2n}} f(x-\tilde{x}-t \tilde{v},v-\tilde{v}) g(\tilde{x},\tilde{v}) \dx \tilde{v} \dx \tilde{x} \\ 
		&= \int_{\R^{2n}} f(\tilde{x},\tilde{v}) g(x-\tilde{x}-t(v- \tilde{v}),v-\tilde{v})  \dx \tilde{v} \dx \tilde{x}.
	\end{align*}
	We have
	\begin{equation*}
		\norm{f \ast_t g}_{q} \le \norm{f}_{1} \norm{g}_{q}.
	\end{equation*}
\end{lemma}

\begin{proof}
	From $ (f \ast_t g)(x,v) = \Gamma(-t)g \ast f $	we deduce the result by the well-known Young integral inequality and the isometry property of $\Gamma(-t)$ in $L^q(\R^{2n})$.
\end{proof}

\begin{remark}
	We note that the last statement is nontrivial due to the fact, that $\ast_t$ is a noncommutative convolution. In other words the underlying group is noncommutative. The fact that the Young inequality continues to hold is a consequence of the translation group $(\Gamma(-t))_{t \in \R}$ being an isometry in $L^q(\R^{2n})$ for all $p \in [1,\infty)$.
\end{remark}

\begin{theorem} \label{thm:kinmkihlin} 
	Let $n_1,n_2 \in \N$ and $m \colon \R^{n_1+n_2} \to \R$ be a smooth function in $\{ (z,\xi) \in \R^{n_1+n_2} \colon z =0 \text{ or } \xi = 0 \}^c$, such that for all multi-indices $\alpha \in \N^{n_1}_0, \beta \in \N^{n_2}_0$ with $\abs{\alpha}+ \abs{\beta} \le  \left[ \frac{n_1+n_2}{2} \right]+1$ we have
	\begin{equation*}
		\abs{z}^{\abs{\alpha}}\abs{\xi}^{\abs{\beta}} \abs{\partial_z^\alpha \partial_\xi^\beta m(z,\xi)} \le C(\alpha,\beta).
	\end{equation*}
	Then $m$ defines a bounded Fourier multiplier on $L^{q_1}(\R^{n_1};L^{q_2}(\R^{n_2}))$ for all $q_1,q_2 \in (1,\infty) $ and the norm of $m$ depends linearly on the constants $C(\alpha,\beta)$.  
\end{theorem}

\begin{proof}
	This is a consequence of \cite[Theorem 7]{antonic_hormandermihlin_2016}, noting that above condition implies the H\"ormander condition.
\end{proof}

The following lemma is going to be very useful when dealing with the $k$-variable in multiplier estimates. It says that it suffices to consider always the multiplier where the dilation in $k$ has been neglected. 

\begin{lemma} \label{lem:dilkinmik} 
	Let $\alpha_1,\alpha_2 \in \R$. If $m \in L^\infty(\R^{2n})$ is an $L^q(\R^{2n})$-Fourier multiplier, then for all $t \neq 0$ the dilated function $m(t^{\alpha_1}\cdot,t^{\alpha_2}\cdot)$ is an $L^q(\R^{2n})$-Fourier multiplier with constant independent of $t$.
\end{lemma}

\begin{proof}
	This follows immediately by coordinate transformation.
\end{proof}

\begin{lemma} \label{lem:gammamult}
	If $m \in L^\infty(\R^{2n})$ is an $L^q(\R^{2n})$-Fourier multiplier, then for all $t \in \R$ the function $ (k,\xi) \mapsto m(k,\xi-tk) $ is an $L^q(\R^{2n})$-Fourier multiplier with constant independent of $t$.
\end{lemma}

\begin{proof}
	This follows from the isometry property of $u \mapsto  u(k,\xi-tk)$.
\end{proof}

We consider the Cauchy problem
	\begin{equation} \label{eq:helppde}
		\begin{cases}
			\partial_t u = -c(-\Delta_v)^{\beta/2} u -ct^\beta(-\Delta_x)^{\beta/2} u \\
			u(0) = g,
		\end{cases}
	\end{equation}
	where $c>0$ is some constant.

\begin{prop} \label{prop:helpprob}
	For all $\beta \in (0,2]$, $s \ge 0$, $p,q \in (1,\infty)$, $\mu \in (\frac{1}{p},1]$, $T \in (0,\infty]$ and any function $g \in {^{\mathrm{kin}}B}_{qp}^{s+\mu-1/p,\beta}(\R^{2n})$ the solution $u$ of the Cauchy problem \eqref{eq:helppde} given by
	\begin{equation*}
		\hat{u}(t,k,\xi) = \exp(-c\abs{\xi}^\beta t-c\abs{k}^\beta t^{\beta+1})\hat{g}(k,\xi)=:\eta_\beta(t,k,\xi)\hat{g}(k,\xi).
	\end{equation*}
	satisfies $(-\Delta_v)^{\beta(s+1)/2} u +t^{\beta(s+1)}(-\Delta_x)^{\beta(s+1)/2} u \in L^p_\mu((0,T);L^q(\R^{2n}))$.
\end{prop}

For the proof of this result we need another characterisation of anisotropic Besov spaces as stated in \cite[Chapter 5]{triebel_theory_2006}. Let $n \in \N$, $p,q \in (1,\infty)$, $\alpha \in (0,\infty)^{2n} $ and $\sigma \in \R$. Let $\varphi_0^\alpha \in \S$ with $\phi_0^\alpha(z) = 1$ for all $z \in \R^{2n}$ such that $\abs{z}_\infty \le 1$ and $\phi_0^\alpha(z) = 0$ for all $z \in \R^{2n}$ satisfying $\abs{2^{-\alpha} z}_\infty \ge 1$.  For $j \in \N$ we define 
\begin{equation*}
	\varphi_j^\alpha(z) = \varphi_0^\alpha(2^{-j\alpha}z)- \varphi_0^\alpha({2^{-(j-1)\alpha}}z)
\end{equation*}
for all $z \in \R^{2n}$. We have
\begin{equation*}
	\sum_{j = 0}^\infty \varphi_j^\alpha(z) = 1
\end{equation*}
for all $z \in \R^{2n}$ and we have $\supp \varphi_0^\alpha \subset R_1^\alpha$ as well as $\supp \varphi_j^\alpha \subset R_{j+1}^\alpha \setminus R_{j-1}^\alpha$, where $R_j^\alpha = \{ z\in \R^{2n} \colon \abs{x_{l}}_\infty \le 2^{\alpha_l j}, \; l = 1,\dots,2n \} $. We call $ \varphi^\alpha $ an anisotropic resolution of unity. The map
\begin{equation*}
	\norm{f | B_{qp}^{\sigma,\alpha}(\R^{2n})} := \norm{f | B_{qp}^{\sigma,\alpha}(\R^{2n})}_{\varphi^\alpha} :=  \norm{\mathcal{F}^{-1}(\varphi_0^\alpha \hat{f})}_{q}+\left( \sum_{j = 1}^\infty 2^{j\sigma p} \norm{\mathcal{F}^{-1}(\varphi_j^\alpha \hat{f})}_{q}^p \right)^\frac{1}{p}
\end{equation*}
defines an equivalent norm on $B_{qp}^{\sigma ,\alpha}(\R^{2n})$. It can be shown that this norm is independent of the anisotropic resolution of unity.

\begin{proof}[Proof of Proposition \ref{prop:helpprob}]
	It suffices to consider the case $T = \infty$. We are going to use the Littlewood-Paley decomposition, the above definition of Besov spaces and Mikhlin's multiplier theorem to prove the proposition. We abbreviate $w = (-\Delta_v)^{\beta(s+1)/2} u +t^{\beta(s+1)}(-\Delta_x)^{\beta(s+1)/2} u $, i.e. $\hat{w} = (\abs{\xi}^{\beta(s+1)} + t^{\beta(s+1)} \abs{k}^{\beta(s+1)}) \hat{u}$, write $a = s+\mu-1/p$ and $\sigma  = \frac{2\beta}{\beta+2}a$.
	We set
	\begin{align*}
		  w_j&:= {\F^{-1}( \hat{g}(k,\xi)\varphi_j^\alpha(k,\xi)  (\abs{\xi}^{\beta(s+1)} + t^{\beta(s+1)} \abs{k}^{\beta(s+1)}) \eta_\beta(t,k,\xi))}  \\
		&= {\F^{-1}( \hat{g}(k,\xi)\varphi_j^\alpha(k,\xi)  (\abs{\xi}^{\beta(s+1)} + t^{\beta(s+1)} \abs{k}^{\beta(s+1)}) \eta_\beta(t,k,\xi) \psi^{(j)} (k,\xi))}   \\
		&= : {\F^{-1}( \hat{g}(k,\xi)\varphi_j^\alpha(k,\xi) m^{(j)}_0(t,k,\xi))},
	\end{align*}
	where $\psi^{(j)}$ is a cutoff function satisfying the assumptions in Corollary \ref{cor:m0mult} corresponding to $j \in \N_0$. For example we might choose $\psi^{(j)} = \varphi_{j-1}+\varphi_j+\varphi_{j+1}$. By Corollary \ref{cor:m0mult} the function $m^{(j)}_0$ defines an $L^q(\R^{2n})$-Fourier multiplier and its norm can be estimated by
	\begin{equation} \label{eq:locmult}
		\norm{m^{(j)}_0(t,\cdot,\cdot)} \lesssim \left( t^{-(s+1)} \land 2^{\frac{2\beta}{\beta+2}j(s+1)} \right).	
	\end{equation}
	 We introduce the functions $g_j = \F^{-1}(\varphi_j^\alpha \hat{g})$ and note that
	\begin{equation*}
		w = \sum_{j = 0}^\infty \F^{-1}(\varphi_j^\alpha(k,\xi)\hat{g}(k,\xi)  m^{(j)}_0(t,k,\xi)) = \sum_{j = 0}^\infty w_j(t).
	\end{equation*}
	Writing $o = 2^{\frac{2\beta}{\beta+2}}$ the considerations thus far show
\begin{align*}
	&\int_0^\infty t^{p-\mu p}\left( \int_{\R^{2n}} \abs{w(t,x,v)}^q \dx (x,v) \right)^{\frac{p}{q}} \dx t  = \int_0^\infty t^{p-\mu p}  \norm{ \sum_{j = 0}^\infty w_j(t)}_q^p \dx t \\
	&\le \int_0^\infty t^{p-\mu p} \left( \sum_{j = 0}^\infty \norm{w_j(t)}_q \right)^p \dx t \lesssim \int_0^\infty t^{p-\mu p} \left( \sum_{j = 0}^\infty \left( t^{-(s+1)} \land o^{j(s+1)} \right)
 \norm{g_j}_q \right)^p \dx t \\
 	&\le \int_{0}^{1} t^{p-\mu p} \left( \sum_{j = 0}^\infty \left( t^{-(s+1)} \land o^{j(s+1)} \right)
 \norm{g_j}_q \right)^p \dx t +\int_{1}^\infty t^{p-\mu p} \left( \sum_{j = 0}^\infty t^{-(s+1)} \norm{g_j}_q \right)^p \dx t \\
 	&\lesssim \int_{0}^{1} t^{p-\mu p} \left( \sum_{to^{j} \le 1}  o^{j(s+1)}\norm{g_j}_q \right)^p \dx t +\int_{0}^{1} t^{p-\mu p} \left( \sum_{to^{j} > 1} t^{-(s+1)}
 \norm{g_j}_q \right)^p \dx t + I_3 \\
  	&=: I_1 +I_2 + I_3.
\end{align*}
We estimate the first term as follows
\begin{align*}
	I_1 &= \int_{0}^{1} t^{p-\mu p} \left( \sum_{to^{j} \le 1}  o^{j(s+1)}\norm{g_j}_q \right)^p \dx t \\
	&\le \int_{0}^{1} t^{p-\mu p} \left( \sum_{to^{j} \le 1}  o^{-rjp'}\right)^{\frac{p}{p'}} \left( \sum_{to^{j} \le 1}  o^{jp(s+1)+rjp}\norm{g_j}_q^p \right) \dx t \\
	&\lesssim \int_{0}^{1} t^{p-\mu p+rp} \left( \sum_{to^{j} \le 1}  o^{jp(s+1)+rjp}\norm{g_j}_q^p \right) \dx t \\
	&= \sum_{j = 1}^\infty  o^{jp(s+1)+rjp}\norm{g_j}_q^p \int_{0}^{o^{-j}} t^{p-\mu p+rp}     \dx t \lesssim \sum_{j = 1}^\infty  2^{\frac{2\beta}{\beta+2}ajp}\norm{g_j}_q^p = \norm{g |{^{\mathrm{kin}}B}_{qp}^{a,\beta}(\R^{2n})}^p
	\end{align*}
for some $r \in (\mu-1/p-1,\infty)$. In the first inequality we have used the H\"older inequality with dual exponent $p' = \frac{p}{p-1}$ and that the first resulting sum can be estimated by using the formula for the partial sum of the geometric series. In the next equation we used Fubini's theorem, noting that the integral is finite due to the restriction on $r$. 

Let us now estimate the term $I_2$. Let $r \in (-1,a-1)$,then 
\begin{align*}
	I_2 &= \int_{0}^{1} t^{p-\mu p} \left( \sum_{to^{j} > 1} t^{-(s+1)}
 \norm{g_j}_q \right)^p \dx t \\
 &\lesssim \int_{0}^{1} t^{-ap-1} \left( \sum_{to^{j} > 1} o^{-jp'\left(1+r \right)}  \right)^{\frac{p}{p'}} \left( \sum_{to^{j} > 1} o^{jp\left(1+r \right)} \norm{g_j}_q^p \right) \dx t \\
 &\lesssim \int_{0}^{1} t^{-ap-1} t^{\left(1+r \right)p} \left( \sum_{to^{j} > 1} o^{jp\left(1+r \right)} \norm{g_j}_q^p \right) \dx t \\ 
 &\lesssim \sum_{j = 1}^\infty o^{jp\left(1+r \right)}  \norm{g_j}_q^p \int_{o^{-j} }^{1} t^{-ap-1+p+pr}\dx t \\
 &\lesssim \sum_{j = 1}^\infty o^{ajp} \norm{g_j}_q^p = \norm{g |{^{\mathrm{kin}}B}_{qp}^{a,\beta}(\R^{2n})}^p.
\end{align*}
Here, we have once again used the H\"older inequality for the sum and Fubini's theorem. Moreover we have used the formula for the truncated geometric series.
The last estimate is a direct calculation of the $\dx t$-integral. 

The third term can be estimated as 
 \begin{align*}
 		I_3 \lesssim \left( \sum_{j = 0}^\infty
 \norm{g_j}_q \right)^p \le \left( \sum_{j = 0}^\infty o^{-ajp'} \right)^{\frac{p}{p'}} \sum_{j = 0}^\infty o^{ajp} \norm{g_j}_q^p \lesssim \sum_{j = 0}^\infty o^{ajp} \norm{g_j}_q^p = \norm{g |{^{\mathrm{kin}}B}_{qp}^{a,\beta}(\R^{2n})}^p,
 \end{align*}
since $t^{-\mu p-sp}$ is integrable on $(1,\infty)$. We conclude
\begin{equation*}
	\int_0^\infty t^{p-\mu p} \left( \int_{\R^{2n}} \abs{w(t,x,v)}^q \dx (x,v) \right)^\frac{p}{q} \dx t \lesssim \norm{g |{^{\mathrm{kin}}B}_{qp}^{a,\beta}(\R^{2n})}^p.
\end{equation*}
This concludes the proof of the proposition. The technique was inspired by the work in \cite{kim_heat_2017}.
\end{proof}

It remains to show the localized estimates for the norm of the multiplier 
$$(k,\xi) \mapsto \left(\abs{k}^{\beta(s+1)} +\abs{\xi}^{\beta(s+1)} \right) \exp(-\abs{k}^\beta t-\abs{\xi}^\beta t)$$
 with respect to a dilated cuboid. We note that it suffices to consider the multiplicator
 \begin{equation*}
 	m_1^{(j)}(t,k,\xi) = (\abs{\xi}^2 + \abs{k}^2)^{{\beta(s+1)/2}} \exp(- \gamma(\abs{\xi}^2 + \abs{k}^2)^{{\beta(s+1)/2}} ) \tilde{\psi}^{(j)} (k,\xi),
 \end{equation*}
 where $\tilde{\psi}^{(j)} (k,\xi) = t^{-(s+1)}{\psi}^{(j)} (t^{-\frac{\beta+1}{\beta}}k,t^{-\beta}\xi)$ and $c>0$ is a small constant independent of $t$. To see this, we have used that by Lemma \ref{lem:dilkinmik} a dilation of the multiplier does neither change the multiplier property nor the multiplier constant. Furthermore, we use that the functions
 \begin{equation*}
 	(k,\xi) \mapsto \frac{\abs{k}^{\beta(s+1)} +\abs{\xi}^{\beta(s+1)}}{(\abs{\xi}^2 + \abs{k}^2)^{{\beta(s+1)/2}}}
 \end{equation*}
 and
 \begin{equation*}
 		(k,\xi) \mapsto \exp(\gamma(\abs{\xi}^2 + \abs{k}^2)^{{\beta(s+1)/2}}-\abs{k}^\beta t-\abs{\xi}^\beta t)
 \end{equation*}
with $\gamma>0$ small enough define bounded $L^q(\R^{2n})$-Fourier multipliers. The next proposition deals with multipliers of the form $\R^{2n} \to \R, z \mapsto \abs{z}^{\beta(s+1)}\exp(-\abs{z}^\beta) \psi(z)$, where $\psi$ is the cutoff function corresponding to a cuboid. Afterwards, we apply this result to the multiplicator $m_1^{(j)}$ to deduce the desired estimate for $m_0^{(j)}$. The next two results are inspired by \cite[Lemma 5.1, Lemma 5.2, Lemma 5.3]{kemppainen_decay_2016}.

\begin{lemma} \label{lem:basismult}
	Let $\beta \in (0,2]$, $s \ge 0$ and $\gamma >0 $. The function $\rho_t \colon [0,\infty) \to \R$, defined as $\rho_t(x) = x^{\frac{\beta}{2}(s+1)} \exp(-\gamma x^\frac{\beta}{2}t)$ for all $t \ge 0$ satisfies
	\begin{equation*}
		x^n \abs{\partial_x^{n} \rho_t(x)} \lesssim \rho_t(\frac{x}{2})
	\end{equation*}
	for all $n \in \N_0$.
\end{lemma}

\begin{proof}
	First, we consider the function $s \colon [0,\infty)^2 \to \R$, $s(t,x) = \exp(-\gamma xt)$. Clearly,
	\begin{equation*}
		(-1)^j \partial_x^j s(t,x) \ge 0
	\end{equation*}
	for all $j \in \N_0$ and $x \ge 0$. Let us consider the function $\tilde{s}(t,x) = s(t,x^\frac{\beta}{2})$. We claim that $(-1)^j \partial_x^j \tilde{s}(t,x) \ge 0 $ holds, too. As for $\beta = 2$ we have $\tilde{s} = s$ and thus need to only consider the case $0<\beta <2$. One can show that the $j$-th derivative of $\tilde{s}$ is the finite sum of terms of the form 
	\begin{equation*}
		c(\beta) (-1)^k x^{(\frac{\beta}{2}-1)\gamma-m} [\partial_x^l s](t,x^\frac{\beta}{2})
	\end{equation*}
	for a constant $c(\beta)>0$ and $l,k,m,\gamma \in \N_0$ with $l \le j$ and $k+j = l \; \mathrm{ mod } \; 2$ by a simple induction argument. The claim follows. 	
	
	Next, by Taylor's theorem we deduce
	\begin{equation*}
		\tilde{s}\left(t, \frac{x}{2}\right)=\sum_{j=0}^{n} \frac{\partial_{x}^{j} s(t, x)}{j !}\left(-\frac{x}{2}\right)^{j}+\frac{\partial_{\mu}^{n+1} s(t, \eta)}{(n+1) !}\left(-\frac{x}{2}\right)^{n+1} \ge \frac{\partial_{x}^{j} s(t, x)}{j !}\left(-\frac{x}{2}\right)^{j}
	\end{equation*}
	for all $0 \le j \le n$ as every summand is nonnegative. To conclude, we use Leibniz's formula to write
	\begin{equation*}
		x^n \partial_x^n \rho_t(x) = \sum_{j = 0}^n \binom{n}{j} \left(x^j \partial_x^j \tilde{s}(t,x) \right) \left(x^{n-j} \partial_x^{n-j}x^{\frac{\beta}{2}(s+1)} \right)
	\end{equation*}
	and from here the claim follows by a straightforward estimation. 
\end{proof}

\begin{prop} \label{pro8p:locmik}
	Let $\beta \in (0,2]$, $s \ge 0$ and $\gamma >0$. Let $0<l_1<u_1$, $0<l_2<u_2$ and let $\psi \in C_c^\infty(\R^{2n})$ be any cutoff function such that $\psi(z) = 1$ for all $z = (z_1,z_2) \in \R^{2n}$ with $\abs{z_1}^\beta \in [l_1,u_1]$ and $\abs{z_2}^\beta \in [l_2,u_2]$, $\psi = 0$ if $\abs{z_1}^\beta \in [ \frac{1}{2} l_1 ,2u_1]^c$ or $\abs{z_2}^\beta \in [\frac{1}{2}l_2,2u_2]^c$ and $\abs{z}^\alpha \abs{\partial^\alpha \psi(z)} \le C=C(\alpha,\frac{u_1}{l_1},\frac{u_2}{l_2})$ for all multi-indices $\alpha \in \N_0^{2n}$ and any $z \in \R^{2n}$. 	
	
	\noindent Let $m_2(t,z) = \abs{z}^{2\frac{\beta}{2}(s+1)}\exp(-\gamma \abs{z}^{2\frac{\beta}{2}}t)\psi(z) = \rho_t(\abs{z}^2)\psi(z)$ with $\rho_t$ as in Lemma \ref{lem:basismult}. Then $m_2(t) \in C^\infty(\R^{2n} \setminus \{0 \})$ and for every multi-index $\eta \in \N_0^{2n}$ the partial derivative $\partial_z^\eta m_2$ of order $\abs{\eta}$ is the sum of finitely many terms of the form 
	\begin{equation*} \label{eq:diffmik}
		c(\eta) \rho_t^{(j)}(\abs{z}^2) \partial^\alpha \psi \prod_{i = 1}^{2n} z_i^{\gamma_i}, 
	\end{equation*} 
	where $\alpha \in \N^{2n}_0$ with $\abs{\alpha} \le \abs{\eta}$, $\gamma_i \ge 0$, $2j - \sum_{i=1}^{2n} \gamma_i \ge 0$, $j \le \abs{\eta}$, $\sum_{i = 1}^{2n}\gamma_i = 2j-\abs{\eta}+\abs{\alpha}$ and $c(\eta)>0$. In particular $m_2(t)$ is an $L^q(\R^{2n})$-Fourier multiplier with norm bounded by
	\begin{align*}
		\norm{m_2(t)}	\lesssim  1 \land \left(u_1+u_2 \right).
	\end{align*}
\end{prop}

\begin{proof}
	We are going to prove the claim on the structure of the derivatives of $m_2$ by induction over the length of $\eta$. If $\abs{\eta} = 0$ we see that $\partial^\eta m_2 = m_2$ is of the desired form. Let us assume that the claim is true for all $\eta \in \N^{2n}_0$ of length $\abs{\eta} = b \in \N_0$. Let $\tilde{\eta} \in \N^{2n}_0$ with $\abs{\tilde{\eta}} = b+1$, i.e. there exists $\eta \in \N^{2n}_0$ such that $ \partial^{\tilde{\eta}} m_2 = \partial_{z_l} \partial^\eta m_2$ for some $l \in \{ 1, \dots, 2n \}$. By induction hypothesis $\partial^\eta m_2$ is the finite sum of first order partial derivatives with respect to  $z_l$ of terms of the form described in equation \eqref{eq:diffmik}. We therefore need to only study one of these terms. If $\gamma_l = 0$ we deduce
	\begin{equation*}
		 \partial_{z_l}\left[ \rho_t^{(j)}(\abs{z}^2) \partial^{\alpha} \psi(z) \prod_{i = 1}^{2n} z^{\gamma_i}  \right]  = \rho_t^{(j+1)}(\abs{z}^2) 2 z_l \partial^{\alpha} \psi(z) \prod_{i = 1}^{2n} z^{\gamma_i}  + \rho_t^{(j)}(\abs{z}^2) \partial_{z_l}\partial^{\alpha} \psi(z) \prod_{i = 1}^{2n} z^{\gamma_i}. 
	\end{equation*}
	The first term is of the desired form if we choose new parameters $\tilde{\alpha}= \alpha$, $\tilde{\gamma}_i = \gamma_i$ for $i \neq l$ and $\tilde{\gamma}_l = 1$. Moreover, we have
	\begin{equation*}
		0 \le \sum_{i = 1}^{2n}\tilde{\gamma}_i = \sum_{i = 1}^{2n}{\gamma}_i +1 = 2j- \abs{\eta} + \abs{\alpha} +1 = 2(j+1)- \abs{\tilde{\eta}}+ \abs{\tilde{\alpha}}.
	\end{equation*}
	Turning to the second term, by choosing $\tilde{\alpha}_i = \alpha_i$ for $i \neq j$ and $\tilde{\alpha}_l= \alpha_l+1$ and keeping the other parameters the same we deduce $\abs{\tilde{\alpha}} = \abs{\alpha}+1 \le \abs{\eta}+1 = \abs{\tilde{\eta}}$ as well as 
	\begin{equation*}
		\sum_{i = 1}^{2n}{\gamma}_i = 2j- \abs{\eta}+ \abs{\alpha}= 2j- \abs{\tilde{\eta}}+ \abs{\tilde{\alpha}}.
	\end{equation*}
	If $\gamma_l >0$ it is 
	\begin{align*}
		 \partial_{z_l}\left[ \rho_t^{(j)}(\abs{z}^2) \partial^{\alpha} \psi(z) \prod_{i = 1}^{2n} z^{\gamma_i}  \right]  &= \rho_t^{(j+1)}(\abs{z}^2) 2 z_l \partial^{\alpha} \psi(z) \prod_{i = 1}^{2n} z^{\gamma_i}  + \rho_t^{(j)}(\abs{z}^2) \partial_{z_l}\partial^{\alpha} \psi(z) \prod_{i = 1}^{2n} z^{\gamma_i}\\
		 &\hphantom{=}+ \rho_t^{(j)}(\abs{z}^2)  \partial^{\alpha} \psi(z) \gamma_l z^{\gamma_l-1} \ \prod_{\overset{i = 1}{i \neq l}}^{2n} z^{\gamma_i}.
	\end{align*}
	The first two terms have already been considered in the case $\gamma_l = 0$. It remains to study the last term. Setting $\tilde{\gamma}_i= \gamma_i$ for $i \neq l $ and $\tilde{\gamma}_l = \gamma_l -1$ we conclude
	\begin{equation*}
		0 \le \sum_{i = 1}^{2n}{\tilde{\gamma}}_i = \sum_{i = 1}^{2n}{{\gamma}}_i -1 = 2j-\abs{\eta}+ \abs{\alpha}-1 = 2j-\abs{\tilde{\eta}}+ \abs{{\alpha}},
	\end{equation*}
	whence also the last term is of the desired form. 
	
	To deduce the bound on the Mikhlin norm we only need to consider a term $T_t(z)$ of the form given in equation \eqref{eq:diffmik} corresponding to an index of length $\abs{\eta} \le n+1$. We deduce
	\begin{align*}
		\abs{z}^{\abs{\eta}} \abs{T_t(z)} &\le c(\eta) \abs{z}^{\abs{\eta}} \abs{\rho_t^{(j)}(\abs{z}^2)}\abs{z}^{\sum_{i = 1}^{2n} \gamma_i} \abs{\partial^\alpha \psi(z)}  \\
		&= c(\eta) \abs{z}^{2j}\abs{\rho_t^{(j)}(\abs{z}^2)}\abs{z}^{\abs{\eta}-2j-\abs{\alpha}+\sum_{i = 1}^{2n} \gamma_i} \abs{z}^{\alpha} \abs{\partial^\alpha \psi(z)} \\
		&\lesssim \rho_t(\frac{1}{2}\abs{z}^2) \abs{z}^{\alpha} \abs{\partial^\alpha \psi(z)}  \\
		& \lesssim \rho_t(\frac{1}{2}\abs{z}^2)  \mathds{1}_{[\frac{1}{2}l_1,2u_1]}(\abs{z_1}^2)\mathds{1}_{[\frac{1}{2}l_2,2u_2]}(\abs{z_2}^2),
	\end{align*}
	where we have used that $\abs{z}^{2j}\abs{\rho_t^{(j)}(\abs{z}^2)}$ can be bounded by $\rho_t(\abs{z}^2/2)$ with a constant independent of $t$ by Lemma \ref{lem:basismult}. We have 
	\begin{align*}
		&(x_1+x_2)^{(s+1)} \exp(-(x_1+x_2)t)  \mathds{1}_{[\frac{1}{2}l_1,2u_1]}(x_1)\mathds{1}_{[\frac{1}{2}l_2,2u_2]}(x_2) \\ 
		&\lesssim x_1^{(s+1)} \exp(-x_1t) \mathds{1}_{[\frac{1}{2}l_1,2u_1]}(x_1)+ x_2^{(s+1)} \exp(-x_2t) \mathds{1}_{[\frac{1}{2}l_2,2u_2]}(x_2) \\
		&\lesssim 1 \land u_1^{s+1} + 1 \land u_2^{s+1} \lesssim 1 \land (u_1^{s+1}+u_2^{s+1}). \qedhere
	\end{align*}
\end{proof}

Let $o = 2^{\frac{2\beta}{\beta+2}}$ and set $l_1 = o^{(j-1)(\beta+1)}t^\beta$, $u_1 = o^{(j+1)(\beta+1)}t^\beta$, $l_2 = o^{j-1}t$ and $u_2 = o^{j+1}t$. First, we observe that $\frac{u_1}{l_1} =\frac{u_2}{l_2} = o^{2(\beta+1)}$ is independent of $t$ and $j$. Thus, the constant in the bound on the derivatives for the cutoff function in Proposition \ref{pro8p:locmik} does not depend on $j$ or $t$. This is in particular the case for the cutoff function chosen in the proof of Proposition \ref{prop:helpprob}.  By Proposition \ref{pro8p:locmik} the function $m_1$ defines an $L^q(\R^{2n})$-Fourier multiplier with norm bounded by 
\begin{equation*}
	\norm{m_1(t,k,\xi)} \lesssim t^{-(s+1)} \land \left(\left( \frac{u_1}{t} \right)^{s+1}+\left( \frac{u_2}{t} \right)^{s+1}\right).
\end{equation*}
Furthermore, $\left( \frac{u_2}{t} \right)^{s+1} \lesssim o^{j(s+1)}$ for all $t \ge 0$ and 
\begin{equation*}
	\left( \frac{u_1}{t} \right)^{s+1} \lesssim o^{j(s+1)} (o^jt)^{\beta(s+1)} \lesssim o^{j(s+1)} .
\end{equation*}
for $t \le o^{-j}$. If $t > o^{-j}$, then $t^{-1} < o^j$, whence $t^{-(s+1)}< o^{j(s+1)}$ and consequently
\begin{equation*}
	\norm{m_1(t,k,\xi)} \lesssim t^{-(s+1)} \land o^{j(s+1)}.
\end{equation*}

We have proven the following corollary.

\begin{coro} \label{cor:m0mult}
	The function $m_0^{(j)}$ defines a bounded $L^q(\R^{2n})$-Fourier multiplier.  
\end{coro}

\begin{lemma} \label{lem:itsez}
	There exists a constant $c>0$ such that
	\begin{equation*}
		c\abs{\xi}^\beta+c\abs{k}^\beta \le  \int_0^1 \abs{\xi+rk}^\beta \dx r
	\end{equation*}
	for all $k,\xi \in \R^n$. The function
	\begin{equation*}
		h(t,k,\xi) = \exp(-t \int_0^1 \abs{\xi-rk}^\beta \dx r+\frac{c}{2}t\abs{\xi}^\beta+\frac{c}{2}t\abs{k}^\beta)
	\end{equation*}
	defines a bounded $L^q$-multiplier for all $t \ge 0$ with constant $C= C(\beta,n)$ independent of $t$.
\end{lemma}

\begin{proof}
	The first claim has been proven in \cite[Lemma 2.1]{niebel_kinetic_nodate}. To prove the latter we note that by a scaling argument we can reduce the argument to the case $t = 1$. By an induction argument similar to that in the proof of Proposition \ref{pro8p:locmik} one can show that the derivative $\partial_k^\eta \partial_\xi^\rho h(1,k,\xi)$ with multi-indices $\eta,\rho \in \N^n_0$ is the sum of finitely many terms of the form 
	\begin{equation*}
		c_0 k^a \xi^b \abs{\xi}^{c\beta-\alpha} \abs{k}^{d\beta-\gamma} \prod_{j = 1}^\lambda \left( \int_0^1 r^{\tau_j} (\xi-rk)^{\delta^{(j)}} \abs{\xi-rk}^{w_j \beta-\theta_j} \dx r \right)^{\mu_j} m(1,k,\xi),
	\end{equation*}
	where $a,b,\delta^{(1)},\dots, \delta^{(\lambda)} \in \N_0^n$, $c_0 \in \R$, $\lambda,\alpha,d,\gamma,\tau_1,\dots,\tau_\lambda,w_1,\dots,w_\lambda,\theta_1,\dots,\theta_\lambda, \mu_1 \dots, \mu_\lambda \in \N_0$ with 
	\begin{equation*}
		(c+d)\beta+\abs{a}+ \abs{b}-\alpha-\gamma + \sum_{j = 1}^\lambda \mu_j (w_j \beta + \abs{\delta^{(j)}}-\theta_j) + \abs{\eta}+ \abs{\rho} \ge 0.
	\end{equation*}
	We omit the details. From here it is clear that $m(1,k,\xi)$ satisfies the condition in Mikhlin's theorem as $\sup_{x>0} x^r \exp(-x) < \infty$ for all $r \ge 0$.
\end{proof}

We collect some useful embeddings for anisotropic Besov spaces. 

\begin{lemma} \label{lem:besovLq}
	Let $\alpha \in (0,\infty)^{2n}$, $p,q \in (1,\infty)$, $s>0$ and $0<\epsilon<s$. Then, 
	\begin{equation*}
		B_{qp}^{s,\alpha}(\R^{2n})\hookrightarrow B^{s,\alpha}_{q\infty}(\R^{2n})  \hookrightarrow  	B_{q1}^{s-\epsilon,\alpha}(\R^{2n}) \hookrightarrow 	B_{q1}^{0,\alpha}(\R^{2n}) \hookrightarrow L^q(\R^{2n}).
	\end{equation*}
\end{lemma}

\begin{proof}
	The first three embeddings follow from \cite[Theorem 5.28]{triebel_theory_2006} combined with \cite[Proposition 2.2]{sawano_theory_2018} and a straightforward calculation. The third embedding is clear. And the last embedding follows from the representation
	\begin{equation*}
			g = \F^{-1}(\varphi_0 \hat{g})+\sum_{j = 1}^\infty \F^{-1}(\varphi_j \hat{g})
	\end{equation*}
	 as in the proof of \cite[Proposition 2.1]{sawano_theory_2018}.
\end{proof}

\begin{lemma} \label{lem:contembed}
	Let $\alpha \in (0,\infty)^{2n}$ be any anisotropy, $s \in \R$ and $p,q \in (1,\infty)$ such that $s>\frac{2n}{q}$, then
	\begin{equation*}
		B^{s,\alpha}_{qp}(\R^{2n})  \hookrightarrow B_{\infty,1}^{0,\alpha}(\R^{2n}) \hookrightarrow BUC(\R^{2n})
	\end{equation*}
	and consequently
	\begin{equation*}
			B^{s,\alpha}_{qp}(\R^{2n}) \hookrightarrow C_0(\R^{2n}).
	\end{equation*}
\end{lemma}

\begin{proof}
	A proof for the first embedding follows from \cite[Theorem 5.28]{triebel_theory_2006} and \cite[Proof of Proposition 2.5]{sawano_theory_2018}. The second embedding follows from the equality
	\begin{equation*}
		g = \F^{-1}(\varphi_0 \hat{g})+\sum_{j = 1}^\infty \F^{-1}(\varphi_j \hat{g})
	\end{equation*}
	 as in the proof of \cite[Proposition 2.4]{sawano_theory_2018}. Finally, as $C_c^\infty(\R^{2n})$ is dense in $B^{s,\alpha}_{qp}(\R^{2n})$ the last embedding follows by an approximation argument. 
\end{proof}

\begin{lemma} \label{lem:laplacianbesov}
	We consider the anisotropy $\alpha = (\frac{3}{2},\dots,\frac{3}{2},\frac{1}{2},\dots,\frac{1}{2})  \in (0,\infty)^{2n}$. For all $p,q \in (1,\infty)$ and any $s \in \R$ we have that
	\begin{equation*}
		(-\Delta_v)^{\frac{\gamma}{2}} \colon B_{qp}^{s,\alpha}(\R^{2n}) \to B_{qp}^{s-\gamma/2,\alpha}(\R^{2n})
	\end{equation*}
	is continuous.
\end{lemma}

\begin{proof}
	The calculation in the isotropic setting can be found in \cite[Theorem 2.12]{sawano_theory_2018}. These calculations continue to hold in the anisotropic case. However, as we have $2^{(j-1)/2} \le \abs{\xi} \le 2^{(j+1)/2}$ in the support of a function $\varphi_j^\alpha$ it follows that the regularity index decreases only by $\gamma/2$. 
\end{proof}

\begin{coro} \label{cor:gradvembed}
	Let $p,q \in (1,\infty)$ and $\mu \in (\frac{1}{p},1]$, then
	\begin{equation*}
		{^{\mathrm{kin}}B}_{qp}^{\mu-1/p,2}(\R^{2n}) \hookrightarrow {_v C}^1_0(\R^{2n})
	\end{equation*}
	if $\mu-1/p>1/2+2n/q$.
\end{coro}

\begin{proof}
	By Lemma \ref{lem:contembed} and Lemma \ref{lem:laplacianbesov} we deduce $(-\Delta_v)^{\frac{1}{2}} u \in {^{\mathrm{kin}}B}_{qp}^{\mu-1/p-1/2,2}(\R^{2n}) \hookrightarrow C_0(\R^{2n})$ if $\mu-1/p-1/2-2n/q>0$. 
\end{proof}

\begin{lemma} \label{lem:C0gamabuc}
	Let $u\in C_0(\R^{2n})$ and $T \in (0,\infty)$. Then, $\Gamma u \colon [0,T] \times \R^{2n} \to \R$ is uniformly continuous. 
\end{lemma}
\begin{proof}
	Let $\epsilon>0$, then there exists $\delta_0>0$ and $R>0$ such that 
	\begin{equation*}
		\abs{u(x,v)-u(y,w)} < \epsilon
	\end{equation*}
	for all $x,y,v,w \in \R^n$ with $\abs{x-y}+\abs{v-w} < \delta_0$ and $\abs{u(x,v)}< \epsilon/2$ for all $x,v \in \R^n$ with $\abs{x}+\abs{v} \ge R$. We choose $\delta < \min \{ R, (1+2R+T)^{-1}\delta_0 \}$, then for $(t,x,v) \in [0,T] \times \R^n \times B_{2R}(0)$ and $(s,y,w) \in [0,T] \times \R^{2n}$ with $\abs{t-s}+\abs{x-y}+\abs{v-w} < \delta$ we deduce
	\begin{equation*}
		\abs{x+tv-y-sw} \le \abs{x-y}+\abs{t-s}\abs{v}+\abs{s}\abs{v-w} < \delta_0,
	\end{equation*}
	whence $\abs{u(x+tv,v)-u(y+sw,w)} < \epsilon$. If $(t,x,v) \in [0,T] \times \R^n \times B_{2R}^c(0)$ and $(s,y,w) \in [0,T] \times \R^{2n}$ with $\abs{t-s}+\abs{x-y}+\abs{v-w} < \delta$ we deduce $w \in B_R(0)^c$. Hence, we have
	\begin{equation*}
		\abs{u(x+tv,v)-u(y+sw,w)} \le \abs{u(x+tv,v)}+\abs{u(y+sw,w)} < \epsilon
	\end{equation*}
	as $\abs{x+tv}+\abs{v} \ge R$ and $\abs{y+sw} +\abs{w} \ge R$.
\end{proof}

\bibliographystyle{amsplain}

\providecommand{\bysame}{\leavevmode\hbox to3em{\hrulefill}\thinspace}
\providecommand{\MR}{\relax\ifhmode\unskip\space\fi MR }
\providecommand{\MRhref}[2]{%
  \href{http://www.ams.org/mathscinet-getitem?mr=#1}{#2}
}
\providecommand{\href}[2]{#2}

\end{document}